\documentclass[12pt,reqno]{amsart}
\usepackage{amsmath,amsfonts,amssymb}
\allowdisplaybreaks
\usepackage[utf8]{inputenc}
\usepackage[english]{babel}
\usepackage[T1]{fontenc}
\usepackage{lmodern}
\usepackage{mathtools}
\usepackage{scrextend}
\usepackage{enumitem}
\usepackage{graphicx}
\usepackage{stmaryrd}
\usepackage{xcolor}
\usepackage{pifont} 
\usepackage{array}
\usepackage{tabularx}
\usepackage{ltablex}
\usepackage{tikz}

\newcommand{\cuthere}{%
	\noindent
	\raisebox{-2.8pt}[0pt][0.95\baselineskip]{\ding{34}}
	\unskip{\tiny\dotfill}
}

\usepackage[initials]{amsrefs}
\date{\today}

\usepackage[text={460pt,600pt},centering]{geometry}
\usepackage[colorlinks=true, pdfstartview=FitV, linkcolor=blue, citecolor=red, urlcolor=blue,pagebackref=false]{hyperref}
\numberwithin{equation}{section}

\usepackage{cleveref}

\crefformat{equation}{(#2#1#3)}
\crefrangeformat{equation}{(#3#1#4)--(#5#2#6)}
\crefmultiformat{equation}{(#2#1#3)}%
{ and~(#2#1#3)}{, (#2#1#3)}{ and~(#2#1#3)}

\usepackage{amsthm}

\theoremstyle{plain}
\newtheorem{thm}{Theorem}[section]
\newtheorem{corollary}[thm]{Corollary}
\newtheorem{proposition}[thm]{Proposition}
\newtheorem{lemma}[thm]{Lemma}
\newtheorem{thm*}{Theorem}
\newtheorem{proposition*}[thm*]{Proposition}
\newtheorem{lemma*}[thm*]{Lemma}
\newtheorem{question*}[thm*]{Question}
\newtheorem{corollary*}[thm*]{Corollary}

\theoremstyle{definition}
\newtheorem{remark}[thm]{Remark}

\newtheorem{question}[thm]{Question}
\newtheorem{remark*}[thm*]{Remark}

\usepackage{amsopn}
\usepackage[foot]{amsaddr}

\newcommand{\norm}[1]{\left\lVert#1\right\rVert}
\newcommand{\real}{\mathbb{R}}
\newcommand{\dd}{\mathrm{d}}
\newcommand{\ex}[1]{\operatorname{\mathbb{E}}\left[#1\right]}

\newcommand{\pr}[1]{\operatorname{\mathbb{P}}\left(#1\right)}

\renewcommand{\epsilon}{\varepsilon}
\renewcommand{\phi}{\varphi}

\newcommand{\Mf}{\mathcal{M}_f}

\newcommand{\T}{\mathbb{T}}
\newcommand{\TT}{\mathbb{T}^*}

\newcommand{\lawd}{\overset{(d)}{=}}
\newcommand{\ind}[1]{\mathbf{1}_{\left\{#1\right\}}}
\newcommand{\rdtree}{\mathcal{T}}
\renewcommand{\H}{\mathfrak{h}}

\newcommand{\expp}[1]{\exp\left\lbrace#1\right\rbrace}
\newcommand{\exstar}[2]{\operatorname{\mathbb{E}}^{\psi,\ast}_{#1}\left[#2\right]}

\newcommand{\rddtree}{\tau}
\renewcommand{\root}{\emptyset}
\newcommand{\supp}[1]{\mathrm{supp}(#1)}
\renewcommand{\root}{\emptyset}

\newcommand{\g}{\mathfrak{g}}

\newcommand{\e}{\mathrm{e}}

\renewcommand{\for}[1]{\operatorname{\mathbb{P}}_{\! #1}}
\newcommand{\pibar}{\bar{\pi}}
\newcommand{\ndelta}{\operatorname{\mathbf{N}^{\psi_\delta}}}
\newcommand{\rfor}{\operatorname{\mathbb{Q}^\psi_\delta}}

\renewcommand{\leqslant}{\leq}
\renewcommand{\geqslant}{\geq}
\renewcommand{\TT}{\mathbb{T}^*}
\renewcommand{\for}[1]{\operatorname{\mathbb{P}}^\psi_{\! #1}}
\newcommand{\explo}{\operatorname{\mathbb{P}}^\psi}
\newcommand{\fordelta}[1]{\mathbb{P}^{\psi_\delta}_{\! #1}}
\newcommand{\n}{\operatorname{\mathbf{N}^\psi}}
\newcommand{\rfordelta}{\operatorname{\mathbb{Q}^{\psi_\delta}_\delta}}
\newcommand{\tilted}[1]{\operatorname{\mathbf{P}^\psi_{#1}}}
\newcommand{\atom}[1]{\operatorname{\mathbf{P}^{\psi,\mathrm{a}}_{#1}}}
\newcommand{\tiltedh}[1]{\operatorname{\mathbf{P}^\psi_{#1,\mathnormal{h}}}}
\newcommand{\abov}[1]{\operatorname{\theta}_{#1}}
\newcommand{\off}{\xi^\delta}
\newcommand{\total}{W^\delta}
\newcommand{\init}{Z_0^\delta}

\newcommand{\pruned}[1]{#1^{-}}

\newcommand{\dinf}{\delta-\epsilon}
\newcommand{\dsup}{\delta+\epsilon}
\newcommand{\ndinf}{\operatorname{\mathbf{N}^{\psi_{\delta--\epsilon}}}}
\newcommand{\rfordinf}{\operatorname{\mathbb{Q}^\psi_{\delta--\epsilon}}}

\newcommand{\D}{\mathcal{D}}
\newcommand{\DD}{\mathcal{D}_0}
\newcommand{\dS}{d_{S}}
\newcommand{\dK}{d_{0}}
\newcommand{\w}{w}
\newcommand{\wplus}{w_+}
\newcommand{\unique}[1]{E_#1}
\newcommand{\B}{\mathcal{B}_+}
\renewcommand{\supp}[1]{\mathrm{supp}(#1)}
\newcommand{\dbl}{d_{\mathrm{BL}}}

\begin{document}
	\begin{abstract}
		We study the maximal degree of (sub)critical Lévy trees which arise as the scaling limits of Bienaymé-Galton-Watson trees. We determine the genealogical structure of large nodes and establish a Poissonian decomposition of the tree along those nodes. 
		Furthermore, we make sense of the distribution of the Lévy tree conditioned to have a fixed maximal degree. In the case where the Lévy measure is diffuse, we show that the maximal degree is realized by a unique node whose height is exponentially distributed and we also prove that the conditioned Lévy tree can be obtained by grafting a  Lévy forest on an independent size-biased Lévy tree with a degree constraint at a uniformly chosen leaf. Finally, we show that the Lévy tree conditioned on having large maximal degree converges locally to an immortal tree (which is the continuous analogue of the Kesten tree) in the critical case and to a condensation tree in the subcritical case. Our results are formulated in terms of the exploration process which allows to drop the Grey condition.
	\end{abstract}
	\title[Conditioning  Lévy trees by their maximal degree]{Conditioning (sub)critical Lévy trees by their maximal degree: decomposition and local limit}
	\date{\today}
	\author{Romain Abraham$^1$}
	\address{$^1$
		Institut Denis Poisson,
		Universit\'{e} d'Orl\'{e}ans,
		Universit\'e de Tours,
		CNRS,
		France}
	\email{romain.abraham@univ-orleans.fr}
	
	\author{Jean-Fran\c{c}ois Delmas$^2$}
	\address{$^2$
		CERMICS, Ecole des Ponts, France}
	\email{delmas@cermics.enpc.fr}
	
	\author{Michel Nassif$^3$}
	\address{$^3$
		MAP5, Université Paris Cité, CNRS, France}
	\email{michel.nassif@u-paris.fr}
\maketitle
\section{Introduction and main results}
Lévy trees are random metric spaces that encode the genealogical structure of continuous-state branching processes (CB processes for short). As such, they arise as the scaling limits of Bienaymé-Galton-Watson trees. Lévy trees were introduced by Le Gall and Le Jan \cite{le1998branching} and Duquesne and Le Gall \cite{duquesne2002random} in order to generalize Aldous' Brownian tree \cite{aldous1991continuum}. They also appear as scaling limits of various models of trees and graphs, see e.g.~Haas and Miermont \cite{haas2012scaling}, and are naturally related to fragmentation processes, see Miermont \cites{miermont2003fragmentations, miermont2005fragmentations}, Haas and Miermont \cite{haas2004genealogy}, Abraham and Delmas \cite{abraham2008fragmentation}.

In the present paper, we study the maximal degree of a general Lévy tree. More precisely, we first establish a Poissonian decomposition of the Lévy tree along large nodes. Then, we make sense of the distribution of the Lévy tree conditioned to have a fixed maximal degree. In the case where the Lévy measure is diffuse, we show that the maximal degree is realized by a unique node, and we describe how to reconstruct the tree by grafting a Lévy forest on an independent size-biased Lévy tree (with a restriction on the maximal degree) at a uniform leaf. Finally, we investigate the asymptotic behavior of the Lévy tree conditioned to have large maximal degree.

These questions arise naturally in the study of random trees and have been thoroughly investigated in the case of Bienaymé-Galton-Watson trees. The first results in this direction were obtained by Jonsson and Stefánsson \cite{jonsson2011condensation} who showed that a condensation phenomenon appears for a certain class of subcritical Bienaymé-Galton-Watson trees conditioned to have a large size, in the sense that with high probability there exists a unique node with degree proportional to the size. Furthermore, the tree converges locally to a condensation tree consisting of a finite spine with random geometric length onto which independent and identically distributed Bienaymé-Galton-Watson trees are grafted. This was later generalized by Janson \cite{janson2012simply}, with further results by Kortchemski \cite{kortchemski2015condensation}, Abraham and Delmas \cite{abraham2014condensation}, Stufler \cite{stufler2020maximal}. On the other hand, He \cite{he2017conditioning} shows that Bienaymé-Galton-Watson trees conditioned on having large maximal degree converge locally to Kesten's tree (which consists of an infinite spine onto which independent and identically distributed Bienaymé-Galton-Watson trees are grafted) in the critical case and to a condensation tree in the subcritical case.

In the continuum setting, Bertoin \cite{bertoin2011maximal} determined the distribution of the maximal degree of a stable Lévy tree (his result is formulated in terms of Lévy processes). Using the formalism of CB processes, He and Li \cite{he2016distribution} treated the case of a general branching mechanism (in fact their result is more general as they considered CB processes with immigration). In \cite{he2014maximal}, they also studied the local limit of a CB process conditioned to have large maximal degree (i.e.~large maximal jump). In the critical case, they showed that it converges locally to a CB process with immigration. Later , He \cite{he2022local} extended the local convergence result to the whole genealogy: more precisely, he showed that a critical Lévy tree conditioned on having large maximal degree converges locally to an immortal tree (which is the continuous counterpart of Kesten's tree, consisting of an infinite spine onto which trees are grafted according to a Poisson point process). We improve these results by considering the density version of the conditioning instead of the tail version: more explicitly, we study the asymptotic behavior of critical Lévy trees conditioned to have maximal degree equal to (and not larger than) a given value. Density versions are finer than their tail counterparts and are usually more difficult to prove.

The existing litterature in the subcritical case is less developped. He and Li \cite{he2014maximal} showed that a subcritical CB process conditioned to have large maximal degree converges locally to a CB process with immigration which is killed (i.e.~sent to infinity) at an independent exponential time, thus underlining a condensation phenomenon. We improve this result in several directions. Again we consider the density version of the conditioning instead of the tail version. We also extend the convergence result to the whole genealogical structure instead of the population size at a given time: this gives more information and, as an example, allows us to see that only one large node emerges. Finally, we are also able to describe precisely what happens above the condensation node.

For the sake of clarity, we shall formulate our results in terms of Lévy trees in the introduction. This requires an additional assumption on the branching mechanism, namely the Grey condition (see below), in order to have a nice topology on the set of trees. Indeed, this condition ensures that the Lévy tree is a \emph{compact} real tree. However, it is superfluous and will be dropped in the rest of the paper where we will deal with the exploration process instead. Let us mention that a forthcoming work by Duquesne and Winkel \cite{duquesne2022mass} should allow us to use the formalism of real trees even for a general branching mechanism not necessarily satisfying the Grey condition.

Before stating our main results, we need to recall some definitions and to set notations.
\subsection{Real trees}
We recall the formalism of real trees, see \cite{evans2007probability}. A quadruple $(T,d, \root, \mu)$ is called a real tree if $(T,d)$ is a metric space equipped with a distinguished vertex $\root \in T$ called the root and a nonnegative finite measure $\mu$ on $T$ and if the following two properties hold for every $x,y \in T$:
\begin{enumerate}[label=(\roman*),leftmargin=*]
	\item (Unique geodesics). There exists a unique isometric map $f_{x,y} \colon [0,d(x,y)] \to T$ such that $f_{x,y}(0) = x$ and $f_{x,y}(d(x,y)) = y$.
	\item (Loop-free). If $\phi$ is a continuous injective map from $[0,1]$ into $T$ such that $\phi(0) = x$ and $\phi(1) = y$, then we have $\phi([0,1]) = f_{x,y}\left([0,d(x,y)]\right)$.
\end{enumerate}
For every vertex $x \in T$, we define its height by $H(x) = d(\root, x)$. The height of the tree is defined by $\H(T) = \sup_{x\in T}H(x)$. Note that if $(T,d)$ is compact, then $\H(T) < \infty$.

We will denote by $\T$ the set of (isometry classes of) \emph{compact} real trees. Let us mention that it can be equipped with the Gromov-Hausdorff-Prokhorov distance which makes it a Polish space, see e.g. \cite{addario2017scaling}.

We will also need the set $\TT$ of (isometry classes of) compact real trees that are \emph{marked}, i.e. equipped with a distinguished vertex in addition to the root $\root$. Again, $\TT$ can be made into a Polish space when equipped with a marked variant of the Gromov-Hausdorff-Prokhorov distance.

\subsection{Local convergence of real trees} We will make use of the notion of local convergence for \emph{locally compact} real trees which we now recall. For every $h >0$, define the restriction mapping on the set of (isometry classes of) real trees by:
\begin{equation*}
	r_h(T,d,\root,\mu) = (T^h,d_{|T^h \times T^h}, \root,\mu_{|T^h}) \quad \text{where } T^h = \{x \in T\colon\, H(x)\leqslant h\}.
\end{equation*}
In other words, $r_h(T)$ is the real tree obtained from $T$ by removing all nodes whose height is larger than $h$, equipped with the same metric and measure restricted to $T^h$. Recall that the Hopf-Rinow theorem implies that if $T$ is a locally compact real tree, the closed ball $r_h(T)$ is compact. We say that a sequence $T_n$ of locally compact trees converges locally to a locally compact tree $T$ if for every $h>0$, the sequence $r_h(T_n)$ converges for the Gromov-Hausdorff-Prokhorov distance to $r_h(T)$.

\subsection{Grafting procedure} 
Given a real tree $T \in \T$ and a finite or countable family $(( x_i,T_i), \, i\in I)$ of elements of $ T\times \T$, we denote by
\begin{equation*}
	T\circledast_{i\in I} (x_i,T_i)
\end{equation*}
the real tree obtained by grafting $T_i$ on $T$ at the node $x_i$. For a precise definition, we refer the reader to \cite[Section 2.4]{abraham2013forest}.

\subsection{Lévy trees}
Let $\psi$ be a branching mechanism given by:
\begin{equation}
	\psi(\lambda) = \alpha \lambda +\beta \lambda^2 + \int_{(0,\infty)} \left(\e^{-\lambda r}-1+\lambda r\right)\, \pi(\dd r),
\end{equation}
where $\alpha,\beta \geqslant 0$ and $\pi$ is a $\sigma$-finite measure on $(0,\infty)$ such that $\int_{(0,\infty)} (r\wedge r^2)\, \pi(\dd r)< \infty$. The branching mechanism $\psi$ is said to be critical (resp.~subcritical) if $\alpha = 0$ (resp.~$\alpha >0$). In what follows, we assume that $\pi \neq 0$ as otherwise all branching points of the Lévy tree will be binary. Whenever we are dealing with Lévy trees, we always assume that the Grey condition holds:
\begin{equation}\label{eq: grey intro}
	\int^\infty \frac{\dd \lambda}{\psi(\lambda)} < \infty,
\end{equation}
which is equivalent to the compactness of the Lévy tree. In the rest of the paper, this condition will be relaxed to:
\begin{equation}
	\beta >0 \quad \text{or} \quad \int_{(0,1)} r\, \pi(\dd r ) = \infty.
\end{equation}

We will consider a Lévy tree $\rdtree$ under its excursion measure which is denoted by $\n$. Here we briefly recall some results on Lévy trees but we refer the reader to Duquesne and Le Gall \cite{duquesne2002random,duquesne2005probabilistic} for a complete presentation on the subject. One can define a $\sigma$-finite measure $\n$ on the space $\T$, called the excursion measure of the Lévy tree, with the following properties.
\begin{enumerate}[label=(\roman*)]
	\item \textbf{Mass measure.} For $\n$-almost every $\rdtree$, the mass measure $\mu$ is supported on the set of leaves $\operatorname{Lf}(\rdtree)\coloneqq \{x \in \rdtree\colon\, \rdtree \setminus \{x\} \ \text{is connected}\}$. Furthermore, the total mass $\sigma \coloneqq \mu(\rdtree)$ satisfies:
	\begin{equation}
		\n\left[1-\e^{-\lambda \sigma}\right] = \psi^{-1}(\lambda).
	\end{equation}
	\item \textbf{Local times.} For $\n$-almost every $\rdtree$, there exists a process $(L^{a}, \, a \geqslant 0)$ with values in the space of finite measures on $\rdtree$ which is càdlàg for the weak topology and such that
	\begin{equation}
		\mu(\dd x) = \int_0^\infty \dd a\, L^{a}(\dd x).
	\end{equation}
	For every $a \geqslant 0$, the measure $L^{a}$ is supported on $\rdtree(a)\coloneqq\{x \in \rdtree \colon\, H(x) = a\}$ the set of nodes at height $a$. Furthermore, the real-valued process $(L^{a}_\sigma \coloneqq \langle L^{a},1\rangle, \, a \geqslant 0)$ is a $\psi$-CB process under its canonical measure.
	\item \textbf{Branching property.} For every $a\geqslant 0$, let $(\rdtree^{i}, \, i \in I_a)$ be the subtrees of $\rdtree$ originating from level $a$. Then, under $\n$ and conditionally on $r_a(\rdtree)\coloneqq \{x\in \rdtree\colon \, H(x)\leqslant a\}$, the measure $\sum_{i \in I_a} \delta_{\rdtree^{i}}$ is a Poisson point measure with intensity $L_\sigma^a \n$.
	\item \textbf{Branching points.} For $\n$-almost every $\rdtree$, the branching points of $\rdtree$ are either binary or of infinite degree. The set of binary branching points is empty if $\beta=0$ and is a countable dense subset of $\rdtree$ if $\beta>0$. The set 
	\begin{equation*}
		\operatorname{Br_\infty}(\rdtree) \coloneqq\{x \in \rdtree \colon\, \rdtree\setminus\{x\}\text{ has infinitely many connected components}\}
	\end{equation*} 
	of infinite branching points is nonempty with $\n$-positive measure if and only if $\pi \neq 0$. If $\langle \pi,1\rangle = \infty$, the set $\operatorname{Br_\infty}(\rdtree)$ is countable and dense in $\rdtree$ for $\n$-almost every $\rdtree$. Furthermore, the set $\{H(x), \, x \in \operatorname{Br_\infty}(\rdtree) \}$ coincides with the set of discontinuity times of the mapping $a\mapsto L^{a}$. For every such discontinuity time $a$, there is a unique $x_a \in \operatorname{Br_\infty}(\rdtree)\cap \rdtree(a)$ and $\Delta_a >0$ such that
	\begin{equation*}
		L^{a} = L^{a-} + \Delta_a \delta_{x_a}.
	\end{equation*}
	For convenience, we define $\Delta_a$ for every $a \geqslant 0$ by setting $\Delta_a = 0$ if $L^{a} = L^{a-}$. In particular, we have $L^{a}_\sigma = L^{a-}_{\sigma} + \Delta_a$, that is $\Delta_a$ is exactly the size of the jump of the associated CB process at time $a$. We will call $\Delta_a$ the degree (or the mass) of the node $x_a$. This is an abuse of language since a node $x_a \in  \operatorname{Br_\infty}(\rdtree)$ has infinite degree by definition. 
\end{enumerate}
\subsection{Main results}
We denote by $\Delta$ the maximal degree of the Lévy tree $\rdtree$ under $\n$:
\begin{equation}
	\Delta = \sup_{a \geqslant 0} \Delta_a.
\end{equation}
The first result of this paper gives the joint distribution of the maximal degree $\Delta$ and the total mass $\sigma$ under $\n$. The distribution of the maximal degree was already obtained by Bertoin \cite[Lemma 1]{bertoin2011maximal} in the stable Lévy case then by He and Li \cite{he2016distribution} in the general case.

For the sake of notational simplicity, if $\nu$ is a measure on $\real$ we will write $\nu(a,b)$ (resp.~$\nu[a,b)$) instead of $\nu((a,b))$ (resp.~$\nu([a,b))$). We will also write $\nu(a)$ for $\nu(\{a\})$.
Denote by $\pibar \colon \real_+ \to (0,\infty]$ the tail of the Lévy measure $\pi$:
\begin{equation}\label{eq: definition pibar}
	\pibar(\delta) = \pi(\delta,\infty), \quad \forall \delta \geqslant 0,
\end{equation}
and define the Laplace exponent $\psi_\delta$ for every $\delta>0$ by:
\begin{align}\label{eq: definition psi delta}
	\psi_\delta(\lambda) &= \left(\alpha + \int_{(\delta,\infty)} r \, \pi(\dd r)\right)\lambda + \beta \lambda^2 +\int_{(0,\delta]} \left(\e^{-\lambda r}-1+\lambda r\right)\, \pi(\dd r)\notag \\
	&= \psi(\lambda) + \int_{(\delta,\infty)} \left(1-\e^{-\lambda r}\right)\, \pi(\dd r).
\end{align}
Observe that, in terms of the associated Lévy process, this corresponds to removing all jumps with size larger than $\delta$. If the Lévy measure $\pi$ is finite, we also define:
\begin{equation}\label{eq: definition psi 0}
	\psi_0(\lambda) = \left(\alpha + \int_{(0,\infty)} r \, \pi(\dd r)\right)\lambda + \beta \lambda^2.
\end{equation}
\begin{proposition}\label{prop: distribution degree intro}
	For every $\delta >0$ and $\lambda \geqslant 0$, we have:
	\begin{equation}\label{eq: joint distribution intro}
		\n\left[1-\e^{-\lambda \sigma}\ind{\Delta\leqslant \delta}\right] = \psi_\delta^{-1}(\pibar(\delta)+\lambda).
	\end{equation}
	Furthermore, if the Lévy measure $\pi$ is finite, we have:
	\begin{equation}
		\n\left[1-\e^{-\lambda \sigma}\ind{\Delta= 0}\right] = \psi_0^{-1}(\langle \pi, 1\rangle+\lambda).
	\end{equation}
\end{proposition}
The proof is given in Section \ref{sect: distribution of degree}.
\begin{remark}
	Let us make a connection with He and Li \cite{he2014maximal}. Recall that under $\n$ the process $(L^{a}_\sigma, \, a \geqslant 0)$ is distributed as a $\psi$-CB process under its canonical measure and that the maximal degree $\Delta$ of the Lévy tree corresponds to the maximal jump of the associated CB process. In particular, taking $\lambda =0$ in \eqref{eq: joint distribution intro} gives the distribution of the maximal jump of a $\psi$-CB process, which was already obtained by He and Li, see \cite[Corollary 4.2]{he2014maximal}. In fact, their result is much more general (see \cite[Theorem 4.1]{he2014maximal}) since they consider a CB process with immigration and in this context, they compute the distribution of the \emph{local} maximal jump which in terms of the Lévy tree corresponds to the maximal degree up to a fixed level $h$. However, they do not give the \emph{joint} distribution of  $\Delta$ and $\sigma$, which in terms of the CB process corresponds to the total mass:
	\begin{equation*}
		\sigma = \int_0^\infty L^{a}_\sigma \, \dd a.
	\end{equation*}
\end{remark}

Next, we give a decomposition of the Lévy tree along the large nodes. More precisely, we identify the distribution of the pruned Lévy tree obtained by removing all nodes with degree larger than $\delta$ (and the subtrees above them). This is again a Lévy tree with branching mechanism $\psi_\delta$ under its excursion measure. Furthermore, one can recover the Lévy tree from the pruned one by grafting Lévy forests at uniformly chosen leaves in a Poissonian manner. Before stating the result, we first need to introduce some notations. For every $r>0$, denote by $\for{r}$ the distribution of the random real tree $\rdtree = \{\root\}\circledast_{i\in I} \rdtree_i$ obtained by gluing together at their root the atoms $(\rdtree_i, \, i \in I)$ of a $\T$-valued Poisson point measure with intensity $r\n[\dd \rdtree]$. This should be interpreted as the distribution of a Lévy forest with initial degree $r>0$. Furthermore, for every $\delta>0$ such that $\pibar(\delta) >0$, set:
\begin{equation*}
	\rfor(\dd \rdtree) = \frac{1}{\pibar(\delta)}\int_{(\delta,\infty)} \pi(\dd r) \for{r}(\dd \rdtree)
\end{equation*}
which is the distribution of a Lévy forest with random initial degree with distribution $\pi$ conditioned on being larger than $\delta$.

\begin{thm}\label{thm: degree decomposition intro}
	Let $\delta\geqslant 0$ such that $\pibar(\delta)<\infty$. Under $\ndelta$ and conditionally on $(\rdtree,\root,d,\mu)$, let $\left((x_i,\rdtree_{i}), \, i\in I\right)$ be the atoms of a Poisson point measure on $\rdtree\times \T$ with intensity $\pibar(\delta)\,\mu(\dd x) \rfor(\dd \widetilde{\rdtree})$. Then, under $\ndelta$, the random tree $\rdtree \circledast_{i\in I} (x_i, \rdtree_{i})$ has distribution $\n$.
\end{thm}
\begin{figure}[h]
	\centering
	\begin{tikzpicture}
		\node at (.05\linewidth, .05\linewidth){\includegraphics[width=.5\linewidth]{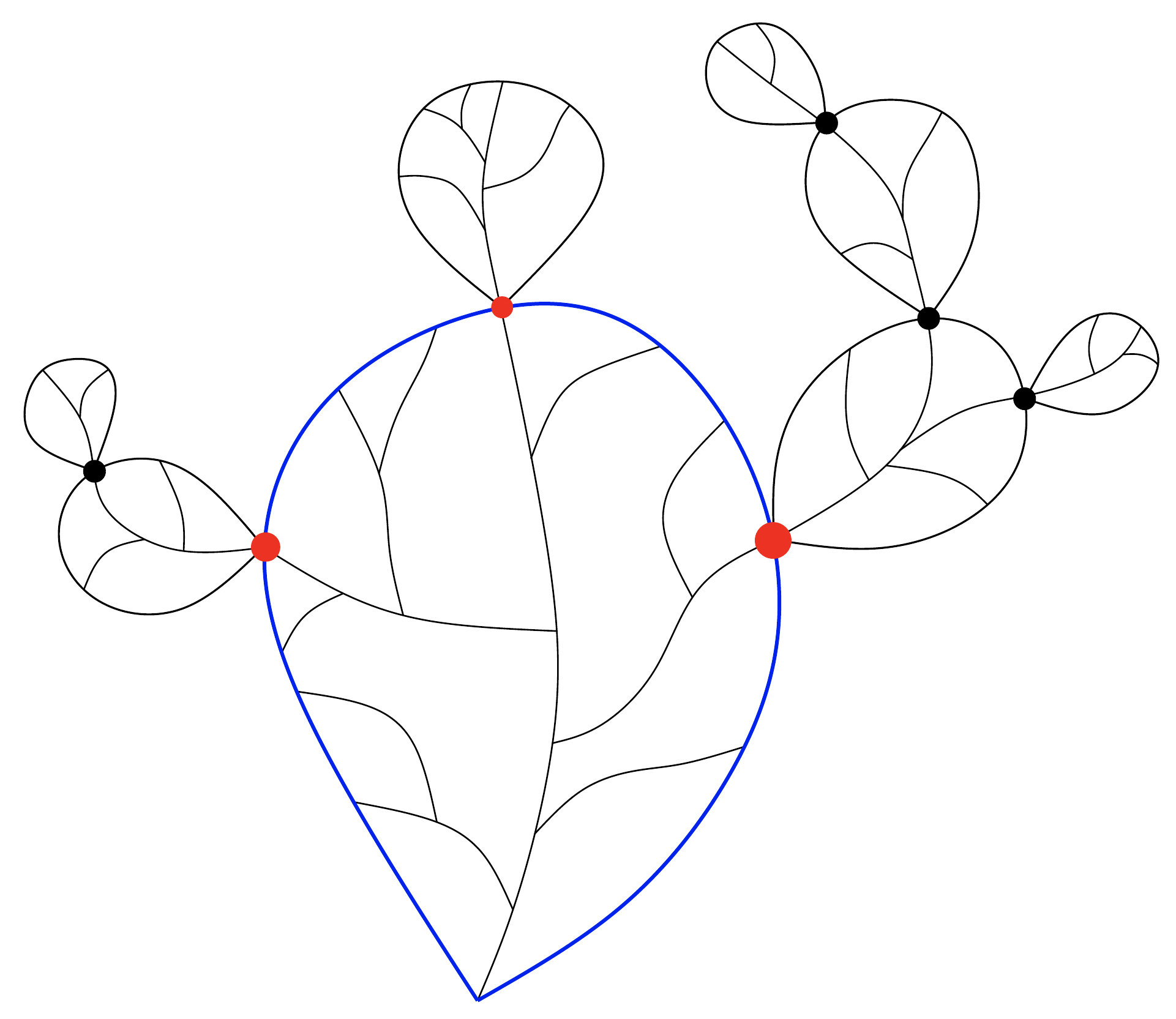}};
		\node at (.5\linewidth, .05\linewidth){\includegraphics[width=.3\linewidth]{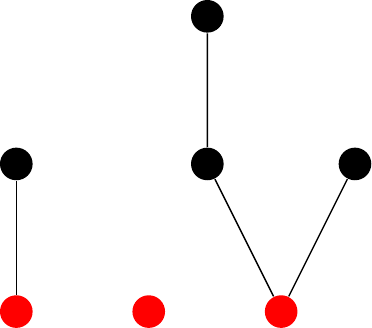}};
	\end{tikzpicture}
	\caption{Decomposition of the Lévy tree $\rdtree$ along the nodes with degree larger than $\delta$ (left) and the associated discrete forest (right). In blue: the pruned subtree $\rdtree^\delta$, in red: the first-generation nodes with degree larger than $\delta$.}
	\label{fig: degree decomposition}
\end{figure}

See Theorem \ref{thm: degree decomposition} for a more precise statement. In particular, the pruned Lévy tree $\rdtree^\delta$ which is obtained from $\rdtree$ by removing all nodes with degree larger than $\delta$ is again a Lévy tree with branching mechanism $\psi_\delta$. Thanks to this decomposition, we prove in Proposition \ref{prop: galton watson} that the discrete forest formed by nodes with degree larger than $\delta$ is a Bienaymé-Galton-Watson forest and we specify its initial distribution and its offspring distribution, see Figure \ref{fig: degree decomposition}.
\begin{remark}
	Theorem \ref{thm: degree decomposition intro} is a special case of the main result in \cite{abraham2010pruning}. In that paper, the authors study a pruning procedure on the Lévy tree defined as follows: they add some marks on the skeleton of the tree according to a Poisson point measure with intensity $\alpha_1 \Lambda$ (where $\Lambda$ is the length measure on $\rdtree$ which is the equivalent of the Lebesgue measure) and add some other marks on the infinite branching points $x_a$ with probability $p(\Delta_a)$ where $p$ is a nonnegative measurable function satisfying:
	\begin{equation*}
		\int_{(0,\infty)} r p(r) \, \pi(\dd r) <\infty.
	\end{equation*}
	Then they show that the subtree $\rdtree^{\alpha_1, p}$ containing the root obtained from $\rdtree$ by removing all the marks is again a Lévy tree and identify its branching mechanism. Furthermore, they determine the distribution of the subtrees above the marks conditionally on $\rdtree^{\alpha_1,p}$.
	It is obvious that the tree $\rdtree^\delta$ coincides with $\rdtree^{\alpha_1,p}$ where $\alpha_1 = 0$ and $p = \mathbf{1}_{(\delta,\infty)}$. Since $p$ satisfies the integrability assumption above (as $\int_{(1,\infty)} r\,\pi(\dd r) <\infty$), their result applies and gives the joint distribution of the pruned tree $\rdtree^\delta$ and the subtrees originating from the nodes with degree larger than $\delta$. However, the proof is much simpler in our particular setting.
\end{remark}

One of our main results is the next theorem giving a decomposition of the Lévy tree at its largest nodes.
Under $\n$, denote by $M_\delta$ the random variable defined by:
\begin{equation*}
	M_\delta = \frac{\e^{\g(\delta)\sigma}-1}{\g(\delta)}, \quad \text{where} \quad \g(\delta) = \pi(\delta)\e^{-\delta\n[\Delta>\delta]}.
\end{equation*}
This should be interpreted as $M_\delta = \sigma$ if $\g(\delta) = 0$ (i.e.~if $\delta$ is not an atom of $\pi$).
\begin{thm}\label{thm: conditioning by degree intro} There exists a regular conditional probability $\n[\cdot|\Delta=\delta]$ for $\delta>0$ such that $\pi[\delta,\infty)>0$, which is given by, for every measurable and bounded $F \colon \T \to \real$:
	\begin{multline}
		\n[F(\rdtree)|\Delta=\delta] = \frac{1}{\n[M_\delta \ind{\Delta< \delta}]}\sum_{k=0}^\infty \frac{\g(\delta)^{k}}{(k+1)!}\\
		\times \n\left[\int \prod_{i=1}^{k+1} \mu(\dd x_i) \for{\delta}(\dd \rdtree_i |\Delta\leqslant \delta) F(\rdtree\circledast_{i=1}^{k+1}(x_i,\rdtree_i))\ind{\Delta<\delta}\right],
	\end{multline}
	where $\n[M_\delta\ind{\Delta<\delta}]< \infty$. In particular, if $\delta >0$ is not an atom of the Lévy measure $\pi$, we have:
	\begin{equation}
		\n[F(\rdtree)|\Delta=\delta] = \frac{1}{\n[\sigma \ind{\Delta<\delta}]} \n\left[\int \mu(\dd x) \for{\delta}(\dd \widetilde{\rdtree}|\Delta\leqslant \delta) F(\rdtree\circledast (x,\widetilde{\rdtree}))\ind{\Delta<\delta}\right].
	\end{equation}
	Furthermore, if $\langle \pi, 1\rangle = \infty$, then $\n$-a.e.~$\Delta>0$, and if $\langle \pi,1\rangle < \infty$, then we have:
	\begin{equation}\label{eq: distribution of the tree with no nodes}
		\n[F(\rdtree)\ind{\Delta=0}] = \operatorname{\mathbf{N}}^{\psi_0}[F(\rdtree)\e^{-\langle \pi,1\rangle \sigma}].
	\end{equation}
\end{thm}

The proof is given in Section \ref{sect: conditioning on degree}. Some comments are in order. 
\begin{enumerate}[label=(\roman*)]
	\item Recall that the distribution of $\Delta$ is given in Proposition \ref{prop: distribution degree intro}. Together with the distribution of $\rdtree$ conditionally on $\Delta=\delta$, we can recover the unconditional distribution of the Lévy tree via the disintegration formula:
	\begin{equation*}
		\n[F(\rdtree)] = \n[F(\rdtree)\ind{\Delta=0}] + \int_{(0,\infty)} \n[\Delta\in \dd \delta] \n[F(\rdtree)|\Delta=\delta],
	\end{equation*} 
	where the first term on the right-hand side vanishes if $\pi$ is infinite.
	\item Assume that $\delta>0$ is not an atom of $\pi$. Then, conditionally on $\Delta=\delta$, the Lévy tree can be constructed as follows: take $\widetilde{\rdtree}$ with distribution $\n[\sigma \ind{\Delta\leqslant \delta}]^{-1} \allowbreak\n[\sigma \ind{\Delta\leqslant \delta} \, \dd \rdtree]$, choose a leaf uniformly at random in $\widetilde{\rdtree}$ (i.e.~according to its normalized mass measure $\tilde{\sigma}^{-1}\tilde{\mu}$) and on this leaf graft an independent Lévy forest with initial degree $\delta$ conditioned to have maximal degree $\Delta\leqslant \delta$. In fact, since $\delta$ is not an atom, this random forest will have no other nodes with degree $\delta$ besides the root. This entails that, conditionally on $\Delta=\delta$, there is a unique node realizing the maximum degree. 
	\item The situation is different when $\delta>0$ is an atom of $\pi$. In that case, conditionally on $\Delta=\delta$, the number of first-generation nodes realizing the maximal degree has a Poisson distribution. More precisely, conditionally on $\Delta=\delta$, the Lévy tree can be constructed as follows: take $\widetilde{\rdtree}$ with distribution $\n[M_\delta \ind{\Delta<\delta}]^{-1} \n[M_\delta\ind{\Delta<\delta} \, \dd \rdtree]$,
	and, conditionally on $\widetilde{\rdtree}$, graft a Poisson point measure with intensity $\g(\delta)\, \tilde{\mu}(\dd x) \allowbreak\for{\delta}(\dd \rdtree|\Delta\leqslant \delta)$ conditioned on containing at least one point.
\end{enumerate}

As a consequence, we show in Proposition \ref{prop: joint distribution degree and height} that if the Lévy measure $\pi$ is diffuse, then $\n$-a.e.~there is a unique node $X_\Delta$ with degree $\Delta$. Denote by $H_\Delta = H(X_\Delta)$ its height. Then we give a decomposition of the Lévy tree conditioned on $\Delta=\delta$ and $H_\Delta = h$, see Theorem \ref{thm: disintegration wrt degree and height}.

Finally, we turn to the behavior of a Lévy tree conditioned to have a large maximal degree. Other conditionings have been considered in the past. Duquesne \cite{duquesne2009immigration} (this is also related to Williams' decomposition, see \cite{abraham2009williams}) proved that a (sub)critical Lévy tree conditioned on having a large height converges locally to the immortal tree (which consists of an infinite spine onto which trees are grafted according to a Poisson point process). Later, He \cite{he2022local} proved the same convergence for a critical Lévy tree conditioned on having a large maximal degree $\Delta>\delta$ or a large width. In fact, his result is more general as it allows to condition by any measurable function of the tree satisfying a natural monotonicity property.

Here we treat both the critical and the subcritical cases and we consider the density version $\Delta=\delta$. Similarly to the discrete case, two drastically different types of limiting behavior appear. In the subcritical case, there is a condensation phenomenon where a node with infinite degree emerges at a finite exponentially distributed height. Denote by $X_\Delta$ the lowest node with degree $\Delta$ and let $\mathcal{F}^+_\Delta$ be the forest above $X_\Delta$, seen as a point measure on $\real_+ \times \T$. To be more precise, the forest $\mathcal{F}^+_\Delta = \sum_{i \in I} \delta_{(\ell_i, \rdtree_i)}$ is obtained by decomposing the path of the exploration process (or the height process) into excursions away from $0$, with each excursion arriving at local time $\ell_i$ and coding a tree $\rdtree_i$. Finally, let
$\pruned{\rdtree}_\Delta$ be the pruned Lévy tree, that is the Lévy tree $\rdtree$ after removing $X_\Delta$ and $\mathcal{F}^+_\Delta$. We refer the reader to Theorem \ref{thm: convergence subcritical delta} and Theorem \ref{thm: convergence subcritical} for a precise statement.
\begin{thm}
	Assume that $\psi$ is subcritical and that the Lévy measure $\pi$ is unbounded. Let $F \colon \TT \to \real$ be continuous and bounded, $\Phi \colon \real_+\times \T \to \real_+$ be continuous with bounded support and let $A_\delta$ be equal to any one of the following events: $\{\Delta=\delta\}$, $\{\Delta>\delta\}$, $\{\rdtree$ has exactly one node with degree larger than $\delta\}$ or $\{\rdtree$ has exactly one first-generation node with degree larger than $\delta \}$.
	We have:
	\begin{multline}
		\lim_{\delta \to \infty} \n\left[F(\pruned{\rdtree}_\Delta, X_\Delta)\e^{-\langle\mathcal{F}^+_\Delta,\Phi\rangle}\middle|A_\delta\right]= \alpha \n\left[\int_\rdtree F(\rdtree,x)\, \mu(\dd x)\right] \\
		\times\expp{-\int_0^\infty \dd \ell \n\left[1-\e^{-\Phi(\ell,\rdtree)}\right]}.
	\end{multline}
	In particular, conditionally on $A_\delta$, the height $H(X_\Delta)$ of $X_\Delta$ converges to an exponential distribution with mean $1/\alpha$.
\end{thm}
The last result should be interpreted as local convergence in distribution to a “condensation tree'' described as follows: start with a size-biased Lévy tree $\widetilde{\rdtree}$ with distribution $\alpha \n[\sigma \dd \rdtree]$, choose a leaf uniformly at random in $\widetilde{\rdtree}$ and on this leaf graft an independent Lévy forest with infinite degree (i.e.~a Poisson point measure on $\real_+\times \T$ with intensity $\dd \ell \n[\dd \rdtree]$). However, the limiting object is a (random) real tree which is not locally compact and the way to circumvent this difficulty is by considering the subtree above the condensation node as a point measure instead.

In the critical case, it should be no surprise that the density version $\Delta=\delta$ gives rise to the same limiting behavior as the tail version $\Delta>\delta$, namely local convergence to the immortal tree. Intuitively, this means that the condensation node goes to infinity and thus becomes invisible to local convergence. Before stating the result, let us define the immortal tree. Let $\sum_{i\in I} \delta_{(s_i,\rdtree_{i})}$ be a Poisson point measure on $\real_+\times\T$ with intensity 
\begin{equation*}
	\dd s \left(2\beta \n[\dd \rdtree] + \int_0^\infty r \,\pi(\dd r) \for{r}(\dd \rdtree)\right).
\end{equation*}
The immortal Lévy tree $\rdtree_\infty^\psi$ with branching mechanism $\psi$ is the real tree obtained by grafting the point measure $\sum_{i\in I} \delta_{(s_i,\rdtree_{i})}$ on an infinite branch. More formally, set:
\begin{equation}
	\rdtree_\infty^\psi= \real_+ \circledast_{i\in I} (s_i,\rdtree_i),
\end{equation}
where $\real_+$ is considered as a real tree rooted at $0$ and equipped with the Euclidean distance and the zero measure. In particular, thanks to \cite[Theorem 4.5]{duquesne2005probabilistic}, we have the following identity which is simply a restatement of Lemma 3.2 in \cite{duquesne2009immigration} in terms of trees:
\begin{equation}
	\ex{F(r_h(\rdtree_\infty^\psi))} = \e^{-\alpha h}\n\left[L_\sigma^h \, F(r_h(\rdtree))\right], \quad \forall h >0.
\end{equation}
\begin{thm}
	Assume that $\psi$ is critical and that $\pi$ is unbounded. Either let $A_\delta = \{\Delta=\delta\}$ and assume that the additional assumption
	\begin{equation}\label{eq: pi small atoms intro}
		\lim_{\delta \to \infty} \frac{\pi(\delta)}{\n[\sigma\ind{\Delta<\delta}]\pibar(\delta)\int_{[\delta,\infty)} r\, \pi(\dd r)} = 0
	\end{equation} 
	holds, or let $A_\delta$ be equal to any of the following events: $\{\Delta>\delta\}$, $\{\rdtree$ has exactly one node with degree larger than $\delta\}$ or $\{\rdtree$ has exactly one first-generation node with degree larger than $\delta \}$. Then, conditionally on $A_\delta$, the Lévy tree $\rdtree$ converges in distribution locally to the immortal Lévy tree $\rdtree_\infty^\psi$, i.e.~we have:
	\begin{equation}
		\lim_{\delta \to \infty}\n\left[F(r_h(\rdtree))|A_\delta\right] = \ex{F(r_h(\rdtree_\infty^\psi))}.
	\end{equation}
\end{thm}
We refer to Theorem \ref{thm: convergence critical delta} and Theorem \ref{thm: convergence critical} for a precise statement. The assumption \eqref{eq: pi small atoms intro} is a technical condition which guarantees a fast decay for the size of the atoms of $\pi$. Observe that we have $\lim_{\delta\to \infty}\n[\sigma \ind{\Delta< \delta}] = \n[\sigma]$ which is infinite since $\psi$ is critical. Also notice that \eqref{eq: pi small atoms intro} is automatically satisfied if the Lévy measure $\pi$ is diffuse.

It is worth noting that in the critical case, conditioning by the different events $A_\delta$ yields the same limiting behavior even though in general they are not equivalent in $\n$-measure. In the stable (critical) case $\psi(\lambda) = \lambda^\gamma$ with $\gamma \in (1,2)$, these quantities can be computed explicitly, see Proposition \ref{prop: three conditionings stable}. In that case, we also show in Proposition \ref{prop: galton watson stable} that, conditionally on $\Delta>\delta$, the distribution of the Bienaymé-Galton-Watson forest of nodes with degree larger than $\delta$ is independent of $\delta$.

The rest of the paper is organized as follows. In Section \ref{sect: notation}, we set notation and we introduce the main object we will be dealing with, namely the exploration process. We compute the distribution of the maximal degree in Section \ref{sect: distribution of degree}, then we give a Poissonian decomposition of the exploration process along the large nodes and study their structure in Section \ref{sect: decomposition}. In Section \ref{sect: conditioning on degree} (resp.~Section \ref{sect: conditioning on degree and height}), we make sense of the exploration process conditioned to have a fixed maximal degree (resp.~a fixed maximal degree at a given height). Sections \ref{sect: local limit} and \ref{sect: other conditionings} deal with the local convergence of the exploration process conditioned to have large maximal degree. Finally, Section \ref{sect: stable} is devoted to the study of the stable case $\psi(\lambda) = \lambda^\gamma$.

\section{The exploration process and the Lévy tree}\label{sect: notation}
In this section, we will recall the construction of the exploration process introduced in \cite{le1998branching} and later developped in \cite{duquesne2002random}.
\subsection{Notation}
If $E$ is a Polish space, let $\B(E)$ be the set of real-valued and nonnegative measurable functions defined on $E$ endowed with its Borel $\sigma$-field. For any measure $\nu$ on $E$ and any function $f \in \B(E)$, we write $\langle \nu,f \rangle = \int_E f(x) \, \nu(\dd x)$. We also denote by $\supp{\nu}$ the closed support of the measure $\nu$ in $E$.

We denote by $\Mf(E)$ the set of finite measures on $E$ endowed with the topology of weak convergence. For every $\nu \in \Mf(\real_+)$, we set:
\begin{equation}
	H(\nu) = \sup \supp{\nu},
\end{equation}
with the convention that $H(0) = 0$. Moreover, we let
\begin{equation}
	\Delta(\nu) = \sup\{\nu(x)\colon \, x \geqslant 0\}
\end{equation}
be the largest atomic mass of $\nu$. We say that $\nu$ is diffuse if it has no atoms and set $\Delta(\nu) = 0$ by convention.

Denote by 
\begin{equation}
	\D = D(\real_+, \Mf(\real_+))
\end{equation} 
the set of $\Mf(\real_+)$-valued càdlàg functions equipped with the Skorokhod $J_1$-topology. For a function $\mu =(\mu_t,\, t \geqslant 0)\in \D$, let
\begin{equation}
	\Delta(\mu) = \sup_{t \geqslant 0} \Delta(\mu_t)
\end{equation}
be the largest atom of the entire path of $\mu$.

\subsection{The underlying Lévy process and the height process}
We consider a (sub)critical branching mechanism of the form
\begin{equation}
	\psi(\lambda) = \alpha \lambda + \beta \lambda^2 + \int_{(0,\infty)} (\e^{-\lambda r} - 1+\lambda r) \, \pi(\dd r), \quad \forall \lambda \geqslant 0,
\end{equation}
where $\alpha,\beta \geqslant 0$ and $\pi \neq 0$ is a $\sigma$-finite measure on $(0,\infty)$ satisfying $\int_{(0,\infty)} (r\wedge r^2)\, \pi(\dd r) < \infty$. We consider a spectrally positive Lévy process $X = (X_t, \, t \geqslant 0)$ with Laplace exponent $\psi$ starting from $0$. Namely, we have:
\begin{equation*}
	\ex{\e^{-\lambda X_t}} = \e^{t \psi(\lambda)}, \quad \forall t,\lambda \geqslant 0.
\end{equation*}
We assume that $X$ is of infinite variation a.s.~which is equivalent to the following condition:
\begin{equation}\label{eq: infinite variation}
	\beta >0 \quad \text{or} \quad \int_{(0,1)} r\, \pi(\dd r ) = \infty.
\end{equation}
Duquesne and Le Gall \cite{duquesne2002random} proved that there exists a process $H = (H_t, \, t \geqslant 0)$ called the $\psi$-height process such that for every $t \geqslant 0$, we have the following convergence in probability:
\begin{equation}\label{eq: definition height process}
	H_t = \liminf_{\epsilon \to 0} \frac{1}{\epsilon} \int_0^t \ind{I_t^s < X_s < I_t^s + \epsilon} \, \dd s,
\end{equation}
where, for every $0\leqslant s \leqslant t$, $I_t^s = \inf_{s\leqslant r \leqslant t} X_r$ is the past infimum. They also proved a Ray-Knight theorem for $H$ which shows that the $\psi$-height process $H$ describes the genealogy of the $\psi$-CB process, see \cite[Theorem 1.4.1]{duquesne2002random}.

\subsection{The exploration process}\label{subsect: exploration process}
Although the height process is not Markov in general, it is a simple function of a measure-valued Markov process, the so-called exploration process that we now introduce. The exploration process $\rho = (\rho_t, \, t\geqslant 0)$ is the $\Mf(\real_+)$-valued process defined as follows:
\begin{equation}
	\rho_t(\dd r) = \beta \mathbf{1}_{[0,H_t]}(r)\, \dd r + \sum_{\substack{0<s\leqslant t,\\X_{s-}<I_t^s}} (I_t^s - X_{s-})\delta_{H_s}(\dd r).
\end{equation}
In particular, the total mass of $\rho_t$ is $\langle \rho_t,1\rangle = X_t - I_t$.

We will sometimes refer to $t \geqslant 0$ as a \emph{node} in reference to the corresponding real tree when it is well defined (see Section \ref{subsect: levy tree}). For $s,t \geqslant 0$, we say that $s$ is an ancestor of $t$ and we write $s\preccurlyeq t$ if $s\leqslant t$ and $H_s = \inf_{s\leqslant r \leqslant t} H_r$.  The set
\begin{equation}\label{eq: ancestral line}
	\{s \geqslant 0 \colon \, s\preccurlyeq t\}
\end{equation}
is called the ancestral line of $t$. We say that $t \geqslant 0$ is a first-generation node with property $A\subset \Mf(\real_+)$ if $\rho_t \in A$ and $\rho_s \notin A$ for every (strict) ancestor $s$ of $t$. For example, we will say that $t$ is a first-generation node with mass larger than $\delta>0$ if $\Delta(\rho_t) > \delta$ and $\Delta(\rho_s)\leqslant \delta$ for every $s\preccurlyeq t$ with $s\neq t$. Given $0\leqslant t_1\leqslant \cdots \leqslant t_n$, there exists a unique $s \geqslant 0$ such that $r\preccurlyeq t_i$ for every $1\leqslant i \leqslant n$ if and only if $r \preccurlyeq s$. We write $s = t_1\wedge \cdots \wedge t_n$ and call it the most recent common ancestor (MRCA for short) of $t_1, \ldots, t_n$.

One can recover the heigth process from the exploration process as follows. Denote by $\Delta X_t = X_t - X_{t-}$ the jump of the process $X$ at time $t$.
\begin{proposition}
	Almost surely for every $t>0$, we have:
	\begin{enumerate}[label=(\roman*)]
		\item $H(\rho_t) = H_t$,
		\item $\rho_t= 0$ if and only if $H_t = 0$,
		\item if $\rho_t \neq 0$, then $\supp{\rho_t} = [0,H_t]$,
		\item $\rho_t = \rho_{t-} + \Delta X_t  \delta_{H_t}$.
	\end{enumerate}
\end{proposition}
In the definition of the exploration process, since $X$ starts from $0$, we have $\rho_0 = 0$. In order to state the Markov property of $\rho$, we have to define the process $\rho$ starting from any initial measure $\nu \in \Mf(\real_+)$. To that end, for every $a \in [0, \langle \nu,1\rangle]$, we define the erased measure $k_a \nu$ by:
\begin{equation*}
	k_a \nu [0,r] = \nu[0,r] \wedge (\langle \nu ,1\rangle - a), \quad \forall r \geqslant 0.
\end{equation*}
If $a > \langle \nu,1\rangle$, we set $k_a \nu = 0$. In words, the measure $k_a \nu$ is obtained from $\nu$ by erasing a mass $a$ backward starting from $H(\nu)$. For $\mu \in \Mf(\real_+)$ with bounded support, we define the concatenation $[\mu,\nu]\in \Mf(\real_+)$ of the measures $\mu, \nu$ by:
\begin{equation*}
	\langle [\mu,\nu], f \rangle = \langle \mu, f \rangle + \langle \nu, f(H(\mu)+\cdot)\rangle, \quad \forall f \in \B(\real_+).
\end{equation*}
Finally, we set $\rho_0^\nu = \nu$ and
\begin{equation*}
	\rho_t^\nu = [k_{-I_t}\nu, \rho_t], \quad\forall t > 0.
\end{equation*}
We say that $\rho^\nu = (\rho_t^\nu, \, t \geqslant 0)$ is the exploration process started at $\nu$ and we write $\mathbb{P}_\nu$ for its distribution.
\begin{proposition}
	For any $\nu \in \Mf(\real_+)$, the process $\rho^\nu = (\rho_t^\nu, \, t \geqslant 0)$ is a càdlàg strong Markov process in $\Mf(\real_+)$.
\end{proposition}

\subsection{The excursion measure of the exploration process}
Let us introduce the excursion measure $\n$. Denote by $I = (I_t, \, t\geqslant 0)$ the infimum process of $X$:
\begin{equation}
	I_t = \inf_{0\leqslant s \leqslant t} X_s.
\end{equation}
Standard results (see e.g.~\cite{bertoin1996levy}) entail that $X-I$ is a strong Markov process with values in $\real_+$ and that the point $0$ is regular. Furthermore, $-I$ is a local time at $0$ for $X-I$. We denote by $\n$ the associated excursion measure of the process $X-I$ away from $0$. It is not difficult to see from \eqref{eq: definition height process} that $H_t$ (and thus also $\rho_t$) only depends on the excursion of $X-I$ above $0$ which straddles time $t$. It follows that the excursion measure of $\rho$ away from $0$ is the “distribution'' of $\rho$ under $\n$. We still denote it by $\n$ and we let 
\begin{equation}
	\sigma = \inf\{t >0\colon\, \rho_t=0\}
\end{equation} be the lifetime of $\rho$ under $\n$ (this coincides with the lifetime of $X-I$ under $\n$). In particular, the following holds for every $\lambda >0$:
\begin{equation}\label{eq: laplace transform sigma}
	\n\left[1-\e^{-\lambda \sigma}\right] = \psi^{-1}(\lambda) \quad \text{and}\quad \n\left[\sigma \e^{-\lambda \sigma}\right] = \frac{1}{\psi'\circ \psi^{-1}(\lambda)},
\end{equation}
where $\psi^{-1}$ is the inverse function of $\psi$. By letting $\lambda \to 0$ we obtain:
\begin{equation}\label{eq: expectation sigma}
	\n[\sigma] = \frac{1}{\alpha},
\end{equation}
with the convention that $1/0 = \infty$. Let us recall Bismut's decomposition for the exploration process. Let $J_a$ be the random element in $\Mf(\real_+)$ defined by $J_a(\dd r)= \mathbf{1}_{[0,a]}(r)\,\dd U_r$, where $U$ is a subordinator with Laplace exponent
\begin{equation}\label{eq: definition phi}
	\phi(\lambda) = \frac{\psi(\lambda)}{\lambda}-\alpha=\beta \lambda + \int_0^\infty \left(1-\e^{-\lambda r}\right)\pibar(r)\, \dd r,
\end{equation}
where the tail $\pibar$ of the Lévy measure $\pi$ is defined in \eqref{eq: definition pibar}.
\begin{proposition}\label{prop: representation exploration}
	For every $F \in \B(\Mf(\real_+))$, we have:
	\begin{equation}
		\n\left[\int_0^\sigma F(\rho_t) \, \dd t\right] = \int_0^\infty \dd a \, \e^{-\alpha a} \ex{F(J_a)}.
	\end{equation}
\end{proposition}

\subsection{Local times of the height process}\label{subsect: local times}
Although the height process $H$ is not Markov in general, one can show that its local time process exists under $\mathbb{P}$ or $\n$. More precisely, for every $a> 0$, there exists a continuous nondecreasing process $(L_s^{a}, \, s \geqslant 0)$ which can be characterized via the approximation:
\begin{equation*}
	\lim_{\epsilon \to 0} \n\left[\ind{\sup H> h} \sup_{0\leqslant s \leqslant t} \left|\frac{1}{\epsilon} \int_0^s \ind{a-\epsilon < H_r\leqslant a} \, \dd r - L_s^{a}\right|\right] = 0, \quad \forall t,h \geqslant 0.
\end{equation*}
Moreover, $\n$-a.e. the support of the measure $L^{a}(\dd s)\coloneqq \dd_s L_s^{a}$ is contained in $\{s\geqslant 0\colon \,H_s = a\}$ and we have the occupation time formula $\int_0^\infty \dd a \,L^{a}(\dd s) = \mathbf{1}_{[0,\sigma]}(s) \, \dd s$. Furthermore, the process $(L_\sigma^{a}, \, a \geqslant 0)$ is a $\psi$-CB process under its canonical measure.

Let us recall the excursion decomposition of the exploration process above level $h >0$. Set $\tau^{h}_s = \inf\{t >0 \colon \, \int_0^t \ind{H_r\leqslant h} \, \dd r > s\}$ and define the truncated exploration process by:
\begin{equation}\label{eq: definition exploration under h}
	r_h(\rho) = (\rho_{\tau^{h}_s}, \, s \geqslant 0).
\end{equation}
Denote by $\mathcal{E}_h$ the $\sigma$-field generated by the process $r_h(\rho)$. Let $(\alpha_i, \beta_i), i\in I_h$ denote the excursion intervals of $H$ above level $h$. For every $i\in I$, we define the measure-valued process $\rho^{i}$ by:
\begin{equation*}
	\langle \rho^{i}_s , f \rangle = \int_{(a,\infty)} f(r-a) \, \rho_{\alpha_i + s} (\dd r) \quad \text{if } 0<s<\beta_i - \alpha_i
\end{equation*}
and $\rho^{i}_s = 0$ if $s = 0$ or $s \geqslant \beta_i-\alpha_i$. Finally, let $\ell_i= L^{h}_{\alpha_i}$ be the local time at level $h$ at the beginning of the excursion $\rho^{i}$.
\begin{proposition}\label{prop: branching property}
	Under $\n$, conditionally on $\mathcal{E}_h$, the random measure $\sum_{i \in I} \delta_{(\ell_i, \rho^{i})}$ is a Poisson point measure with intensity $\mathbf{1}_{[0,L_\sigma^h]}(\ell) \, \dd \ell \n[\dd \rho]$.
\end{proposition} 

\subsection{The dual process}
We shall need the $\Mf(\real_+)$-valued process $\eta = (\eta_t, \, t \geqslant 0)$ defined by:
\begin{equation}
	\eta_t(\dd r) = \beta \mathbf{1}_{[0,H_t]}(r)\, \dd r + \sum_{\substack{0<s\leqslant t,\\X_{s-}<I_t^s}} (X_s - I_t^s)\delta_{H_s}(\dd r).
\end{equation}
The process $\eta$ is the dual process of $\rho$ under $\n$ thanks to the following time-reversal property.
\begin{proposition}\label{prop: time reversal}
	The processes $((\rho_t, \eta_t), \, t \geqslant 0)$ and $((\eta_{(\sigma -s)-},\rho_{(\sigma-s)-}), \, t\geqslant 0)$ have the same distribution under $\n$.
\end{proposition}

\subsection{Grafting procedure}
We now explain how to insert a finite collection of measured-valued processes into a measure-valued process. Let $\mu = (\mu(t), \, 0\leqslant t < \sigma)$ be a $\Mf(\real_+)$-valued function with lifetime $\sigma \in (0,\infty]$ such that $\mu(t)$ has bounded support for every $t \in [0,\sigma)$ and let $\sum_{i= 1}^N  \delta_{(s_i,\mu_i)}$ be a finite point measure on $\real_+\times \D$ where the $s_i$ are arranged in increasing order and each $\mu_i$ has a finite lifetime $\sigma^{i}$. Set $s_0 = \Sigma_0 = 0$ and
\begin{equation*}
	\Sigma_i = \sum_{k=1}^{i} \sigma^k, \quad \forall i \geqslant 1.
\end{equation*}
Define a measure-valued process $\tilde{\mu}$ by:
\begin{equation*}
	\tilde{\mu}(t) = \begin{cases}
		\mu(t-\Sigma_i) & \text{if } s_{i-1} +\Sigma_{i-1} \leqslant t < (s_i \wedge \sigma) + \Sigma_{i-1}, \\
		[\mu(s_i), \mu_i(t-s_i-\Sigma_{i-1})] & \text{if } s_i + \Sigma_{i-1}\leqslant t < s_i + \Sigma_i \text{ and } s_i < \sigma.
	\end{cases}
\end{equation*}
Observe that the $(s_i, \mu_i)$ such that $s_i \geqslant \sigma$ do not play a role in this construction and that $\tilde{\mu}$ has lifetime
\begin{equation*}
	\sigma + \sum_{i\colon s_i < \sigma} \sigma^k.
\end{equation*}
We denote this grafting procedure by:
\begin{equation}\label{eq: definition grafting}
	\mu \circledast_{i= 1}^N (s_i,\mu_i) = \tilde{\mu}.
\end{equation}
In words, this is the process obtained from $\mu$ by inserting the measure-valued process $\mu_i$ into $\mu$ at time $s_i < \sigma$.

\subsection{A Poissonian decomposition of the exploration process}\label{subsect: branching property}
Let $\nu \in \Mf(\real_+)$. We write $\mathbb{P}^{\psi,\ast}_\nu$ for the distribution of the exploration process $\rho$ starting at $\nu$ and killed when it first reaches $0$. Let us introduce two probability measures on $\D$ that will play a major role in the rest of the paper. For every $r>0$, we will write $\for{r}$ for $\mathbb{P}^{\psi,\ast}_{r\delta_0}$. This should be interpreted as the distribution of the exploration processes with initial mass $r$. Furthermore, for every $\delta >0$ such that $\pibar(\delta)>0$, set:
\begin{equation}\label{eq: definition random forest}
	\rfor(\dd \rho) = \frac{1}{\pibar(\delta)} \int_{(\delta,\infty)} \pi(\dd r) \for{r}(\dd \rho),
\end{equation}
which is the distribution of the exploration process starting from a random initial mass with distribution $\pi$ conditioned on being larger than $\delta$.

We decompose the path of $\rho$ under $\mathbb{P}^{\psi,\ast}_\nu$ according to excursions of the total mass of $\rho$ above its minimum. Let $(\alpha_i,\beta_i), i\in I$ denote the excursion intervals of the process $X-I$ away from $0$ under $\mathbb{P}^{\psi,\ast}_\nu$. Define the measure-valued process $\rho^{i}$ by $\rho_{(\alpha_i +s)\wedge \beta_i} = [k_{-I_{\alpha_i}} \nu, \rho_s^{i}]$.
\begin{lemma}\label{lemma: poisson decomposition exploration}
	The random measure $\sum_{i\in I} \delta_{(-I_{\alpha_i},\rho^{i})}$ is under $\mathbb{P}^{\psi,\ast}_\nu$ a Poisson point measure with intensity $\mathbf{1}_{[0,\langle \nu,1\rangle)}(u)\, \dd u \n[\dd \rho]$. In particular, under $\for{r}$, it is a Poisson point measure with intensity $\mathbf{1}_{[0,r]}(u)\, \dd u \n[\dd \rho]$.
\end{lemma}
Using this decomposition, we can give another useful interpretation of the measure $\for{r}$. Let $\rho$ be the exploration process starting from $0$ and let $(L_s^0,\, s \geqslant 0)$ be its local time process at $0$. Then the process $(\tilde{\rho}^{(r)}_t, \, t\geqslant 0)$ defined by:
\begin{equation}\label{eq: representation forest r}
	\tilde{\rho}_t^{(r)} = (r- L_t^0)_+ \delta_0 + \rho_t \ind{L_t^0 \leqslant r}
\end{equation}
has distribution $\for{r}$.

In the next lemma, we identify the distribution of the exploration process above level $H(\nu)$ starting from $\nu$. For a measure $\nu \in \Mf(\real_+)$ and a positive real $a>0$, define $\abov{a}(\nu)$ as the measure $\nu$ shifted by $a$. More formally, define a measure $\abov{a}(\nu)$ on $\real_+$ by setting:
\begin{equation*}
	\left\langle \abov{a}(\nu),f \right\rangle = 
	\int_{[a,\infty)} f(x-a)\, \nu(\dd x),
\end{equation*}
for every $f \in \B(\real_+)$ if $a \leqslant H(\nu)$ and $\abov{a}(\nu) =0$ if $a>H(\nu)$.

\begin{lemma}\label{lemma: exploration above initial level}
	Let $\nu \in \Mf(\real_+)$ such that $\nu(H(\nu))>0$. Under $\mathbb{P}^{\psi,\ast}_\nu$, the process $\tilde{\rho} = (\abov{H(\nu)}(\rho_t),\, t\geqslant 0)$ stopped at the first time it hits $0$ has distribution $\for{\nu(H(\nu))}$.
\end{lemma}
\begin{proof}
	We shall use the Poisson decomposition of Lemma \ref{lemma: poisson decomposition exploration}. Using its notations, we have $\rho_{(t+\alpha_i) \wedge \beta_i} = [k_{-I_{\alpha_i}}\nu, \rho_t^{i}]$ where $\sum_{i\in I} \delta_{(-I_{\alpha_i}, \rho^{i})}$ is a Poisson point measure with intensity $\mathbf{1}_{[0,\langle \nu, 1\rangle)}(u) \, \dd u \n[\dd \rho]$. Thus, the exploration process above level $H(\nu)$ stopped at the first time it hits $0$ satisfies:
	\begin{equation*}\abov{H(\nu)}(\rho_{(t+\alpha_i) \wedge \beta_i}) = (\nu(H(\nu))+I_{\alpha_i})\delta_0 + \rho_t^{i}  
	\end{equation*} 
	if $-I_{\alpha_i} \leqslant \mu(H(\nu))$ and it is zero if $-I_{\alpha_i} > \nu(H(\nu))$. Applying Lemma \ref{lemma: poisson decomposition exploration} again, it is easy to see that this is also the Poisson decomposition of $\rho$ under $\for{\nu(H(\nu))}$. This proves the desired result.
\end{proof}

\subsection{The Lévy tree}\label{subsect: levy tree}
Recall that the Grey condition
\begin{equation}\label{eq: grey condition}
	\int^\infty \frac{\dd \lambda}{\psi(\lambda)} < \infty
\end{equation}
is equivalent to the almost sure extinction of the $\psi$-CB process in finite time. If it holds, then the height process $H$ admits a continuous version and one can use the coding of real trees by continuous excursions (see e.g.~\cite{evans2007probability}) in order to define the Lévy tree $\rdtree$ as the tree coded by the height process $H$ under its excursion measure $\n$. Then the Grey condition implies that $\rdtree$ is a \emph{compact} real tree. In the rest of the paper we forego this assumption, but we still interpret the results in terms of trees as they are easier to understand.

\section{Distribution of the maximal degree}\label{sect: distribution of degree}
Under $\n$, denote by $\Delta = \Delta(\rho)$ the largest atomic mass of the exploration process. Thanks to \cite[Theorem 4.6]{duquesne2005probabilistic}, if $\langle \pi,1\rangle<\infty$ then the set of discontinuity times of $\rho$ is $\n$-a.e.~finite (and possibly empty). On the other hand, if $\langle \pi,1\rangle = \infty$ then it is countable and dense in $[0,\sigma]$. In particular, in that case we have that $\n$-a.e.~$\Delta>0$. The main result of this section is the following proposition giving the joint distribution of the lifetime $\sigma$ and the maximal degree $\Delta$ under $\n$. Recall from \eqref{eq: definition pibar} and \eqref{eq: definition psi delta} the definitions of $\pibar$ and $\psi_\delta$, and define:
\begin{equation}\label{eq: definition psi delta minus}
	\psi_{\delta-}(\lambda) = \lim_{\epsilon \uparrow 0}\psi_{\delta-\epsilon}(\lambda)=\left(\alpha + \int_{[\delta,\infty)} r\, \pi(\dd r)\right)\lambda + \beta \lambda^2 + \int_{(0,\delta)} \left(\e^{-\lambda r}-1+\lambda r\right)\, \pi(\dd r).
\end{equation}
\begin{proposition}\label{prop: joint distribution degree lifetime}
	For every $\delta >0$ and $\lambda \geqslant 0$, we have:
	\begin{align}
		\n\left[1-\e^{-\lambda \sigma}\ind{\Delta \leqslant \delta}\right] &= \psi_\delta^{-1}(\pibar(\delta)+\lambda), \label{eq: joint distribution degree lifetime large}\\
		\n\left[1-\e^{-\lambda \sigma}\ind{\Delta < \delta}\right] &= \psi_{\delta-}^{-1}(\pi[\delta,\infty)+\lambda).
		\label{eq: joint distribution degree lifetime strict}
	\end{align}
	In particular, we have:
	\begin{align}
		\n[\Delta >\delta] &= \psi_\delta^{-1}(\pibar(\delta)), \label{eq: distribution degree strict}\\
		\n[\Delta\geqslant \delta] &= \psi_{\delta-}^{-1}(\pi[\delta,\infty)). \label{eq: distribution degree large}
	\end{align}
	Furthermore, if $\langle \pi,1\rangle <\infty$, then we have:
	\begin{equation}
		\n\left[1-\e^{-\lambda \sigma}\ind{\Delta = 0}\right] = \psi_0^{-1}(\langle \pi,1\rangle+\lambda).\label{eq: degree 0 lifetime}
	\end{equation}
\end{proposition}

\begin{proof}
	We only prove \eqref{eq: joint distribution degree lifetime large}, the proof of \eqref{eq: joint distribution degree lifetime strict} being similar. Fix $\delta >0$ and let $\lambda,\mu \geqslant 0$. Let
	\begin{equation*}
		A = \{\nu \in \Mf(\real_+)\colon \, \nu \text{ has an atom with mass} >\delta\}.
	\end{equation*}
	We shall compute
	\begin{equation}\label{eq: definition v}
		v(\lambda,\mu) = \n\left[1-\e^{-\lambda \sigma - \mu \int_0^\sigma \dd t \,\ind{\rho_t \in A}}\right].
	\end{equation}
	We have:
	\begin{align*}
		v(\lambda,\mu) &= 
		\n\left[\int_0^\sigma \dd t \,(\lambda+\mu \ind{\rho_t \in A}) \e^{-\lambda(\sigma-t)-\mu\int_t^\sigma \dd s \,\ind{\rho_s \in A}}\right] \\
		&= 	 \n\left[\int_0^\sigma \dd t \,(\lambda+\mu \ind{\rho_t \in A})\exstar{\rho_t}{\e^{-\lambda\sigma-\mu\int_0^\sigma \dd s \,\ind{\rho_s \in A}}}\right],
	\end{align*}
	where we applied the Markov property for the last equality. We shall use Lemma \ref{lemma: poisson decomposition exploration} to compute the last expectation.
	
	For a measure $\nu \in \Mf(\real_+)$, denote by $H'(\nu)$ the first atom of $\nu$ with mass larger than $\delta$:
	\begin{equation*}
		H'(\nu) = \inf\{x \geqslant 0  \colon \, \nu(x) >\delta\},
	\end{equation*}
	with the convention that $\inf \emptyset = +\infty$.
	
	Suppose that $\nu \in A$. Recall from Section \ref{subsect: branching property} the excursion decomposition of the exploration process above the minimum of its total mass under $\mathbb{P}^{\psi,\ast}_\nu$. Notice that if $-I_{\alpha_i} < \nu\left([H'(\nu),H(\nu)]\right) - \delta$, then $\rho_{(\alpha_i +s)\wedge \beta_i} \in A$ for every $s \geqslant 0$. On the other hand, if $-I_{\alpha_i} > \nu\left([H'(\nu),H(\nu)]\right) - \delta$, then $\rho_{(\alpha_i +s)\wedge \beta_i} \in A$ if and only if $\rho_s^{i} \in A$. It follows that
	\begin{multline*}
		\exstar{\nu}{\e^{-\lambda\sigma-\mu\int_0^\sigma \dd s \,\ind{\rho_s \in A}}}\\
		\begin{aligned}
			&= \exstar{\nu}{\expp{-\sum_{i\in I} \left(\lambda \sigma^{i}+\mu \int_0^{\beta_i - \alpha_i}\dd s\, \ind{\rho_{\alpha_i +s} \in A}\right)}}\\
			&= \expp{-\int_0^{\langle \nu,1\rangle} \dd u \n\left[1-\e^{-(\lambda+\mu\ind{u<\nu[H'(\nu),H(\nu)] - \delta})\sigma -\mu \ind{u>\nu[H'(\nu),H(\nu)] - \delta} \int_0^\sigma \dd s\, \ind{\rho_s\in A}}\right]}\\
			&= \expp{-\left(\nu[H'(\nu),H(\nu)] - \delta\right)\psi^{-1}(\lambda+\mu)-\left(\nu[0,H'(\nu)) + \delta\right)v(\lambda,\mu)}.
		\end{aligned}
	\end{multline*}
	
	Now suppose that $\nu \notin A$. Then $\mathbb{P}_\nu^{\psi,\ast}$-a.s. we have the equality $\{\rho_{(\alpha_i +s)\wedge \beta_i} \in A\}= \{\rho_s^{i} \in A\}$. It follows that
	\begin{equation*}
		\exstar{\nu}{\e^{-\lambda\sigma-\mu\int_0^\sigma \dd s \,\ind{\rho_s \in A}}} = \expp{-\langle \nu , 1\rangle v(\lambda, \mu)}.
	\end{equation*}
	We deduce that $v(\lambda,\mu)$ is equal to
	\begin{multline}\label{eq: v equal to}
		(\lambda+\mu)\n\left[\int_0^\sigma \dd t \,\ind{\rho_t \in A} \expp{-\left(\rho_t[H'_t,H_t] - \delta\right)\psi^{-1}(\lambda+\mu)-\left(\rho_t[0,H'_t) + \delta\right)v(\lambda,\mu)}\right]\\
		+\lambda\n\left[\int_0^\sigma \dd t \,\ind{\rho_t \notin A}\expp{-\langle \rho_t,1\rangle v(\lambda,\mu)}\right],
	\end{multline}
	where $H'_t = H'(\rho_t)$.
	
	Thanks to Proposition \ref{prop: representation exploration}, for every $\theta, \omega \geqslant 0$ we have:
	\begin{equation}\label{eq: exploration and subordinator}
		\n\left[\int_0^\sigma \dd t \,\ind{\rho_t \in A}\expp{-\theta\rho_t[0,H'_t)-\omega \rho_t[H'_t,H_t]}\right] = \int_0^\infty \dd a \, \e^{-\alpha a} f(a,\theta,\omega),
	\end{equation}
	where we set
	\begin{equation*}
		f(a,\theta,\omega)\coloneqq \ex{\ind{J_a \in A} \e^{-\theta  J_a[0,H'(J_a))-\omega J_a[H'(J_a),H(J_a)]}}.
	\end{equation*}
	
	Recall that $J_a(\dd r)= \mathbf{1}_{[0,a]}(r)\, \dd U_r$ where $U$ is a subordinator with Laplace exponent $\phi$ defined in \eqref{eq: definition phi}. Denote by $T$ be the time of the first jump of $U$ exceeding $\delta$:
	\begin{equation}
		T \coloneqq \inf \left\lbrace r>0\colon \, \Delta U_r >\delta\right\rbrace
	\end{equation}
	Then it is clear that $H(J_a)=a$, $H'(J_a)= T$ and $\{J_a \in A\} = \{T \leqslant	a\}$. Thus, we get:
	\begin{align*}
		f(a,\theta,\omega) &= \ex{\ind{T \leqslant a} \e^{-\theta  U_{T-}-\omega \Delta U_{T}-\omega (U_a-U_{T})}} \\
		&=\ex{\ind{T \leqslant a} \e^{-\theta  U_{T-}-\omega \Delta U_{T}-\phi(\omega)(a-T)}},
	\end{align*}
	where we used the strong Markov property at time $T$ for the last equality.
	
	Set
	\begin{equation}
		s_\delta = \int_\delta^\infty \pibar(r)\, \dd r \quad \text{and} \quad \phi_\delta(\lambda) =\beta \lambda +\int_0^\delta (1-\e^{-\lambda r})\pibar(r)\, \dd r.
	\end{equation}
	Using basic results on Poisson point processes, we have that $T$ is exponentially distributed with mean $1/s_\delta$, $\Delta U_{T}$ has distribution $s_\delta^{-1}\mathbf{1}_{[\delta,\infty)}(x) \pibar(x)\, \dd x$ and is independent of $T$, and the process $(U_r, \, 0\leqslant r < T)$ is distributed as $(V_r, \, 0\leqslant r < T)$, where $V$ is a subordinator with Laplace exponent $\phi_\delta$, independent of $(T, \Delta U_{T})$. Therefore, it follows that
	\begin{equation*}
		f(a,\theta,\omega) = \int_0^{a} \dd t \,\e^{-s_\delta t-\phi_\delta(\theta)t - \phi(\omega)(a-t)} \int_\delta^\infty \dd x\, \pibar(x) \e^{-\omega x}.
	\end{equation*}
	We deduce from \eqref{eq: exploration and subordinator} that
	\begin{multline}\label{eq: v 1}
		\n\left[\int_0^\sigma \dd t\, \ind{\rho_t \in A}\expp{-\theta\rho_t([0,H'_t))-\omega \rho_t([H'_t,H_t])}\right]\\
		= \frac{1}{(\alpha + \phi(\omega))(s_\delta + \alpha + \phi_\delta(\theta))}\int_\delta ^\infty \dd x\, \pibar(x) \e^{-\omega x}.
	\end{multline}
	Similar arguments yield
	\begin{align}\label{eq: v 2}
		\n\left[\int_0^\sigma \dd t\, \ind{\rho_t \notin A}\expp{-\theta\langle \rho_t,1\rangle }\right] 
		&=\int_0^\infty \dd a \,\e^{-\alpha a} \ex{\ind{J_a \notin A} \e^{-\theta \langle J_a,1\rangle}} \nonumber\\
		&= \int_0^\infty \dd a \,\e^{-\alpha a} \ex{\ind{T > a} \e^{-\theta U_a}}\nonumber\\
		&= \frac{1}{s_\delta +\alpha +\phi_\delta(\theta)}\cdot
	\end{align}
	
	It follows from \eqref{eq: v equal to}, \eqref{eq: v 1} and \eqref{eq: v 2} that
	\begin{multline}\label{eq: v equal to bis}
		v(\lambda,\mu) = \frac{(\lambda+\mu) \e^{\delta(\psi^{-1}(\lambda+\mu)-v(\lambda,\mu))}}{(\alpha + \phi\circ\psi^{-1}(\lambda + \mu))(s_\delta + \alpha + \phi_\delta\circ v(\lambda,\mu))} \int_\delta ^\infty \dd x \,\pibar(x) \e^{-\psi^{-1}(\lambda+\mu) x}\\
		+\frac{\lambda}{ s_\delta + \alpha + \phi_\delta\circ v(\lambda,\mu)}\cdot
	\end{multline}
	From \eqref{eq: definition v}, it is clear by monotone convergence that $v(\lambda,\mu)\uparrow v(\lambda)$ as $\mu \uparrow \infty$, where
	\begin{equation*}
		v(\lambda) \coloneqq \n\left[1-\e^{-\lambda\sigma}\ind{\Delta\leqslant \delta}\right].
	\end{equation*}
	Furthermore, thanks to a Tauberian theorem, we have as $\mu \to \infty$:
	\begin{equation}\label{eq: equivalent en infty}
		\int_\delta^\infty \e^{-\psi^{-1}(\lambda+\mu)x}\pibar(x)\, \dd x \sim \frac{\pibar(\delta)\e^{-\delta\psi^{-1}(\lambda+\mu)}}{\psi^{-1}(\lambda+\mu)}\cdot
	\end{equation}
	Thus, letting $\mu \to \infty$ in \eqref{eq: v equal to bis} and using that
	$
	\psi^{-1}(x)\left(\alpha + \phi\circ \psi^{-1}(x) \right)= x
	$
	for every $x>0$, we get:
	\begin{equation}\label{eq: implicit}
		v(\lambda) = \frac{\pibar(\delta)\e^{-\delta v(\lambda)} +\lambda}{s_\delta+\alpha + \phi_\delta \circ v(\lambda)}\cdot
	\end{equation}
	
	Notice that for every $x>0$, we have:
	\begin{align*}
		s_\delta+\alpha +\phi_\delta(x) &= \alpha + \phi(x)+\int_\delta^\infty \e^{-x r}\pibar(r)\, \dd r \\
		&= \frac{1}{x}\left(\psi(x)+\int_{(\delta,\infty)} \left(1-\e^{-x r}\right)\, \pi(\dd r) - \pibar(\delta)\left(1-\e^{- x \delta}\right)\right)\\
		&= \frac{1}{x}\left(\psi_\delta(x) - \pibar(\delta)\left(1-\e^{- x \delta}\right)\right),
	\end{align*}
	where we used \eqref{eq: definition phi} and Fubini's theorem for the second equality and the definition of $\psi_\delta$ for the last. Thus \eqref{eq: implicit} becomes
	\begin{equation*}
		\psi_\delta\circ v(\lambda) = \pibar(\delta)+\lambda.
	\end{equation*}
	This yields \eqref{eq: joint distribution degree lifetime large}. Then \eqref{eq: degree 0 lifetime} follows by letting $\delta\to 0$.
\end{proof}

As a consequence, the following corollary states that the distribution of $\Delta$ under $\n$ on $(0,\infty)$ and the Lévy measure $\pi$ have the same support and the same atoms.
\begin{corollary}\label{lemma: pi diffuse} The measures $\n[\Delta\in \cdot]_{|(0,\infty)}$ and $\pi$ have the same support. Furthermore, for every $\delta>0$,  $\n[\Delta= \delta]>0$ if and only if $\delta$ is an atom of the Lévy measure $\pi$.
\end{corollary}
\begin{proof}This is clear from \eqref{eq: distribution degree strict} and \eqref{eq: distribution degree large}.
\end{proof}
\begin{remark}\label{rk: atoms of degree}
	More precisely, if $\delta>0$ is an atom of $\pi$, we have:
	\begin{equation}
		\n[\Delta=\delta] = \psi_{\delta-}^{-1}(\pi[\delta,\infty)) - \psi_\delta^{-1}(\pibar(\delta)).
	\end{equation}
	Furthermore, if $\langle \pi,1\rangle < \infty$, then we have:
	\begin{equation}
		\n[\Delta>0] = \psi_0^{-1}(\langle \pi,1\rangle)>0.
	\end{equation}
\end{remark}

\section{Degree decomposition of the Lévy tree}\label{sect: decomposition}
In this section, we give a decomposition of the Lévy tree along the large nodes. More precisely, we identify the distribution of the pruned Lévy tree obtained by removing large nodes. Furthermore, we show that the initial Lévy tree can be recovered in distribution from the pruned one by grafting Lévy forests in a Poissonian manner. We apply this decomposition to describe the structure of the discrete tree formed by large nodes.

\subsection{A Poissonian decomposition of the Lévy tree}\label{subsect: poisson decomposition}
The main result of this section is the following Poissonian decomposition along the nodes with mass larger than $\delta$. Recall from \eqref{eq: definition grafting} the definition of the grafting procedure $\circledast$.
\begin{thm}\label{thm: degree decomposition}
	The following holds:
	\begin{enumerate}[label=(\roman*)]
		\item Let $\delta \geqslant 0$ such that $\pibar(\delta)<\infty $. Under $\ndelta$, let $\left((s_i,\rho_{i}), \, i \in I\right)$ be the atoms of a Poisson point measure with intensity $\pibar(\delta)\,\dd s\allowdisplaybreaks\rfor(\dd \tilde{\rho})$, independent of $\rho$. Then, under $\ndelta$, the process $\rho \circledast_{i\in I} (s_i,\rho_{i})$ has distribution $\n$.
		\item Let $\delta>0$. Under $\operatorname{\mathbf{N}^{\psi_{\delta--}}}$, let $\left((s_i,\rho_{i}), \, i \in I\right)$ be the atoms of a Poisson point measure with intensity 
		\begin{equation*}
			\dd s \int_{[\delta,\infty)} \pi(\dd r) \for{r}(\dd \tilde{\rho}).
		\end{equation*}
		Then, under $\operatorname{\mathbf{N}^{\psi_{\delta--}}}$, the process $\rho \circledast_{i\in I} (s_i,\rho_{i})$ has distribution $\n$.
	\end{enumerate}
\end{thm}
\begin{remark}
	As mentioned in the introduction, the above theorem is a special case of the main result in \cite{abraham2010pruning} where the number of marks is finite. This greatly simplifies the proof which is why we choose include it. Observe however that the decomposition in \cite{abraham2010pruning} is proved under $\mathbb{P}$ and that an additional argument is needed to show that it still holds under the excursion measures, see the end of the proof below.
\end{remark}
\begin{proof}
	We only prove the first part, the second one being similar. Notice that the statement is trivial if $\pibar(\delta) = 0$ since in that case we have $\psi_\delta = \psi$ and the intensity of the Poisson point measure is $0$. Thus we may assume that $\pibar(\delta)\in (0,\infty)$. We shall start by proving the identity under $\mathbb{P}$ using a coupling argument. Let $X^\delta= (X_t^\delta, \, t \geqslant 0)$ be a Lévy process with Laplace exponent $\psi_\delta$ and let $e= (e_t,\, t \geqslant 0)$ be an independent Poisson point process on $\real_+$ with intensity $\ind{r>\delta}\,\pi(\dd r)$. Define the process $X = (X_t, \, t \geqslant 0)$ by:
	\begin{equation*}
		X_t = X^\delta_t + \sum_{s\leqslant t} e_s, \quad \forall t \geqslant 0.
	\end{equation*}
	Then the process $X$ is also a Lévy process with Laplace transform $\psi_\delta(\lambda) + \int_{(\delta,\infty)}(\e^{-\lambda r}-1)\, \pi(\dd r) = \psi(\lambda)$. In words, the process $X^\delta$ is obtained from $X$ by removing jumps of size larger than $\delta$.
	
	Denote by $\rho$ (resp.~$\rho^\delta$) the exploration process associated with $X$ (resp.~$X^\delta$). Let $T_\delta \coloneqq \inf\{t >0\colon\, \Delta(\rho_t)>\delta\}$ be the first time $\rho$ contains an atom with mass larger than $\delta$. It is clear from the definition that the process $\rho$ jumps exactly when $X$ does, so that $T_\delta= \inf \{t >0\colon \, \Delta X_t >\delta\}$. Therefore, we have that $X_t = X_t^\delta$ for $t<T_\delta$, which implies that $\rho_t = \rho_t^\delta$ for $t<T_\delta$.
	
	Now, from the construction of $X$, we get that $T_\delta = \inf\{t>0\colon \, e_t>\delta\}$, that is $T_\delta$ is the first time that the Poisson point process $e$ enters in $(\delta,\infty)$. Therefore the random time $T_\delta$ is exponentially distributed with mean $1/\pibar(\delta)$ and the jump $ \Delta X_{T_\delta}=e_{T_\delta}$ has distribution $\ind{r>\delta} \, \pi(\dd r)/\pibar(\delta)$ and is independent of $T_\delta$. Furthermore, the pair $(T_\delta, \Delta X_{T_\delta})$ is independent of $X^\delta$.

	Recall from \eqref{eq: ancestral line} the definition of the ancestral line of $t\in [0,\sigma]$. Let $\Delta_t = \sup_{s\preccurlyeq t}\Delta X_s=\sup_{s\preccurlyeq t} \Delta(\rho_s)$ be the maximal degree of the ancestral line of $t$. For every $t \geqslant 0$, let
	\begin{equation}
		A(t) \coloneqq \int_0^t \ind{\Delta_s \leqslant \delta}\, \dd s
	\end{equation} be the Lebesgue measure of the set of individuals prior to $t$ whose lineage does not contain any node with mass larger than $\delta$. Let $C_t \coloneqq \inf\{s \geqslant 0\colon\, A_s >t\}$ be the right-continuous inverse of $A$ and define the pruned exploration process $\tilde{\rho}=(\tilde{\rho}_t = \rho_{C_t}, \, t\geqslant 0)$. In other words, we remove from the tree all the individuals above a node with mass larger than $\delta$ and the pruned exploration process $\tilde{\rho}$ codes the remaining tree. 
	
	Next, let us consider excursions of $\rho$ above nodes of mass larger than $\delta$. Let $T_\delta^{(1)} = T_\delta$ be the first time $\rho$ contains an atom with mass larger than $\delta$ and $L_\delta^{(1)} =L_\delta = \inf\{t >T_\delta\colon\, H_t < H_{T_\delta}\}$ be the first time that atom is erased. Define recursively the stopping times $T_\delta^{(k)}= \inf\{t>L_\delta^{(k-1)}\colon\, \Delta(\rho_t)>\delta\}$ the $k$-th time $\rho$ contains a (first-generation) node with mass larger than $\delta$ and $L_\delta^{(k)} = \inf\{t>T_\delta^{(k)}\colon\, H_t <H_{T_\delta^{(k)}}\}$ the first time that node is erased. Finally, let $\rho^{(k)}$ be the path of the exploration process above level $H_{T^{(k)}_\delta}$ between times $T_\delta^{(k)}$ and $L_\delta^{(k)}$, defined by:
	\begin{equation*}
		\rho^{(k)}_t = \abov{H_{T^{(k)}_\delta}}(\rho_{t+T^{(k)}_\delta}), \quad\forall 0\leqslant t \leqslant L^{(k)}_\delta-T^{(k)}_\delta.
	\end{equation*}
	Notice that by construction, we have:
	\begin{equation*}
		\rho = \tilde{\rho} \circledast_{k=1}^\infty (A(T_\delta^{(k)}),\rho^{(k)}).
	\end{equation*}
	
	Using the strong Markov property under $\mathbb{P}$ at time $T_\delta$ and Lemma \ref{lemma: exploration above initial level}, we get that, conditionally on $\Delta(\rho_{T_\delta})$ (which is equal to $\Delta X_{T_\delta}$), the process $\rho^{(1)}$ has distribution $\for{\Delta X_{T_\delta}}$. 
	
	But the random time $T_\delta$ is exponentially distributed with mean $1/\pibar(\delta)$, the jump $\Delta X_{T_\delta}$ has distribution $\mathbf{1}_{(\delta,\infty)}(r)\, \pi(\dd r)/\pibar(\delta)$ and they are independent. We deduce that $\rho^{(1)}$ is independent of $T_\delta$ and has distribution $\rfor$. Furthermore, $(T_\delta,\Delta X_{T_\delta})$ is generated by the Poisson point process $e$ while $\tilde{\rho}$ is generated by $X^\delta$. These being independent, we deduce that $\tilde{\rho}$ is independent of $(T_\delta, \Delta X_{T_\delta})$, and thus of $(A(T_\delta^{(1)})=T_\delta, \rho^{(1)})$. Iterating this argument and using the strong Markov property, we get that the random measure
	\begin{equation*}
		\sum_{k=1}^\infty \delta_{(A(T_\delta^{(k)}),\rho^{(k)})}
	\end{equation*}
	is a Poisson point measure with intensity $\pibar(\delta)\,\dd s \rfor(\dd \rho)$ and is independent of $\tilde{\rho}$. 
	
	It remains to show that $\tilde{\rho}$ is distributed as $\rho^\delta$. Recall that $\tilde{\rho}_t = \rho_{C_t}$. From this, it is clear that the two processes are equal to $\rho$ before time $T_\delta$. Furthermore, at time $T_\delta$ we have $\tilde{\rho}_{T_\delta} = \rho_{T_\delta-} =\rho_{L_\delta} = \rho^\delta_{T_\delta}$. Now applying the strong Markov property to $\rho$ at $L_\delta$ gives that, conditionally on $\tilde{\rho}_{T_\delta}$, the process $(\rho_{t+L_\delta}, \, t \geqslant 0)$ has distribution $\mathbb{P}_{\tilde{\rho}_{T_\delta}}$. As a consequence, conditionally on $\tilde{\rho}_{T_\delta}$, the process $(\tilde{\rho_t} = \rho_{t+L_\delta-T_\delta}, \, A(T_\delta^{(1)})\leqslant t < A(T_\delta^{(2)}) )$ is distributed as $(\rho^\delta_t, \, S_1\leqslant t < S_2)$, where $0\leqslant S_1\leqslant S_2 \leqslant \ldots$ are the ordered atoms of a Poisson point process on $\real_+$ with intensity $\pibar(\delta)\, \dd s$, independent of $\rho^\delta$. Iterating this argument, we deduce that $\tilde{\rho}$ and $\rho^\delta$ have the same distribution. This proves the Poisson decomposition under $\mathbb{P}$. Therefore, the same decomposition holds under the excursion measures up to a normalizing constant: there exists a constant $c>0$ such that, under $\ndelta$, the process $\rho \circledast_{i\in I} (s_i, \rho_{i})$ has distribution $c\n$, where the random measure $\sum_{i\in I} \delta_{(s_i, \rho_{i})}$ is under $\ndelta$ a Poisson point measure with intensity $\pibar(\delta) \, \dd s\rfor(\dd \tilde{\rho})$. Let $\zeta = \operatorname{Card}\{i \in I\colon \, s_i < \sigma\}$. Then, under $\ndelta$ and conditionally on $\rho$, the random variable $\zeta$ has Poisson distribution with parameter $\pibar(\delta) \sigma$. It follows that
	\begin{equation*}
		\ndelta\left[\zeta\geqslant 1\right] = \ndelta\left[\ndelta\left[\zeta\geqslant 1\middle|\rho\right]\right] = \ndelta\left[1-\e^{-\pibar(\delta)\sigma}\right] = \psi_\delta^{-1}(\pibar(\delta))= \n[\Delta>\delta],
	\end{equation*}
	where in the last equality we used Proposition \ref{prop: joint distribution degree lifetime}. This gives $c=1$ and the result readily follows.
\end{proof}

The following corollary is an immediate consequence of the Poissonian decomposition from Theorem \ref{thm: degree decomposition}.
\begin{corollary}\label{cor: identity in distribution before big jump}
	Let $\delta>0$ and let $F\in \B(\D)$. We have:
	\begin{align}
		\n\left[F(\rho)\ind{\Delta\leqslant \delta}\right] &= \ndelta\left[F(\rho)\e^{-\pibar(\delta) \sigma}\right],\\
		\n\left[F(\rho)\ind{\Delta< \delta}\right] &= \operatorname{\mathbf{N}}^{\psi_{\delta-}}\left[F(\rho)\e^{-\pi[\delta,\infty) \sigma}\right]
	\end{align}
	Furthermore, if $\langle \pi,1\rangle < \infty$, then we have:
	\begin{equation}
		\n\left[F(\rho) \ind{\Delta=0}\right] = \operatorname{\mathbf{N}}^{\psi_{0}}\left[F(\rho)\e^{-\langle\pi,1\rangle \sigma}\right].
	\end{equation}
\end{corollary}

The Poissonian decomposition of Theorem \ref{thm: degree decomposition} also holds for forests.
\begin{proposition}\label{prop: degree decomposition forests}
	Let $\delta >0$ such that $\pibar(\delta)<\infty$ and let $r>0$. Under $\fordelta{r}$ (resp.~$\rfordelta$), let $\left((s_i,\rho_{i}), \, i \in I\right)$ be the atoms of a Poisson point measure with intensity $\pibar(\delta)\,\dd s\rfor(\dd \tilde{\rho})$. Then, under $\fordelta{r}$ (resp.~$\rfordelta$), the process $\rho \circledast_{i\in I} (s_i,{\rho}_{i})$ has distribution $\for{r}$ (resp.~$\rfor$).
\end{proposition}

\subsection{Structure of nodes with mass larger than $\delta$}\label{subsect: structure of big nodes}
Here, we give a description of the structure of nodes with mass larger than $\delta$ under $\n$. Let us start by determining the distribution of the height of MRCA (see Section \ref{subsect: exploration process} for the definition) of the set of nodes with mass larger than $\delta$.
\begin{proposition}
	Under $\n$, conditionally on $\Delta >\delta$, the height of the MRCA of the set of nodes with mass larger than $\delta$ is exponentially distributed with mean $1/\psi_\delta'(\n[\Delta>\delta])$.
\end{proposition}
Notice that, as $\delta\to \infty$, $\psi_\delta'(\n[\Delta>\delta])$ converges to $\alpha$ which is positive in the subcritical case and $0$ in the critical case (this implies that the height of the MRCA goes to infinity).
\begin{proof}
	Under $\ndelta$, denote by $\tau_1\leqslant \tau_2 \leqslant \ldots$ the jump times of a standard Poisson process with intensity $\pibar(\delta)$. Denote by $M = \sup\{i\geqslant 1 \colon \, \tau_i \leqslant \sigma\}$ the number of marks which arrive during the lifetime $\sigma$ and set:
	\begin{equation*}
		J = 
		\begin{cases}
			\inf\{H_s \colon \, \tau_1\leqslant s \leqslant \tau_M\} & \text{if } M \geqslant 1, \\
			\infty & \text{if } M = 0.
		\end{cases}
	\end{equation*}
	It is clear from Theorem \ref{thm: degree decomposition} that, under $\n$, the height of the MRCA of the set of nodes with mass larger than $\delta$ is distributed as $J$ under $\ndelta$, with the convention that this height is equal to $\infty$ if there are no such nodes. Thus, we need to determine the distribution of $J$ under $\ndelta$ and conditionally on $M\geqslant 1$.
	
	Notice that, on the event $\{M\geqslant 2\}$, $J$ agrees with the random variable $K$ defined in \cite[p.96]{duquesne2002random}. Proposition 3.2.3 therein gives:
	\begin{equation}\label{eq: distribution J}
		\ndelta\left[f(J)\ind{M\geqslant 2}\middle| M\geqslant 1\right] = \left(\psi_\delta'(\n[\Delta>\delta])-\frac{\pibar(\delta)}{\n[\Delta>\delta]}\right) \int_0^\infty f(a) \e^{-a \psi_\delta'(\n[\Delta>\delta])} \, \dd a,
	\end{equation}
	where we used that $\psi_\delta(\n[\Delta>\delta])= \pibar(\delta)$ by \eqref{eq: distribution degree strict}.
	
	Next, notice that under $\ndelta$, conditionally on $\rho$, $M$ has Poisson distribution with parameter $\pibar(\delta)\sigma$. Furthermore, conditionally on $\rho$ and on $M= 1$, $\tau_1$ is uniformly distributed on $[0,\sigma]$. Thus, by conditioning on $\rho$, we get:
	\begin{align*}
		\ndelta\left[f(J) \ind{M=1}\right] &= \pibar(\delta)\ndelta\left[\int_0^\sigma f(H_t) \e^{-\pibar(\delta)\sigma}\, \dd t\right]\\
		&= \pibar(\delta)\ndelta\left[\int_0^\sigma f(H_t) \e^{-\pibar(\delta)t} \e^{-\pibar(\delta)(\sigma - t)}\, \dd t\right]\\
		&= \pibar(\delta)\ndelta\left[\int_0^\sigma f(H_t) \e^{-\pibar(\delta)t} \operatorname{\mathbb{E}}_{\rho_t}^{\psi_\delta,\ast}\left[\e^{-\pibar(\delta)\sigma }\right]\, \dd t\right],
	\end{align*}
	where we used the Markov property of the exploration process under $\ndelta$ for the last equality. Thanks to Lemma \ref{lemma: poisson decomposition exploration}, for every $\nu \in \Mf(\real_+)$ we have:
	\begin{equation*}
		\operatorname{\mathbb{E}}_{\nu}^{\psi_\delta,\ast}\left[\e^{-\pibar(\delta)\sigma }\right] = \e^{-\psi_\delta^{-1}(\pibar(\delta)) \langle \nu, 1\rangle} = \e^{-\n[\Delta>\delta]\langle \nu,1\rangle},
	\end{equation*}
	where we used \eqref{eq: distribution degree strict} for the last equality.
	
	Therefore, we get:
	\begin{align*}
		\ndelta\left[f(J) \ind{M=1}\right] &= \pibar(\delta)\ndelta\left[\int_0^\sigma f(H_t) \e^{-\pibar(\delta)t} \e^{-\langle \rho_t, 1\rangle \n[\Delta>\delta]}\, \dd t\right]\\
		&= \pibar(\delta)\ndelta\left[\int_0^\sigma f(H_t) \e^{-\pibar(\delta)(\sigma-t)} \e^{-\langle \eta_t, 1\rangle \n[\Delta>\delta]}\, \dd t\right] \\
		&=\pibar(\delta)\ndelta\left[\int_0^\sigma f(H_t) \e^{-\langle \rho_t+\eta_t, 1\rangle \n[\Delta>\delta]}\, \dd t\right],
	\end{align*}
	where we used the time-reversal property of the exploration process for the second equality and the Markov property for the last. By \cite[Proposition 3.1.3]{duquesne2002random}, we deduce that
	\begin{equation*}
		\ndelta\left[f(J) \ind{M=1}\right]  = \pibar(\delta) \int_0^\infty f(a) \e^{-\psi_\delta'(\n[\Delta>\delta])a}\, \dd a.
	\end{equation*}
	Thanks to Theorem \ref{thm: degree decomposition}, it is clear that $\ndelta[M\geqslant 1] = \n[\Delta>\delta]$. It follows that
	\begin{equation*}
		\ndelta\left[f(J)\ind{M=1}\middle| M\geqslant 1\right] = \frac{\pibar(\delta)}{\n[\Delta>\delta]} \int_0^\infty f(a) \e^{-\psi_\delta'(\n[\Delta>\delta])a}\, \dd a.
	\end{equation*}
	In conjunction with \eqref{eq: distribution J}, this yields:
	\begin{equation*}
		\ndelta\left[f(J)\middle| M\geqslant 1\right] = \psi_\delta'(\n[\Delta>\delta])\int_0^\infty f(a) \e^{-\psi_\delta'(\n[\Delta>\delta])a}\, \dd a.
	\end{equation*}
	This shows that, under $\ndelta$ and conditionally on $M\geqslant 1$, $J$ is exponentially distributed with mean $1/\psi_\delta'(\n[\Delta>\delta])$ and the proof is now complete.
\end{proof}

Let $\rddtree_\delta$ be the (random) discrete forest spanned by nodes with mass larger than $\delta$. More explicitly, $\rddtree_\delta$ starts with $\init$ individuals, where $\init$ is the number of first-generation nodes of $\rho$ with mass larger than $\delta$ (that is nodes of $\rho$ with mass larger than $\delta$ having no ancestors with mass larger than $\delta$). Then, each node $v$ of $\rddtree_\delta$ gets $\off_v$ children, where $\off_v$ is the number of first-generation descendants with mass larger than $\delta$ of the corresponding node in $\rho$. Finally, denote by $\total$ the total population of $\rddtree_\delta$ or equivalently the total number of nodes of $\rho$ with mass larger than $\delta$. We shall identify the distribution of $\rddtree_\delta$. Given two $\mathbb{N}$-valued random variables $Z_0$ and $\xi$, we call a $(Z_0,\xi)$-Bienaymé-Galton-Watson forest a collection of $Z_0$ independent Bienaymé-Galton-Watson trees with offspring distribution (the law of) $\xi$. 

Under $\ndelta$ (resp.~under $\rfordelta$), let $\sum_{i\in I} \delta_{(s_i,\rho^{i})}$ be a Poisson point measure with intensity $\pibar(\delta)\, \dd s \allowbreak\rfor(\dd \tilde{\rho})$ independent of $\rho$ and let 
\begin{equation}
	\zeta = \operatorname{Card}\{i \in I\colon\, s_i < \sigma\}
\end{equation}
be the number of points arriving during the lifetime $\sigma$. Basic properties of Poisson point measures imply that, under $\ndelta$ (resp.~under $\rfordelta$) and conditionally on $\rho$, the random variable $\zeta$ has Poisson distribution with parameter $\pibar(\delta)\sigma$.

\begin{proposition}\label{prop: galton watson}
	Let $\delta>0$ such that $\pibar(\delta)>0$. Under $\n$, the random forest $\rddtree_\delta$ is a $(\init,\off)$-Bienaymé-Galton-Watson forest, where $\init$ is distributed as $\zeta$ under $\ndelta$ and $\off$ is distributed as $\zeta$ under $\rfordelta$. Their Laplace transforms are given by, for every $\lambda >0$:
	\begin{align}
		\n\left[1-\e^{-\lambda \init}\right]  &= \psi_\delta^{-1}\left((1-\e^{-\lambda})\pibar(\delta)\right), \label{eq: initial distribution}\\
		\n\left[\e^{-\lambda \off}\right]  &= \frac{1}{\pibar(\delta)}\int_{(\delta,\infty)} \e^{-r\psi_\delta^{-1}\left((1-\e^{-\lambda})\pibar(\delta)\right)}\, \pi(\dd r).\label{eq: offspring distribution}
	\end{align}
\end{proposition}
\begin{proof}
	That $\rddtree_\delta$ is under $\n$ a Bienaymé-Galton-Watson forest with the mentioned distribution is an immediate consequence of the Poissonian decompositions given in Thereom \ref{thm: degree decomposition} and Proposition \ref{prop: degree decomposition forests}. Let us compute the Laplace transforms.
	
	Recall that, under $\ndelta$ and conditionally on $\rho$, $\zeta$ has Poisson distribution with parameter $\pibar(\delta) \sigma$. Using this, we have:
	\begin{equation}
		\n\left[1-\e^{-\lambda \init}\right]  = \ndelta\left[1-\e^{-\lambda \zeta}\right]
		= \ndelta\left[1-\e^{-(1-\e^{-\lambda})\pibar(\delta)\sigma}\right]
		= \psi_\delta^{-1}\left((1-\e^{-\lambda})\pibar(\delta)\right). \label{eq: initial distribution proof}
	\end{equation}
	This proves \eqref{eq: initial distribution}. Similarly, since under $\rfordelta$ and conditionally on $\rho$, $\zeta$ has Poisson distribution with parameter $\pibar(\delta)\sigma$, a similar computation yields:
	\begin{equation*}
		\n\left[\e^{-\lambda \off}\right]  = \rfordelta(\e^{-\lambda \zeta}) = \rfordelta\left(\e^{- (1-\e^{-\lambda})\pibar(\delta)\sigma}\right) = \frac{1}{\pibar(\delta)} \int_{(\delta,\infty)} \pi(\dd r) \fordelta{r}\left(\e^{-(1-\e^{-\lambda})\pibar(\delta)\sigma}\right).
	\end{equation*}
	But, using the Poisson decomposition of Lemma \ref{lemma: poisson decomposition exploration}, we get that:
	\begin{equation}\label{eq: laplace transform of sigma for a forest}
		\fordelta{r}(\e^{-x\sigma}) = \expp{-r \ndelta\left[1-\e^{-x \sigma}\right]} = \e^{-r \psi_\delta^{-1}(x)}, \quad \forall x \geqslant 0,
	\end{equation}
	and \eqref{eq: offspring distribution} readily follows.
\end{proof}

We end this section with the following result on the criticality of the random forest $\rddtree_\delta$.
\begin{proposition}
	Let $\delta>0$ such that $\pibar(\delta)>0$. The mean of $\off$ is given by:
	\begin{equation}\label{eq: mean offspring}
		\n[\off]= \frac{\int_{(\delta,\infty)} r\, \pi(\dd r)}{\alpha + \int_{(\delta,\infty)}r \, \pi(\dd r)}\cdot
	\end{equation}
	In particular, under $\n$, the Bienaymé-Galton-Watson forest $\rddtree_\delta$ is critical (resp.~subcritical) if $\psi$ is critical (resp.~subcritical).
\end{proposition}
\begin{proof}
	Thanks to Proposition \ref{prop: galton watson}, we have:
	\begin{equation*}
		\n[\off]= \rfordelta(\zeta) 
		= \pibar(\delta)\rfordelta(\sigma)
		= \int_{(\delta,\infty)} \pi(\dd r)\fordelta{r}(\sigma).
	\end{equation*}
	But the Poissonian decomposition of $\fordelta{r}$ gives:
	\begin{equation*}
		\fordelta{r}(\sigma) = r\ndelta[\sigma]= \frac{r}{\alpha+\int_{(\delta,\infty)} z\, \pi(\dd z)},
	\end{equation*}
	where we used \eqref{eq: expectation sigma} for the second equality. This yields \eqref{eq: mean offspring}.
\end{proof}

\section{Conditioning on $\Delta = \delta$}\label{sect: conditioning on degree}
The goal of this section is to make sense of the conditional measure $\n[\cdot|\Delta=\delta]$. For every $\delta>0$, we set:
\begin{equation}\label{eq: definition w}
	\w(\delta) = \n[\sigma \ind{\Delta< \delta}] \quad \text{and} \quad \wplus(\delta) = \n[\sigma \ind{\Delta\leqslant \delta}].
\end{equation}
Notice that if $\delta>0$ is not an atom of the Lévy measure $\pi$, then we have $\w(\delta)= \wplus(\delta)$ by Lemma \ref{lemma: pi diffuse}. Furthermore, thanks to Corollary \ref{cor: identity in distribution before big jump}, \eqref{eq: laplace transform sigma} and \eqref{eq: distribution degree large}, we have:
\begin{equation}\label{eq: w general}
	w(\delta) = \operatorname{\mathbf{N}}^{\psi_{\delta-}}\left[\sigma \e^{-\pi[\delta,\infty)\sigma}\right] = \frac{1}{\psi_{\delta-}'\circ \psi_{\delta-}^{-1}(\pi[\delta,\infty))} = \frac{1}{\psi_{\delta-}'(\n[\Delta\geqslant \delta])}\cdot
\end{equation}
Similarly, we have:
\begin{equation}\label{eq: w plus general}
	\wplus(\delta) = \ndelta\left[\sigma \e^{-\pibar(\delta)\sigma}\right] = \frac{1}{\psi_{\delta}'\circ \psi_\delta^{-1}(\pibar(\delta))} = \frac{1}{\psi_\delta'(\n[\Delta>\delta])}\cdot
\end{equation}

For $\delta >0$, denote by $\tilted{\delta}$ the probability measure on the space $ \real_+\times \D$ defined by:
\begin{equation}\label{eq: definition Pdelta}
	\int_{\real_+\times \D} F\, \dd \!\tilted{\delta}= \frac{1}{w(\delta)}\n\left[\int_0^\sigma F(s,\rho)\, \dd s \, \ind{\Delta< \delta}\right],
\end{equation}
for every $F\in  \B(\real_+\times \D)$. Similarly, we set:
\begin{equation}\label{eq: definition Pdelta plus}
	\int_{\real_+\times \D} F\, \dd \!\tilted{\delta+}= \frac{1}{\wplus(\delta)}\n\left[\int_0^\sigma F(s,\rho)\, \dd s \, \ind{\Delta\leqslant \delta}\right].
\end{equation}
Observe that $\tilted{\delta+} = \lim_{\epsilon\to 0+} \tilted{\delta+\epsilon}$ in the sense of weak convergence of measures.

For every $\delta,\epsilon >0$, let 
\begin{equation}
	E_{\delta,\epsilon} = \{\delta-\epsilon < \Delta < \delta+\epsilon, \, Z_0^{\delta-\epsilon} = 1\}
\end{equation}
be the event that the maximal degree is between $\delta-\epsilon$ and $\delta+\epsilon$ and there is a unique first-generation node with mass larger than $\delta-\epsilon$. The next lemma states that, under the assumption that $\delta$ is not an atom of the Lévy measure $\pi$, the two events $E_{\delta,\epsilon}$ and $\{\delta-\epsilon < \Delta < \delta+\epsilon\}$ are equivalent in $\n$-measure as $\epsilon \to 0$. Recall that $\pi$ is a measure on $(0,\infty)$ and as such, its support $\supp{\pi}$ does not contain $0$.
\begin{lemma}\label{lemma: equivalent conditionings}
	Assume that $\delta \in \supp{\pi}$ is not an atom of the Lévy measure $\pi$ and that $\pibar(\delta)>0$. We have $\n[\delta-\epsilon< \Delta< \delta+\epsilon] \sim \n[E_{\delta,\epsilon}]$ as $\epsilon \to 0$.
\end{lemma}
\begin{proof} We start by observing that, thanks to the Poissonian decomposition of $\for{r}$ given in Lemma \ref{lemma: poisson decomposition exploration}, we have:
	\begin{equation}
		\for{r}(\Delta< \delta) = \begin{cases}
			0 & \text{if } r \leqslant \delta,\\
			\e^{-r \n[\Delta\geqslant \delta]} & \text{if } r >\delta.
		\end{cases}
	\end{equation}
	Similarly, we have:
	\begin{equation}\label{eq: distribution degree forest}
		\for{r}(\Delta\leqslant \delta) = \begin{cases}
			0 & \text{if } r < \delta, \\
			\e^{-r \n[\Delta> \delta]} & \text{if } r \geqslant\delta.
		\end{cases}
	\end{equation}
	
	We deduce that
	\begin{align}
		\rfordinf(\Delta< \dsup) &= \frac{1}{\pibar(\dinf)}\int_{(\dinf,\dsup)} \pi(\dd r) \for{r}(\Delta< \dsup) \notag\\
		&= \frac{1}{\pibar(\dinf)}\int_{(\dinf,\dsup)} \e^{-r \n[\Delta\geqslant\dsup]}\,\pi(\dd r). \label{eq: proba forest has degree delta + epsilon}
	\end{align}
	Since $\pi(\delta) = 0$ and $\pibar(\delta)>0$, this implies that
	\begin{equation}\label{eq: probability a forest with initial size delta - epsilon to have degree delta + epsilon}
		\lim_{\epsilon\to 0}\rfordinf(\Delta< \dsup) =0.
	\end{equation} 
	
	Under $\ndinf$ and conditionally on $\rho$, let $\zeta$ be a Poisson random variable with parameter $\pibar(\dinf)\sigma$ and let $\left((s_i, \rho_i), \, i \geqslant 1\right)$ be independent with distribution $\sigma^{-1}\mathbf{1}_{[0,\sigma]}(s)\,\dd s\rfordinf(\dd \tilde{\rho})$, independent of $\zeta$. Thanks to Theorem \ref{thm: degree decomposition}, we have:
	\begin{align}\label{eq: E delta epsilon}
		\n[E_{\delta,\epsilon}] &= \ndinf[\zeta=1, \, \Delta(\rho_{1})< \dsup]\nonumber\\
		&=\ndinf[\pibar(\dinf)\sigma\e^{-\pibar(\dinf)\sigma}]  \rfordinf(\Delta< \dsup)\nonumber\\
		&=\pibar(\delta-\epsilon)\wplus(\delta-\epsilon)\rfordinf(\Delta< \dsup),
	\end{align}
	where we used \eqref{eq: w plus general} for the last equality.
	Similarly, we have:
	\begin{align*}
		\n[\delta-\epsilon<\Delta< \delta+\epsilon] &= \ndinf[\zeta \geqslant 1; \, \forall i \leqslant \zeta, \, \Delta(\rho_i)< \dsup] \\
		&= \ndinf[\zeta\geqslant 1; \rfordinf(\Delta< \dsup)^\zeta]\\
		&= \ndinf\left[\e^{-\pibar(\dinf)\sigma} \left(\e^{\pibar(\dinf)\sigma\rfordinf(\Delta< \dsup)}-1 \right)\right].
	\end{align*}
	
	Therefore, using the inequality $\e^x - 1 - x \leqslant x^2 \e^{x}/2$ and the fact that the function $x \mapsto x \e^{-x}$ is bounded on $\real_+$ by some constant $C>0$, we deduce that
	\begin{multline*}
		0\leqslant \n[\delta-\epsilon<\Delta< \delta+\epsilon]-\n[E_{\delta,\epsilon}]\\
		\begin{aligned}[t]
			&\leqslant \frac{1}{2}\pibar(\dinf)^2\ndinf\left[\sigma^2\e^{-\pibar(\dinf)\sigma \rfordinf(\Delta\geqslant \delta +\epsilon)}\right]\rfordinf(\Delta< \dsup)^2 \\
			&\leqslant \frac{C}{2} \pibar(\dinf)\ndinf[\sigma] \frac{\rfordinf(\Delta< \dsup)^2}{\rfordinf(\Delta\geqslant \delta +\epsilon)}\\
			&= \frac{C\pibar(\dinf)}{2(\alpha+\int_{(\dinf,\infty)}r\, \pi(\dd r))} \frac{\rfordinf(\Delta< \dsup)^2}{\rfordinf(\Delta\geqslant \delta +\epsilon)}\\
			&\leqslant C_\delta \frac{\rfordinf(\Delta< \dsup)^2}{\rfordinf(\Delta\geqslant \delta +\epsilon)}
		\end{aligned}
	\end{multline*}
	for $\epsilon >0$ small enough and some constant $C_\delta$ which is independent of $\epsilon$, where we used \eqref{eq: expectation sigma} for the equality.

	Furthermore, it is clear from \eqref{eq: w plus general} that
	\begin{equation*}
		\pibar(\dinf)\wplus(\delta-\epsilon)= \pibar(\delta-\epsilon) \n[\sigma \ind{\Delta\leqslant \delta-\epsilon}] \geqslant \pibar(\delta/2) \n[\sigma \ind{\Delta\leqslant \delta/2}],
	\end{equation*}
	for $\epsilon >0$ small enough. In particular, it follows from \eqref{eq: E delta epsilon} that there exists a constant $C_\delta' >0$ such that
	\begin{equation*}
		0\leqslant \frac{\n[\delta-\epsilon<\Delta< \delta+\epsilon]-\n[E_{\delta,\epsilon}]}{\n[E_{\delta,\epsilon}]}
		\leqslant C_\delta'\frac{\rfordinf(\Delta< \dsup)}{\rfordinf(\Delta\geqslant \delta +\epsilon)},
	\end{equation*}
	where the right-hand side goes to $0$ as $\epsilon \to 0$ thanks to \eqref{eq: probability a forest with initial size delta - epsilon to have degree delta + epsilon}. This concludes the proof.
\end{proof}
As a consequence, since $E_{\delta,\epsilon}\subset \{\delta-\epsilon< \Delta< \delta+\epsilon\}$, conditioning on either event is equivalent as $\epsilon \to 0$. We choose to work with the former as computations will be simpler. We shall next give a description of the exploration process conditioned on $E_{\delta,\epsilon}$. 

Let 
\begin{equation}\label{eq: definition Tdelta}
	T_\delta  = \inf\{t >0\colon\, \Delta(\rho_t)>\delta\}
\end{equation}
be the first time that the exploration process contains an atom with mass larger than $\delta$ and let
\begin{equation}
	L_\delta=\inf\{t>T_\delta\colon\, H(\rho_t)<H(\rho_{T_\delta}) \}
\end{equation}
be the first time that node is erased. We split the path of the exploration process into two parts: ${\rho}^{\delta,-}$ is the pruned exploration process (that is the exploration process minus the first node with mass larger than $\delta$):
\begin{equation}\label{eq: definition bottom}
	{\rho}^{\delta,-}_t = \begin{cases}
		\rho_t & \text{if } t < T_{\delta}, \\
		\rho_{t-T_{\delta}+L_{\delta}} & \text{if } t\geqslant T_{\delta},
	\end{cases}
\end{equation}
and ${\rho}^{\delta,+}$ is the path of the exploration process above the unique first-generation node with mass larger than $\delta$:
\begin{equation}\label{eq: definition upper}
	{\rho}^{\delta,+}_t = \abov{H_{T_{\delta}}}(\rho_{(t+T_{\delta})\wedge L_{\delta}}), \quad \forall t \geqslant 0.
\end{equation}
Notice that ${\rho}^{\delta,+}_0$ is a multiple of the Dirac measure at $0$.
\begin{lemma}\label{lemma: decomposition degree delta}
	Let $F,G\in \B(\real_+\times \D )$. For every $\delta,\epsilon>0$ such that $\pibar(\delta-\epsilon)>0$, we have:
	\begin{equation}\label{eq: decomposition degree delta}
		\n\left[F(T_{\delta-\epsilon},{\rho}^{\delta-\epsilon,-})G({\rho}^{\delta-\epsilon,+})\middle|E_{\delta,\epsilon}\right] 
		=\int_{\real_+\times\D} F\, \dd\!\tilted{(\delta--\epsilon)+}\times\rfordinf(G(\rho)|\Delta< \delta+\epsilon).
	\end{equation}
\end{lemma}
\begin{proof}
	By Theorem \ref{thm: degree decomposition}, we have
	\begin{equation*}
		\n\left[F(T_{\delta-\epsilon},{\rho}^{\delta-\epsilon,-})G({\rho}^{\delta-\epsilon,+})\mathbf{1}_{E_{\delta,\epsilon}}\right] = \ndinf\left[F(U,\rho) G(\rho^{\dinf})\ind{\zeta=1, \Delta(\rho^{\dinf})< \delta+\epsilon}\right],
	\end{equation*}
	where, under $\ndinf$, conditionally on $\rho$, $U$ is uniformly distributed on $[0,\sigma]$, $\zeta$ is a Poisson random variable with parameter $\pibar(\dinf)\sigma$, $\rho^{\dinf}$ has distribution $\rfordinf$ and they are independent. We deduce that
	\begin{multline*}
		\n\left[F(T_{\delta-\epsilon},{\rho}^{\delta-\epsilon,-})G({\rho}^{\delta-\epsilon,+})\mathbf{1}_{E_{\delta,\epsilon}}\right] \\
		= \ndinf\left[\pibar(\dinf)\e^{-\pibar(\dinf)\sigma}\int_0^\sigma F(s,\rho)\, \dd s\right]\rfordinf(G(\rho)\ind{\Delta< \delta+\epsilon}).
	\end{multline*}
	Together with Corollary \ref{cor: identity in distribution before big jump}, \eqref{eq: definition Pdelta} and \eqref{eq: E delta epsilon}, this yields the desired result.
\end{proof}
We now turn to the study of the asymptotic behavior of the measures appearing in the right-hand side of \eqref{eq: decomposition degree delta}. Recall that the total variation distance of two probability measures $P,Q$ on some measurable space $(E,\mathcal{E})$ is given by:
\begin{equation*}
	d_{\mathrm{TV}}(P,Q) = \sup\{|P(A) -Q(A)|\colon\, A\in \mathcal{E}\}.
\end{equation*}
\begin{lemma}\label{lemma: weak convergence of bottom part delta}
	Assume that $\delta>0$ is not an atom of the Lévy measure $\pi$. Then, the mapping $r\mapsto\tilted{\mathnormal{r}+}$ is continuous at $\delta$ in total variation distance and $\tilted{\delta+} = \tilted{\delta}$.
\end{lemma}
\begin{proof}
	Thanks to Corollary \ref{lemma: pi diffuse}, we have $\n[\Delta=\delta] = 0$. Then the result readily follows from the definition of the measure $\tilted{\mathnormal{r}+}$.
\end{proof}

Recall that the space $\Mf(\real_+)$ is equipped with the topology of weak convergence which makes it a Polish space, see \cite[Section 8.3]{bogachev2007measure}. It can be metrized by the so-called bounded Lipschitz distance defined for every $\mu, \nu \in \Mf(\real_+)$ by $d_{\mathrm{BL}}(\mu,\nu) = \sup |\langle \mu , f \rangle - \langle \nu,f \rangle|$, where the supremum is taken over all Lipschitz-continuous and bounded functions $f\colon \real_+\to \real$ such that
\begin{equation*}
	\sup_{x\geqslant 0} |f(x)|+ \sup_{x\neq y} \frac{|f(x) - f(y)|}{|x-y|}\leqslant 1.
\end{equation*}
Recall that $\D$ is the space of càdlàg $\Mf(\real_+)$-valued functions defined on $\real_+$, equipped with the Skorokhod $J_1$-topology and let $\dS$ be the Skorokhod distance associated with the distance $d_{\mathrm{BL}}$ on $\Mf(\real_+)$. Denote by $\DD$ the subset of $\D$ consisting of excursions:
\begin{equation}
	\DD \coloneqq \left\{\mu \in \D\colon \, \sigma(\mu) <\infty, \ \mu_t \neq 0, \, \forall 0 < t < \sigma(\mu) \text{ and } \mu_{\sigma(\mu)-} = 0 \text{ if } \sigma(\mu) >0\right\},
\end{equation}
where $\sigma(\mu) = \inf\{t >0 \colon \, \mu(t+\cdot) \equiv 0\}$ is the lifetime of $\mu$. Notice that if $\mu \in \DD$ such that $\sigma(\mu) = 0$ then necessarily $\mu \equiv 0$. Observe that the mapping $\mu \mapsto \sigma(\mu)$ is measurable with respect to the Skorokhod topology since $\sigma(\mu) = \inf\{t \in \mathbb{Q}\cap(0,\infty)\colon\, \mu_t = 0\}$ and $\mu \mapsto \mu_t$ is measurable. We equip $\DD$ with the following distance:
\begin{equation*}
	\dK(\mu, \nu) = \dS(\mu,\nu) + |\sigma(\mu) - \sigma(\nu)|.
\end{equation*}
\begin{lemma}\label{lemma: continuity}
	Let $\nu \in \D$ and $s >0$. The mapping $\mu \mapsto \nu \circledast (s,\mu)$ is continuous from $(\DD, \dK)$ to $(\D, \dS)$.
\end{lemma}
\begin{proof}
	Denote by $\Lambda$ the set of all continuous functions $\lambda \colon \real_+ \to \real_+$ that are (strictly) increasing, with $\lambda(0) = 0$ and $\lim_{t\to \infty} \lambda(t) = \infty$. Let $\mu_n$ be a sequence in $\DD$ converging to $\mu$ with respect to the distance $\dK$. By definition of the Skorokhod topology (see e.g.~Jacod and Shiryaev \cite[Chapter VI]{jacod2003limit}), this means that there exists a sequence $\lambda_n \in \Lambda$ of time changes such that
	\begin{equation*}
		\lim_{n\to \infty}|\sigma_n - \sigma|=0,\quad \lim_{n \to \infty} \sup_{t\in \real_+} \left|\lambda_n(t)-t\right| = 0 \quad \text{and} \quad \lim_{n \to \infty} \sup_{t \leqslant N}\dbl\left(\mu_n\circ \lambda_n(t),\mu(t)\right) = 0,
	\end{equation*}
	for every $N\geqslant 1$, where we set $\sigma_n = \sigma(\mu_n)$ and $\sigma = \sigma(\mu)$. 
	
	Let $\kappa_n = \nu \circledast (s,\mu_n)$ and $\kappa = \nu \circledast (s,\mu)$. Our goal is to show that $\kappa_n$ converges to $\kappa$ with respect to the Skorokhod topology. To this end, let $\epsilon_n >0$ be a sequence converging to $0$ such that $\epsilon_n >\lambda_n(\sigma)-\sigma_n$ and let $\tilde{\lambda}_n \in \Lambda$ be a time change such that $\tilde{\lambda}_n(t) = t$ if $t\leq s$, $\tilde{\lambda}_n(s+t) = s + \lambda_n(t)$ if $t\leqslant \sigma$, $\tilde{\lambda}_n(s+\sigma+t) = s+\sigma_n +t$ if $t\geqslant \epsilon_n$ and $\tilde{\lambda}_n([s+\sigma,s+\sigma+\epsilon_n])= [s+\lambda_n(\sigma), s+\sigma_n+\epsilon_n]$. Notice that if $t\in [s+\sigma,s+\sigma+\epsilon_n]$, we have:
	\begin{equation*}
		\left|\tilde{\lambda}_n(t)-t\right| \leqslant \left|\lambda_n(\sigma) - \sigma-\epsilon_n\right|+\left|\sigma_n+\epsilon_n - \sigma\right|\leqslant |\lambda_n(\sigma) - \sigma| + |\sigma_n - \sigma|+2\epsilon_n.
	\end{equation*}
	It follows that
	\begin{align*}
		\sup_{t\in \real_+} \left|\tilde{\lambda}_n(t) - t\right| &\leqslant \sup_{s\leqslant t\leqslant s+\sigma }\left|\tilde{\lambda}_n(t)-t\right| + \sup_{s+\sigma \leqslant t\leqslant s+\sigma+\epsilon_n}\left|\tilde{\lambda}_n(t) -t\right| + \sup_{t\geqslant s+\sigma+\epsilon_n} \left|\tilde{\lambda}_n(t) -t\right|\\
		&\leqslant \sup_{t\leqslant \sigma}\left|\lambda_n(t) - t\right| +\left|\lambda_n(\sigma)-\sigma\right| + 2\left|\sigma_n - \sigma\right| + 2\epsilon_n, 	\end{align*}
	where the right-hand side goes to $0$ as $n\to \infty$. 
	
	In order to show that $\kappa_n$ converges to $\kappa$ in $\D$, it is enough to check that
	\begin{equation*}
		\lim_{n\to \infty} \sup_{t\leqslant N} \dbl\left(\kappa_n \circ \tilde{\lambda}_n(t), \kappa(t)\right) = 0, \quad \forall N \geqslant 1.
	\end{equation*}
	If $t \leqslant s$, we have $\kappa_n\circ\tilde{\lambda}_n(t) = \kappa(t) = \nu(t)$. If $t \leqslant \sigma$ and $\lambda_n(t) \leqslant \sigma_n$, we have:
	\begin{equation*}
		\kappa_n \circ \tilde{\lambda}_n(s+t) = \kappa_n(s+\lambda_n(t))= \left[\nu(s), \mu_n\circ\lambda_n(t)\right] \quad \text{and} \quad \kappa(s+t) = [\nu(s),\mu(t)].
	\end{equation*}
	It follows that
	\begin{equation*}
		\dbl\left(\kappa_n\circ\tilde{\lambda}_n(s+t),\kappa(s+t)\right) \leqslant \dbl(\mu_n\circ \lambda_n(t),\mu(t)).
	\end{equation*}
	On the other hand, if $t\leqslant \sigma$ and $\lambda_n(t)>\sigma_n$, we have:
	\begin{equation*}
		\kappa_n \circ \tilde{\lambda}_n(s+t) = \nu(s+\lambda_n(t)-\sigma_n) \quad \text{and} \quad \kappa(s+t) = [\nu(s),\mu(t)].
	\end{equation*}
	In that case, we get:
	\begin{align*}
		\dbl\left(\kappa_n\circ\tilde{\lambda}_n(s+t),\kappa(s+t)\right) &\leqslant \dbl\left(\nu(s+\lambda_n(t)-\sigma_n),\nu(s)\right)+\dbl\left(\nu(s),[\nu(s),\mu(t)]\right)\\
		&\leqslant \dbl\left(\nu(s+\lambda_n(t)-\sigma_n),\nu(s)\right)+\langle \mu(t),1\rangle.
	\end{align*}
	
	If $t \in [s+\sigma,s+\sigma +\epsilon_n]$, then $\kappa_n\circ \tilde{\lambda}_n(t)$ is of the form $\nu(u)$ with $u \in [s,s+\epsilon_n]$ or $[\nu(s), \mu_n(u)]$ with $u \in [\lambda_n(\sigma),\sigma_n]$. We deduce that
	\begin{multline*}
		\dbl\left(\kappa_n\circ \tilde{\lambda}_n(t), \kappa(t)\right)\\
		\begin{aligned}[t] &\leqslant \sup_{s\leqslant u \leqslant s+\epsilon_n} \dbl\left(\nu(u),\nu(t-\sigma)\right)+\sup_{\lambda_n(\sigma)\leqslant u \leqslant \sigma_n} \dbl\left([\nu(s),\mu_n(u)],\nu(t-\sigma)\right)\\
		&\leqslant 3\sup_{s\leqslant u \leqslant s+\epsilon_n} \dbl\left(\nu(u),\nu(s)\right) + \sup_{\lambda_n(\sigma)\leqslant u \leqslant \sigma_n} \langle \mu_n(u),1\rangle.
		\end{aligned}
	\end{multline*}
	Finally, if $t\geqslant \epsilon_n$, then we have $\kappa_n\circ \tilde{\lambda}_n(s+\sigma+t) = \kappa(s+\sigma+t) = \nu(t)$. We deduce that
	\begin{multline}\label{eq: convergence skorkhod}
		\sup_{t\leqslant N} \dbl\left(\kappa_n \circ \tilde{\lambda}_n(t), \kappa(t)\right) \\
		\begin{multlined}[t]\leqslant\sup_{t\leqslant N} \dbl\left(\mu_n \circ \lambda_n(t),\mu(t)\right) + \sup_{s\leqslant u \leqslant s+(\lambda_n(\sigma)-\sigma_n)_+} \dbl\left(\nu(u),\nu(s)\right) \\+3\sup_{s\leqslant u \leqslant s+\epsilon_n} \dbl\left(\nu(u),\nu(s)\right) + \sup_{u\leqslant \sigma, \, \lambda_n(u)>\sigma_n} \langle \mu(u),1\rangle + \sup_{\sigma \leqslant u\leqslant N} \langle \mu_n\circ \lambda_n(u),1\rangle.
		\end{multlined}
	\end{multline}
	Observe that
	\begin{equation*}
		\sup_{u\leqslant \sigma, \, \lambda_n(u)>\sigma_n} \langle \mu(u),1\rangle = \sup_{\lambda_n^{-1}(\sigma_n)<u\leqslant \sigma} \langle \mu(u),1\rangle \to 0,
	\end{equation*}
	since $\lambda_n^{-1}(\sigma_n) \to \sigma$ and since $\mu$ is left-continuous at $\sigma$ and $\mu(\sigma) = 0$. Furthermore, using that $\mu(u) = 0$ for $u\geqslant \sigma$, we have:
	\begin{equation*}
		\sup_{\sigma \leqslant u\leqslant N} \langle \mu_n\circ \lambda_n(u),1\rangle \leqslant \sup_{\sigma \leqslant u \leqslant N} \dbl(\mu_n\circ \lambda_n(u),\mu(u)) \to 0.
	\end{equation*}
	Since $\nu$ is right-continuous at $s$, we deduce that the right-hand side of \eqref{eq: convergence skorkhod} converges to $0$, which concludes the proof.
\end{proof}

\begin{lemma}\label{lemma: weak convergence of forest with no big nodes delta}
	For every $\delta \in \supp{\pi}$, the measure $\rfordinf(\cdot|\Delta< \delta+\epsilon)$ converges weakly to $\for{\delta}(\cdot|\Delta\leqslant \delta)$ as $\epsilon \to 0$ on the space $(\DD,\dK)$.
\end{lemma}
\begin{remark}
	Notice that if $\delta = \inf \supp{\pi}$ is positive, then the measure $\pi$ is necessarily finite and
	we have:
	\begin{equation*}
		\for{\delta}(\Delta\leqslant \delta) \geqslant \for{\delta}(\Delta=0)= \e^{-\delta \n[\Delta>0]} >0.
	\end{equation*}
	This implies that the conditional measure $\for{\delta}(\dd \rho|\Delta\leqslant \delta)$ is well defined.
\end{remark}
\begin{proof}
	It is enough to show that for every Lipschitz-continuous and bounded function $F\colon \DD\to \real$, the following convergence holds:
	\begin{equation*}
		\lim_{\epsilon \to 0} \rfordinf\left(F(\rho)\middle|\Delta< \delta+\epsilon\right) = \for{\delta}\left(F(\rho)\middle|\Delta\leqslant \delta\right).
	\end{equation*}
	Fix such a function $F$. From the definition of $\rfordinf$, we have:
	\begin{equation*}
		\rfordinf\left(F(\rho)\ind{ \Delta< \delta+\epsilon}\right) = \frac{1}{\pibar(\dinf)}\int_{(\dinf,\delta+\epsilon)}\pi(\dd r) \for{r}\left(F(\rho) \ind{\Delta< \delta+\epsilon}\right).
	\end{equation*}
	In conjunction with \eqref{eq: proba forest has degree delta + epsilon}, this gives:
	\begin{multline*}
		\rfordinf\left(F(\rho)\middle| \Delta< \delta+\epsilon\right) \\
		= \frac{1}{\int_{(\dinf,\delta+\epsilon)} \e^{-r \n[\Delta\geqslant\delta+\epsilon]}\,\pi(\dd r)}\int_{(\dinf,\delta+\epsilon)}\pi(\dd r) \for{r}\left(F(\rho)\ind{\Delta< \delta+\epsilon}\right).
	\end{multline*}
	
	Now it is not difficult to show that, as $\epsilon \to 0$, we have:
	\begin{equation*}
		\int_{(\dinf,\delta+\epsilon)} \e^{-r \n[\Delta\geqslant \delta+\epsilon]}\,\pi(\dd r) \sim \pi(\dinf,\delta+\epsilon) \e^{-\delta\n[\Delta>\delta]}.
	\end{equation*}
	Thus, as $\epsilon \to 0$, we have:
	\begin{equation*}
		\rfor\left(F(\rho)\middle| \Delta< \delta+\epsilon\right) \sim \frac{\e^{\delta\n[\Delta>\delta]}}{\pi(\dinf,\delta+\epsilon)}\int_{(\dinf,\delta+\epsilon)}\pi(\dd r) \for{r}\left(F(\rho)\ind{\Delta< \delta+\epsilon}\right).
	\end{equation*}
	Thanks to \eqref{eq: distribution degree forest}, we have $\for{\delta}(\Delta\leqslant \delta) = \e^{-\delta\n[\Delta>\delta]}$. Thus, in order to prove the result, it is enough to show that
	\begin{equation}\label{eq: lebesgue point}
		\lim_{\epsilon\to 0} \frac{1}{\pi(\dinf,\delta+\epsilon)}\int_{(\dinf,\delta+\epsilon)}\pi(\dd r) \for{r}\left(F(\rho)\ind{\Delta< \delta+\epsilon}\right) = \for{\delta}\left(F(\rho)\ind{\Delta\leqslant \delta}\right).
	\end{equation}
	
	Write:
	\begin{multline*}
		\frac{1}{\pi(\dinf,\delta+\epsilon)}\int_{(\dinf,\delta+\epsilon)}\pi(\dd r) \for{r}\left(F(\rho) \ind{\Delta< \delta+\epsilon}\right) - \for{\delta}\left(F(\rho)\ind{\Delta\leqslant \delta}\right) \\
		= \begin{multlined}[t]\frac{1}{\pi(\dinf,\delta+\epsilon)}\int_{(\dinf,\delta+\epsilon)}\pi(\dd r) \left[\for{r}\left(F(\rho) \ind{\Delta< \delta+\epsilon}\right) - \for{\delta}\left(F(\rho) \ind{\Delta< \delta+\epsilon}\right)\right]\\
			+ \for{\delta}\left(F(\rho)\ind{\delta<\Delta < \delta+\epsilon}\right).
		\end{multlined}
	\end{multline*}
	By dominated convergence, it is clear that the second term on the right-hand side converges to $0$. 
	
	For the first term, one can couple the measures $\for{r}$ and $\for{\delta}$ in the following way. Let $\rho$ be the exploration process with branching mechanism $\psi$ starting from $0$ and let $(L_t^0, \, t \geqslant 0)$ be its local time process at $0$. Then the process $\tilde{\rho}^{(r)}$ defined in \eqref{eq: representation forest r} has distribution $\for{r}$ while $\tilde{\rho}^{(\delta)}$ has distribution $\for{\delta}$. It follows that
	\begin{multline}\label{eq: term 0}
		\left|\for{r}\left(F(\rho) \ind{\Delta< \delta+\epsilon}\right) - \for{\delta}\left(F(\rho) \ind{\Delta< \delta+\epsilon}\right)\right| \\
		\begin{aligned}[b]
			&=\left|\ex{F(\tilde{\rho}^{(r)}) \ind{\sup_{L_t^0\leqslant r} \Delta(\rho_t)<\delta+\epsilon} - F(\tilde{\rho}^{(\delta)})\ind{\sup_{L_t^0\leqslant \delta} \Delta(\rho_t)<\delta+\epsilon}}\right|\\
			&\leqslant C \pr{\sup_{\delta< L_t^0\leqslant r} \Delta(\rho_t)\geqslant \delta+\epsilon} + C \ex{1\wedge \dK(\tilde{\rho}^{(r)}, \tilde{\rho}^{(\delta)})}.
		\end{aligned}
	\end{multline}
	Using the Poissonian decomposition from Lemma \ref{lemma: poisson decomposition exploration}, we have for $r\in (\delta,\delta+\epsilon)$:
	\begin{align}\label{eq: term 1}
		\pr{\sup_{\delta< L_t^0\leqslant r} \Delta(\rho_t)\geqslant \delta+\epsilon} &\leqslant \pr{\sup_{\delta< L_t^0< \delta+\epsilon} \Delta(\rho_t)\geqslant \delta} \nonumber\\
		&= \pr{\sup_{\delta<-I_{\alpha_i} < \delta+\epsilon} \Delta(\rho^{i})\geqslant \delta} \nonumber\\
		&= 1-\e^{-\epsilon \n[\Delta\geqslant \delta]}.
	\end{align}
	Next, by definition of $\dK$ we have that $\dK(\tilde{\rho}^{(r)}, \tilde{\rho}^{(\delta)}) = |\sigma(\tilde{\rho}^{(r)}) - \sigma(\tilde{\rho}^{(\delta)})|+ \dS(\tilde{\rho}^{(r)}, \tilde{\rho}^{(\delta)})$. We introduce the right-continuous inverse $S$ of the local time process at $0$ given by:
	\begin{equation*}
		S_r = \inf\{t >0\colon\, L_t^0>r\}, \quad \forall r >0.
	\end{equation*}
	It is well known that the process $S$ is a subordinator. Then the process $\tilde{\rho}^{(r)}$ has lifetime $S_r$. Furthermore, we have:
	\begin{equation*}
		\dS(\tilde{\rho}^{(r)}, \tilde{\rho}^{(\delta)}) \leqslant \sup_{t\geqslant 0} \dbl(\tilde{\rho}_t^{(r)}, \tilde{\rho}_t^{(\delta)}).
	\end{equation*}
	For $L_t^0 \leqslant \delta$, the processes $\tilde{\rho}^{(r)}$ and $\tilde{\rho}^{(\delta)}$ differ only by their masses at $0$ so that $\dbl(\tilde{\rho}_t^{(r)}, \tilde{\rho}_t^{(\delta)}) \leqslant r-\delta \leqslant \epsilon$. On the other hand, for $L_t^0 > \delta$, we have $\tilde{\rho}^{(\delta)} = 0$ so that 
	\begin{equation*}\dbl(\tilde{\rho}_t^{(r)}, \tilde{\rho}_t^{(\delta)}) = \langle \tilde{\rho}_t^{(r)} , 1\rangle = (r-L_t^0) + \langle \rho_t ,1\rangle \leqslant \epsilon + \langle \rho_t,1\rangle,
	\end{equation*} 
	where we recall from Section \ref{subsect: exploration process} that $\langle \rho_t,1\rangle = X_t - I_t$, where $X$ is the underlying Lévy process and $I$ is its running infimum. It follows that
	\begin{equation}\label{eq: term 2}
		\dK(\tilde{\rho}^{(r)}, \tilde{\rho}^{(\delta)}) \leqslant S_{\delta+\epsilon} - S_\delta + \epsilon + \sup_{\delta<L_t^0 <\delta+\epsilon} (X_t - I_t),
	\end{equation}
	where the right-hand side converges to $0$ a.s.~as $\epsilon \to 0$.
	
	Combining \cref{eq: term 0,eq: term 1,eq: term 2}, we deduce that
	\begin{equation*}
		\left|\for{r}\left(F(\rho)\ind{\Delta< \delta+\epsilon}\right) - \for{\delta}\left(F(\rho)\ind{\Delta< \delta+\epsilon}\right)\right|\leqslant C_1(\epsilon),
	\end{equation*}
	for every $r \in (\delta,\delta+\epsilon)$, where $C_1(\epsilon)$ does not depend on $r$ and goes to $0$ as $\epsilon \to 0$. Similarly, for every $r\in (\delta-\epsilon,\delta)$, we have:
	\begin{equation*}
		\left|\for{r}\left(F(\rho)\ind{\Delta< \delta+\epsilon}\right) - \for{\delta}\left(F(\rho)\ind{\Delta< \delta+\epsilon}\right)\right|\leqslant C_2(\epsilon).
	\end{equation*}
	Finally, we deduce that
	\begin{equation*}
		\frac{1}{\pi(\dinf,\delta+\epsilon)}\int_{(\dinf,\delta+\epsilon)}\pi(\dd r) \left|\for{r}\left(F(\rho)\ind{\Delta< \delta+\epsilon}\right) - \for{\delta}\left(F(\rho)\ind{\Delta< \delta+\epsilon}\right)\right|
		\leqslant C_1(\epsilon) + C_2(\epsilon).
	\end{equation*}
	Letting $\epsilon \to 0$ proves \eqref{eq: lebesgue point} and the proof is complete.
\end{proof}

We are now in a position to prove the main result of this section which gives a description of the Lévy tree conditioned on having maximal degree $\delta$. For every atom $\delta>0$ of $\pi$, we set:
\begin{equation}\label{eq: definition g}
	\g(\delta) = \pi(\delta)\for{\delta}(\Delta\leqslant \delta)= \pi(\delta)\e^{-\delta \n[\Delta>\delta]},
\end{equation}
where the last equality is due to \eqref{eq: distribution degree forest}. Under $\n$, denote by $M_\delta$ the random variable defined by:
\begin{equation*}
	M_\delta = \frac{\e^{\g(\delta)\sigma }-1}{\g(\delta)}\cdot
\end{equation*}
This should be interpreted as $M_\delta=\sigma$ if $\delta$ is not an atom of $\pi$.

For every atom $\delta>0$ of the Lévy measure $\pi$, we define a probability measure $\atom{\delta}$ on the space $\D$ as follows. 
Take $\tilde{\rho}$ with distribution $\n[M_\delta \ind{\Delta<\delta}]^{-1} \n[M_\delta \ind{\Delta<\delta}\, \dd \rho]$, and, conditionally on $\tilde{\rho}$, let $\left((s_i, \rho_i), \, i \in I\right)$ be the atoms of a Poisson point measure with intensity $\g(\delta) \allowbreak\mathbf{1}_{[0,\sigma]}(s)\,\dd s \allowbreak\for{\delta}(\dd \hat{\rho}|\Delta\leqslant \delta)$ conditioned on containing at least one point. Then $\atom{\delta}$ is defined as the distribution of the process $\tilde{\rho} \circledast_{i\in I} (s_i, \rho_i)$.
\begin{thm}\label{thm: disintegration wrt degree}
	There exists a regular conditional probability $\n[\cdot|\Delta=\delta]$ for $\delta >0$ such that $\pi[\delta,\infty) >0$, which is given by, for every $F\in \B(\D)$:
	\begin{multline}\label{eq: disintegration wrt degree unified}
		\n[F(\rho)|\Delta=\delta] = \frac{1}{\n[M_\delta \ind{\Delta< \delta}]}\sum_{k=0}^\infty \frac{\g(\delta)^{k}}{(k+1)!}\\
		\times \n\left[\int \prod_{i=1}^{k+1} \mathbf{1}_{[0,\sigma]}(s_i)\, \dd s_i \for{\delta}(\dd {\rho}_i |\Delta\leqslant \delta) F(\rho\circledast_{i=1}^{k+1}(s_i,\rho_i))\ind{\Delta<\delta}\right].
	\end{multline}
	In particular, if $\delta>0$ is not an atom of the Lévy measure $\pi$, we have:
	\begin{equation}
		\n[F(\rho)|\Delta=\delta] = \int_{\real_+\times\D}  \tilted{\delta}(\dd s, \dd \tilde{\rho}) \int_{\D} \for{\delta}(\dd \hat{\rho}|\Delta\leqslant \delta) F(\tilde{\rho}\circledast (s,\hat{\rho})).
	\end{equation}
	If $\delta>0$ is an atom of $\pi$, we have:
	\begin{equation}
		\n[F(\rho)|\Delta=\delta] = \int_{\D} \atom{\delta}(\dd \tilde{\rho})\,  F(\tilde{\rho}).
	\end{equation}
\end{thm}
\begin{remark}\label{rk: conditioning by Edelta}
	Let $\unique{\delta}$ be the event that the maximal degree is $\delta$ and there is a unique first-generation node with mass $\delta$. We have:
	\begin{equation}\label{eq: conditioning by Edelta}
		\n\left[F(\rho)\middle|\unique{\delta}\right] = \int_{\real_+\times\D}  \tilted{\delta}(\dd s, \dd \tilde{\rho}) \int_{\D} \for{\delta}(\dd \hat{\rho}|\Delta\leqslant \delta) F(\tilde{\rho}\circledast (s,\hat{\rho})).
	\end{equation}
	When $\delta$ is an atom of $\pi$, this can be proved by taking $k=0$ in \eqref{eq: disintegration wrt degree unified}. Indeed, we have:
	\begin{equation*}
		\n\left[F(\rho)\mathbf{1}_{\unique{\delta}}\middle|\Delta=\delta\right] = \frac{1}{\n[M_\delta \ind{\Delta<\delta}]} \n\left[\int \mathbf{1}_{[0,\sigma]}(s)\,\dd s \for{\delta}(\dd \hat{\rho}|\Delta\leqslant \delta) F(\rho\circledast (s,\hat{\rho}))\right],
	\end{equation*}
	and the result follows by conditionning. When $\delta$ is not an atom of $\pi$, this follows from Theorem \ref{thm: disintegration wrt degree} together with the fact that, conditionally on $\Delta=\delta$, there is a unique node with mass $\delta$ (see Corollary \ref{cor: unique node with mass delta} below). In other words, conditioning the exploration process by $E_\delta$ when $\delta$ is an atom of $\pi$ yields the same distribution as conditioning by $\Delta=\delta$ when $\delta$ is not an atom of $\pi$.
\end{remark}
\begin{proof}
	Assume that $\delta \in \supp{\pi}$ is an atom of $\pi$. Then the event $\{\Delta=\delta\}$ has positive $\n$-measure (see Corollary \ref{lemma: pi diffuse}) and it follows from Theorem \ref{thm: degree decomposition} that $\rho$ conditioned on $\Delta=\delta$ has distribution $\atom{\delta}$. 
	
	Assume then that $\delta \in \supp{\pi}$ is not an atom of $\pi$	and let $F \colon \D \to \real$ be continuous and bounded. Applying Lemma \ref{lemma: equivalent conditionings} and using the fact $E_{\delta,\epsilon}  \subset \{\delta-\epsilon <\Delta<\delta+\epsilon\}$, we have as $\epsilon \to 0$:
	\begin{equation*}
		\n[F(\rho)|\delta-\epsilon<\Delta<\delta+\epsilon]\sim \n[F(\rho)|E_{\delta,\epsilon}].
	\end{equation*}
	But, thanks to Lemma \ref{lemma: decomposition degree delta}, we have:
	\begin{align}
		\n[F(\rho)|E_{\delta,\epsilon}] &= \n[F({\rho}^{\delta-\epsilon,-} \circledast (T_{\delta-\epsilon}, {\rho}^{\delta-\epsilon,+})|E_{\delta,\epsilon}]\nonumber\\
		&= \int_{\real_+\times \D} \tilted{(\delta--\epsilon)+}(\dd s,\dd \tilde{\rho}) \int_\D \rfordinf(\dd \hat{\rho}|\Delta<\delta+\epsilon) F(\tilde{\rho}\circledast (s,\hat{\rho})). \label{eq: conditioning by delta proof}
	\end{align}
	
	Recall from Lemma \ref{lemma: continuity} that for every fixed $(\nu,s) \in \D\times (0,\infty)$, the mapping $\mu \mapsto \nu\circledast (s,\mu)$ is continuous from $\DD$ to $\D$. Together with Lemma \ref{lemma: weak convergence of bottom part delta} and Lemma \ref{lemma: weak convergence of forest with no big nodes delta}, this gives:
	\begin{equation*}
		\lim_{\epsilon\to 0} \n\left[F(\rho)|\delta-\epsilon<\Delta< \delta+\epsilon\right] \\
		= \int_{\real_+\times \D} \tilted{\delta}(\dd s,\dd \tilde{\rho}) \int_{\D} \for{\delta}(\dd \hat{\rho}|\Delta\leqslant \delta) F(\tilde{\rho}\circledast(s,\hat{\rho})). 
	\end{equation*}
	A standard result on measure differentiation, see e.g. \cite[Theorem 1.30]{evans2015measure}, yields the desired result.
\end{proof}

\begin{corollary}\label{cor: unique node with mass delta}
	Assume that $\delta>0$ is not an atom of the Lévy measure $\pi$. Then, under $\n$ and conditionally on $\Delta=\delta$, there is a unique node with mass $\delta$.
\end{corollary}
\begin{proof}
	Notice that $\tilted{\delta}$-a.s.~$\Delta(\rho) <\delta$ by definition. Thus, thanks to Theorem \ref{thm: disintegration wrt degree}, it is enough to show that $\for{\delta}(\cdot|\Delta\leqslant \delta)$-a.s.~there is a unique node with mass $\delta$. We shall use the Poissonian decomposition from Lemma \ref{lemma: poisson decomposition exploration}. Let $\sum_{i\in I} \delta_{(\ell_i,\rho^{i})}$ be a point measure with distribution $\for{\delta}$, that is a Poisson point measure with intensity $\mathbf{1}_{[0,\delta]}(\ell)\, \dd \ell \n[\dd \rho]$. Then it suffices to check that, conditionally on $\sup_{i\in I} \Delta(\rho^{i})\leqslant \delta$, it holds that $\sup_{i\in I} \Delta(\rho^{i})<\delta$. 
	
	Since $\n[\Delta>\delta/2]<\infty$, only finitely many $\rho^{i}$ are such that $\Delta(\rho^{i})>\delta/2$. We deduce that
	\begin{equation*}
		\pr{\sup_{i\in I} \Delta(\rho^{i})<\delta}=\pr{\sup_{i\in I} \Delta(\rho^{i})\leqslant\delta, \, \Delta(\rho^{i})\neq\delta \text{ for all } i \in I}.
	\end{equation*}
	Thanks to Corollary \ref{lemma: pi diffuse}, we have $\n[\Delta=\delta] = 0$, which implies that $\Delta(\rho^{i})\neq \delta$ for all $i \in I$ almost surely. Therefore we get:
	\begin{equation*}
		\pr{\sup_{i\in I} \Delta(\rho^{i})<\delta}=\pr{\sup_{i\in I} \Delta(\rho^{i})\leqslant\delta}.
	\end{equation*}
	This proves the result.
\end{proof}

As an application of Theorem \ref{thm: disintegration wrt degree}, we can compute the joint distribution of the degree $\Delta$ of the exploration process when the Lévy measure $\pi$ is diffuse and the height $H_\Delta$ of the (unique) node with mass $\Delta$. We start by determining the distribution of $H(\rho_s)$ under $\tilted{\delta}(\dd s,\dd \rho)$. Recall from \eqref{eq: definition w} the definition of $\w$.
\begin{lemma}\label{lemma: height under tilted}
	Under $\tilted{\delta}(\dd s,\dd \rho)$ (resp.~$\tilted{\delta+}(\dd s,\dd \rho)$), the random variable $H(\rho_s)$ is exponentially distributed with mean $\w(\delta)$ (resp.~$\wplus(\delta)$).
\end{lemma}
\begin{proof}
	We only prove the result under $\tilted{\delta}$, the other being similar. By definition, we have:
	\begin{align*}
		\int_{\real_+\times \D} \ind{H(\rho_s) > h} \, \tilted{\delta}(\dd s, \dd \rho) &= \frac{1}{\w(\delta)}\n\left[\int_0^\sigma \ind{H_s>h}\, \dd s \,\ind{\Delta<\delta}\right]\\
		&= \frac{1}{\w(\delta)} \operatorname{\mathbf{N}}^{\psi_{\delta-}} \left[\e^{-\pi[\delta,\infty)\sigma} \int_0^\sigma \ind{H_s >h}\right],
	\end{align*}
	where we used Corollary \ref{cor: identity in distribution before big jump} for the last equality. 
	
	Thanks to Bismut's decomposition, see e.g. \cite[Theorem 2.1]{abraham2013forest}, we have for every $\lambda >0$:
	\begin{multline}\label{eq: bismut}
		\operatorname{\mathbf{N}}^{\psi_{\delta-}}\left[\e^{-\lambda\sigma}\int_0^\sigma \ind{H_s>h}\, \dd s\right] \\
		\begin{aligned}[b]
			&=\int_h^\infty \dd t \expp{-t\left[\psi_{\delta-}'(0) +2\beta\operatorname{\mathbf{N}}^{\psi_{\delta-}}[1-\e^{-\lambda\sigma}]+\int_{(0,\delta)} r\,\pi(\dd r) \operatorname{\mathbb{P}}_r^{\psi_{\delta-}}(1-\e^{-\lambda\sigma}) \right]}\\
			&= \int_h^\infty \dd t \expp{-t\left[\psi_{\delta-}'(0) + 2\beta \psi_{\delta-}^{-1}(\lambda)+\int_{(0,\delta)} r (
				1-\e^{-r\psi_{\delta-}^{-1}(\lambda)})\, \pi(\dd r)\right]}\\
			&= \int_h^\infty \dd t \,\e^{-t\psi_{\delta-}'\circ \psi_{\delta-}^{-1}(\lambda)}\\
			&= \frac{1}{\psi_{\delta-}'\circ \psi_{\delta-}^{-1}(\lambda)} \e^{-h\psi_{\delta-}'\circ \psi_{\delta-}^{-1}(\lambda)}.
		\end{aligned}
	\end{multline}
	Applying this to $\lambda =\pi[\delta,\infty)$ and using \eqref{eq: w general}, it follows that
	\begin{equation*}
		\int_{\real_+\times \D} \ind{H(\rho_s) > h} \, \tilted{\delta}(\dd s, \dd \rho) = \e^{-h/\w(\delta)}.
	\end{equation*}
	This proves the result.
\end{proof}

Let 
\begin{equation}\label{eq: definition TDelta}
	T_\Delta = \inf\{t \geqslant 0\colon \, \Delta(\rho_t) = \Delta\} 
\end{equation}
be the first time that $\rho$ contains an atom with mass $\Delta$ and let $H_\Delta = H(\rho_{T_\Delta})$ be the value of the height process at that time. We shall determine the joint distribution of $(\Delta,H_\Delta)$ assuming that the Lévy measure $\pi$ is diffuse.

\begin{proposition}\label{prop: joint distribution degree and height}
	Assume that the Lévy measure $\pi$ is diffuse. Then, $\n$-a.e. there is a unique node with mass $\Delta$. Furthermore, for every $\delta, h>0$, we have:
	\begin{equation}
		\n[\Delta>\delta,H_\Delta>h] = \int_{(\delta,\infty)} \e^{-h/w(r)}\,\n[\Delta\in \dd r].
	\end{equation}
	In other words, under $\n$ and conditionally on $\Delta=\delta$, $H_\Delta$ is exponentially distributed with mean $\w(\delta)$.
\end{proposition}
\begin{question}
	If $\delta$ is an atom of $\pi$, what is the distribution of the height of the MRCA of the nodes with mass exactly $\delta$ under $\n$, conditionally on $\Delta=\delta$?
\end{question}
\begin{proof}
	The first part follows from Corollary \ref{cor: unique node with mass delta}. Then, using Theorem \ref{thm: disintegration wrt degree}, we have:
	\begin{equation*}
		\n[\Delta>\delta, H_\Delta>h] = \int_{(\delta,\infty)} \n[H_\Delta>h|\Delta=r] \, \n[\Delta\in \dd r].
	\end{equation*}
	Now under $\n[\cdot|\Delta=r]$, $H_\Delta$ is distributed as $H_s= H(\rho_s)$ under $\tilted{\mathnormal{r}}(\dd s,\dd \rho)$. Lemma \ref{lemma: height under tilted} allows to conclude.
\end{proof}

\section{Conditioning on $\Delta=\delta$ and $H_\Delta= h$}\label{sect: conditioning on degree and height}
In this section, we assume that the Lévy measure $\pi$ is diffuse. Recall then from Proposition \ref{prop: joint distribution degree and height} that there is a unique node with mass $\Delta$ and $H_\Delta$ is its height. The goal of this section is to make sense of the conditional measure $\n[\cdot|\Delta=\delta, H_\Delta= h]$. Let
\begin{equation}F_{\delta,\epsilon} = \{\delta-\epsilon < \Delta< \delta+\epsilon, \, Z_0^{\delta-\epsilon} = 1,\, h-\epsilon< H(\rho_{T_{\delta-\epsilon}}) <h+\epsilon\}
\end{equation}
be the event that the maximal degree is between $\delta-\epsilon$ and $\delta+\epsilon$, there is a unique first-generation node with mass larger than $\delta-\epsilon$ and its height is between $h-\epsilon$ and $h+\epsilon$. Recall from \eqref{eq: definition w} the definition of $\w$.
\begin{lemma}\label{lemma: equivalent conditionings h}
	Assume that the Lévy measure $\pi$ is diffuse. For every $\delta \in \supp{\pi}$ such that $\pibar(\delta)>0$ and $h>0$, we have as $\epsilon \to 0$:
	\begin{align}
		\n[\delta-\epsilon<\Delta<\delta+\epsilon, h-\epsilon < H_\Delta<h+\epsilon] &\sim \n[F_{\delta,\epsilon}] \nonumber \\
		&\sim 2\epsilon \rfordinf(\Delta<\delta+\epsilon)\pibar(\delta) \e^{-h/\w(\delta)}.
	\end{align}
\end{lemma} 
\begin{proof}
	By Proposition \ref{prop: joint distribution degree and height}, we have:
	\begin{equation*}
		\n[\delta-\epsilon<\Delta<\delta+\epsilon, h-\epsilon < H_\Delta<h+\epsilon]
		= \int_{(\delta-\epsilon,\delta+\epsilon)} \n[\Delta\in \dd r] \,w(r)^{-1} \int_{h-\epsilon}^{h+\epsilon} \e^{-t /\w(r)}\, \dd t.
	\end{equation*}
	A straightforward application of the dominated convergence theorem gives:
	\begin{equation*}
		\n[\delta-\epsilon<\Delta<\delta+\epsilon, h-\epsilon < H_\Delta<h+\epsilon]
		\sim 2\epsilon \int_{(\delta-\epsilon,\delta+\epsilon)} \n[\Delta\in \dd r] \,\w(r)^{-1}\e^{-h /\w(r)}.
	\end{equation*}
	
	Since $\pi$ is diffuse, observe that $\n[\sigma \ind{\Delta=r}] = 0$ for every $r>0$ thanks to Corollary \ref{lemma: pi diffuse}. This implies that $\w$ is continuous and we deduce that
	\begin{equation*}
		\n[\delta-\epsilon<\Delta<\delta+\epsilon, h-\epsilon < H_\Delta<h+\epsilon]
		\sim 2\epsilon \n[\delta-\epsilon<\Delta<\delta+\epsilon]w(\delta)^{-1}\e^{-h /\w(\delta)}.
	\end{equation*}
	But Lemma \ref{lemma: equivalent conditionings} gives:
	\begin{equation*}
		\n[\delta-\epsilon<\Delta<\delta+\epsilon] \sim \n[E_{\delta,\epsilon}] .
	\end{equation*}
	Moreover, thanks to \eqref{eq: E delta epsilon} and the continuity of $w$, we have:
	\begin{equation*}
		\n[E_{\delta,\epsilon}] = \pibar(\delta-\epsilon) w(\delta-\epsilon) \rfordinf(\Delta<\delta+\epsilon)\sim \pibar(\delta)w(\delta)\rfordinf(\Delta<\delta+\epsilon).
	\end{equation*}
	We deduce that
	\begin{equation*}
		\n[\delta-\epsilon<\Delta<\delta+\epsilon, h-\epsilon < H_\Delta<h+\epsilon] \sim 2\epsilon \rfordinf(\Delta<\delta+\epsilon)\pibar(\delta) \e^{-h/w(\delta)}.
	\end{equation*}
	
	On the other hand, thanks to Theorem \ref{thm: degree decomposition}, we have:
	\begin{equation}\label{eq: F delta epsilon}
		\n[F_{\delta,\epsilon}] = \ndinf\left[\pibar(\delta-\epsilon)\e^{-\pibar(\delta-\epsilon)\sigma}\int_0^\sigma \ind{h-\epsilon < H_s<h+\epsilon} \, \dd s\right] \rfordinf(\Delta< \delta+\epsilon).
	\end{equation}
	Using Bismut's decomposition as in \eqref{eq: bismut} , we get:
	\begin{equation}\label{eq: local time equivalent}
		\ndinf\left[\e^{-\pibar(\delta-\epsilon)\sigma}\int_0^\sigma \ind{h-\epsilon < H_s<h+\epsilon} \, \dd s\right] \\
		= \int_{h-\epsilon}^{h+\epsilon} \e^{-t /w(\delta-\epsilon)} \, \dd t \sim 2\epsilon \e^{-h/w(\delta)},
	\end{equation}
	where again we used the continuity of $w$. It follows that
	\begin{equation*}
		\n[F_{\delta,\epsilon}] \sim 2\epsilon \rfordinf(\Delta<\delta+\epsilon)\pibar(\delta) \e^{-h/w(\delta)}.
	\end{equation*}
\end{proof}

For every $\delta,h>0$, denote by $\tiltedh{\delta}$ the probability measure on the space $\real_+\times \D$ defined by:
\begin{equation}\label{eq: definition Ptilted h}
	\int_{\real_+\times \D} F\, \dd \!\tiltedh{\delta} =  \frac{1}{\n[L^h_\sigma\,\ind{\Delta< \delta}] }\n\left[\int_0^\sigma F(s,\rho)\, L^h(\dd s)\,\ind{\Delta< \delta}\right],
\end{equation}
for every $F\in \B(\real_+\times \D )$. Since we are assuming that the Lévy measure $\pi$ is diffuse, we may replace the event $\{\Delta<\delta\}$ by $\{\Delta\leqslant \delta\}$ thanks to Corollary \ref{lemma: pi diffuse}. Thus, using Corollary \ref{cor: identity in distribution before big jump}, \cite[Theorem 4.5]{duquesne2005probabilistic} and \eqref{eq: w plus general}, we have:
\begin{equation}\label{eq: computation local time exp}
	\n[L^h_\sigma\,\ind{\Delta< \delta}]= \ndelta\left[L^h_\sigma\,\e^{-\pibar(\delta)\sigma} \right] = \e^{-h \psi_\delta'(\n[\Delta>\delta])}= \e^{-h/w(\delta)}.
\end{equation}
In particular, the following identity relating the measures $\tilted{\delta}$ and $\tiltedh{\delta}$ holds:
\begin{equation*}
	\frac{1}{\w(\delta)}\int_0^\infty \dd h\, \e^{-h/w(\delta)}\tiltedh{\delta}(\dd s, \dd \rho) = \tilted{\delta}(\dd s,\dd \rho),
\end{equation*}
where we used that $\mathbf{1}_{[0,\sigma]}(s) \, \dd s = \int_0^{a} \dd a \,L^{a}(\dd s)$, see Section \ref{subsect: local times}. The next lemma gives an approximation of the measure $\tiltedh{\delta}$.

\begin{lemma}\label{lemma: local time convergence}
	Let $F\colon \real_+\times \D\to \real$ be measurable and bounded. We have for every $\delta,h >0$:
	\begin{equation}\label{eq: local time convergence F}
		\lim_{\epsilon \to 0} \frac{1}{2\epsilon}\n\left[\int_0^\sigma F(s,\rho)\ind{h-\epsilon<H_s<h+\epsilon}\, \dd s\,\ind{\Delta\leqslant \delta-\epsilon}\right] 
		= \n\left[\int_0^\sigma F(s,\rho)\, L^h(\dd s)\,\ind{\Delta< \delta}\right].
	\end{equation}
\end{lemma}
\begin{proof}
	Recall from Section \ref{subsect: local times} that the measure $L^{a}$ is supported on the set $\{s \in [0,\sigma]\colon \, H_s = a\}$. Thus, we have:
	\begin{equation}\label{eq: occupation time}
		\frac{1}{2\epsilon}\int_0^\sigma F(s,\rho)\ind{h-\epsilon<H_s<h+\epsilon}\, \dd s = \frac{1}{2\epsilon} \int_{h-\epsilon}^{h+\epsilon} \dd a \int_0^\sigma F(s,\rho) \, L^{a}(\dd s). 
	\end{equation}
	Furthermore, $h$ is a jump time for the local time process $a\mapsto L^{a}$ if and only if it is a jump time for the total mass process $a\mapsto L^{a}_\sigma$. But, under $\n$, the process $(L_\sigma^{a}, \, a \geqslant 0)$ is a $\psi$-CB process. In particular, it has no fixed jump times. As a result, $\n$-a.e.~the mapping $a\mapsto L^{a}$ is continuous at $h$. We deduce that the following convergence holds $\n$-a.e.:
	\begin{equation*}
		\lim_{\epsilon\to 0}\frac{1}{2\epsilon}\int_0^\sigma F(s,\rho)\ind{h-\epsilon<H_s<h+\epsilon}\, \dd s\,\ind{\Delta\leqslant \delta-\epsilon}= \int_0^\sigma F(s,\rho)\, L^{h}(\dd s)\,\ind{\Delta<\delta}.
	\end{equation*}
	
	Next, using \eqref{eq: occupation time}, we have:
	\begin{equation*}
		\frac{1}{2\epsilon}\left|\int_0^\sigma F(s,\rho)\ind{h-\epsilon<H_s<h+\epsilon}\, \dd s\right| \ind{\Delta\leqslant \delta-\epsilon} \leqslant \frac{\norm{F}_\infty}{2\epsilon} \int_{h-\epsilon}^{h+\epsilon} L_\sigma^{a} \, \dd a,
	\end{equation*}
	where the last term converges $\n$-a.e. to $\norm{F}_\infty L_\sigma^h$ thanks to the continuity of $a \mapsto L^{a}_\sigma$ at $h$. Furthermore, by \cite[Eq.~(12)]{duquesne2005probabilistic} we have the convergence:
	\begin{equation*}
		\lim_{\epsilon\to 0}\frac{1}{2\epsilon} \n\left[\int_{h-\epsilon}^{h+\epsilon} L_\sigma^{a} \, \dd a\right] = \n[L_\sigma^h].
	\end{equation*}
	Thus, the generalized dominated convergence theorem yields \eqref{eq: local time convergence F}.
\end{proof}

The main result of this section is the following description of the exploration process conditioned on having maximal degree $\delta$ at height $h$.
\begin{thm}\label{thm: disintegration wrt degree and height} Assume that the Lévy measure $\pi$ is diffuse. There exists a conditional probability measure $\n[\cdot|\Delta=\delta, H_\Delta=h]$ for $\delta \in \supp{\pi}$. Furthermore, for every $F\in \B(\D)$, we have:
	\begin{equation}
		\n[F(\rho)|\Delta=\delta,H_\Delta=h] 
		=  \int_{\real_+\times \D}  \tiltedh{\delta}(\dd s, \dd \tilde{\rho}) \int_{\D} \for{\delta}(\dd \hat{\rho}|\Delta\leqslant \delta) F(\tilde{\rho}\circledast (s,\hat{\rho})).
	\end{equation}
\end{thm}
Assuming the Grey condition, this can be interpreted as follows in terms of trees. Under $\n$, conditionally on $\Delta = \delta$, $H_\Delta$ is exponentially distributed with mean $\w(\delta)$. Moreover, conditionally on $\Delta=\delta$ and $H_\Delta= h$, the Lévy tree can be constructed as follows: start with $\widetilde{\rdtree}$ with distribution $\n[L_\sigma^h \,\ind{\Delta< \delta}]^{-1} \ndelta[L_\sigma^h \,\ind{\Delta< \delta}\, \dd \rdtree]$, choose a leaf uniformly at random in $\widetilde{\rdtree}$ at height $h$ (i.e. according to the probability measure $L^h(\dd x)/L_\sigma^h$) and on this leaf graft an independent Lévy forest with initial mass $\delta$ conditioned to have degree $\leqslant \delta$. Notice that this result generalizes Theorem \ref{thm: disintegration wrt degree} when the Lévy measure $\pi$ is diffuse. In particular, one can recover the latter simply by integrating with respect to $h$.
\begin{proof}
	Let $\delta \in \supp{\pi}$ and $h>0$. Thanks to Lemma \ref{lemma: equivalent conditionings h}, we have as $\epsilon \to 0$:
	\begin{equation}\label{eq: equivalent conditionings}
		\n\left[F(\rho)\middle|\delta-\epsilon<\Delta<\delta+\epsilon,h-\epsilon<H_\Delta<h+\epsilon\right] \\
		\sim \n\left[F(\rho)\middle|F_{\delta,\epsilon}\right].
	\end{equation}
	Recall from \eqref{eq: definition bottom} and \eqref{eq: definition upper} the definitions of ${\rho}^{\delta-\epsilon,-}$ and ${\rho}^{\delta-\epsilon,+}$. Using the Poissonian decomposition from Theorem \ref{thm: degree decomposition} and Corollary \ref{cor: identity in distribution before big jump}, we have:
	\begin{multline*}
		\n\left[F(\rho) \mathbf{1}_{F_{\delta,\epsilon}}\right] \\
		\begin{aligned}[t]
		&=\n\left[F\left({\rho}^{\delta-\epsilon,-}\circledast(T_{\delta-\epsilon},{\rho}^{\delta-\epsilon,+})\right)\mathbf{1}_{F_{\delta,\epsilon}}\right] \\
		&= \pibar(\delta-\epsilon) \n\left[\int \mathbf{1}_{[0,\sigma]}(s) \,\dd s\rfordinf(\ind{\Delta< \delta+\epsilon}\, \dd \hat{\rho}) F(\rho\circledast(s,\hat{\rho})) \ind{h-\epsilon<H_s<h+\epsilon,\,\Delta\leqslant \delta-\epsilon}\right].
		\end{aligned}
	\end{multline*}
	By conditioning, it follows from \eqref{eq: F delta epsilon} that
	\begin{multline}\label{eq: conditioning on F delta epsilon}
		\n\left[F(\rho)\middle|F_{\delta,\epsilon}\right]=\frac{1}{\n\left[\int_0^\sigma \ind{h-\epsilon<H_s<h+\epsilon}\, \dd s\,\ind{\Delta\leqslant \delta-\epsilon}\right]}\\
		\times\n\left[\int \mathbf{1}_{[0,\sigma]}(s)\, \dd s \rfordinf(\dd \hat{\rho}|\Delta< \delta+\epsilon)F(\rho\circledast(s,\hat{\rho})) \ind{h-\epsilon<H_s<h+\epsilon,\,\Delta\leqslant \delta-\epsilon}\right].
	\end{multline}
	
	Therefore, using Lemma \ref{lemma: local time convergence}, Lemma \ref{lemma: continuity} and Lemma \ref{lemma: weak convergence of forest with no big nodes delta}, we deduce that
	\begin{multline}
		\lim_{\epsilon \to 0}\n\left[F(\rho)\middle|\delta-\epsilon<\Delta<\delta+\epsilon,h-\epsilon<H_\Delta<h+\epsilon\right]  \\
		= \int_{\real_+\times \D} F\, \dd \! \tiltedh{\delta}\times \for{\delta}(G({\rho})|\Delta\leqslant \delta),
	\end{multline}
	and the result readily follows by using \cite[Theorem 1.30]{evans2015measure}.
\end{proof}

\section{Local limit of the Lévy tree conditioned on large maximal degree}\label{sect: local limit}
In this section, we shall investigate the behavior of the exploration process conditionally on $\Delta = \delta$ as $\delta \to \infty$. We start with the subcritical case. Then recall from \eqref{eq: expectation sigma} that $\n[\sigma] = \alpha^{-1} <\infty$. We define a probability measure $\tilted{\infty}$ on the space $ \real_+\times \D$ by setting:
\begin{equation}
	\int_{\real_+\times \D} F\, \dd \!\tilted{\infty}= \alpha \n\left[\int_0^\sigma F(s,\rho)\, \dd s\right],
\end{equation}
for every $F\in \B(\real_+\times \D)$.
\begin{lemma}\label{lemma: weak convergence of bottom part}
	Assume that $\psi$ is subcritical. The probability measure $\tilted{\delta}$ converges to $\tilted{\infty}$ in total variation distance on the space $\real_+\times \D$ as $\delta \to \infty$.
\end{lemma}
\begin{proof}
	Let $F\colon \real_+\times \D$ be measurable and bounded. We have:
	\begin{equation*}
		\left| \n\left[\int_0^{\sigma} F(s,\rho)\, \dd s\,\ind{\Delta< \delta}\right]-\n\left[\int_0^\sigma F(s,\rho)\, \dd s\right]\right| \leqslant \norm{F}_\infty \n[\sigma \ind{\Delta\geqslant \delta}].
	\end{equation*}
	Since $\psi$ is subcritical, we have $\n[\sigma]< \infty$ and the right-hand side converges to $0$ as $\delta \to \infty$. This proves the result.
\end{proof}

For every measure-valued process $\mu = (\mu_t, \, t \geqslant 0)\in \D$, we define the measure-valued process $R_0(\mu)$ obtained from $\mu$ by removing any atoms at $0$:
\begin{equation}\label{eq: definition R0}
	R_0(\mu)_t = \mu_t - \mu_t(0) \delta_0.
\end{equation}
Denote by $\explo$ the distribution of the exploration process $\rho$ with branching mechanism $\psi$ starting from $0$.
\begin{lemma}\label{lemma: weak convergence of forest with no big nodes}
	Assume that $\psi$ is subcritical and that $\pi$ is unbounded. Under $\for{\delta}(\cdot|\Delta\leqslant \delta)$, the process $R_0(\rho)$ converges in distribution to $\explo$ in the space $(\D,\dS)$ as $\delta \to \infty$.
\end{lemma}
\begin{proof}
	Recall from \eqref{eq: distribution degree forest} that $\for{\delta}(\Delta\leqslant \delta) = \e^{-\delta \n[\Delta>\delta]}$. Since $\psi$ is subcritical, by \cite[Proposition 3.8]{he2014maximal}, we have as $\delta \to \infty$:
	\begin{equation}\label{eq: degree tail equivalent}
		\n[\Delta>\delta]\sim \frac{\pibar(\delta)}{\alpha}\cdot
	\end{equation}
	But $\delta \pibar(\delta) \leqslant \int_{(\delta,\infty)} r\, \pi(\dd r)$ and the last term goes to $0$ as $\delta\to \infty$. It follows that $\lim_{\delta \to \infty} \delta \n[\Delta>\delta] = 0$ and
	\begin{equation}\label{eq: proba big forest}
		\lim_{\delta\to \infty }\for{\delta}(\Delta\leqslant \delta) = 1.
	\end{equation}
	Thus, it suffices to show that for every continuous and bounded function $F \colon \D \to \real$, the following convergence holds:
	\begin{equation}\label{eq: convergence without atom 0}
		\lim_{\delta\to \infty} \for{\delta}(F\circ R_0(\rho) ) = \explo(F(\rho)).
	\end{equation}
	
	Let $\rho$ be the exploration process with branching mechanism $\psi$ starting from $0$, that is $\rho$ has distribution $\explo$. Then, the process $\tilde{\rho}^{(\delta)}$ defined in \eqref{eq: representation forest r} has distribution $\for{\delta}$. Notice that we have $R_0(\tilde{\rho}^{(\delta)})_t = \rho_t \ind{L_t^0\leqslant r}$, which implies that
	\begin{equation*}
		\dS(R_0(\tilde{\rho}^{(\delta)}), \rho) \leqslant \sup_{t\geqslant 0} \dbl(R_0(\tilde{\rho}^{(\delta)})_t, \rho_t) = \sup_{L_t^0>\delta} \langle \rho_t,1\rangle.
	\end{equation*}
	Recall that $\langle \rho_t, 1\rangle = X_t - I_t$. Since the Lévy measure $\pi$ satisfies the integrability assumption $\int_{(0,\infty)} (r\wedge r^2) \, \pi(\dd r) <\infty$, the process $X$ does not drift to $\infty$; see e.g.~\cite[Chapter VII]{bertoin1996levy}. This implies that the following convergence holds a.s.:
	\begin{equation*}
		\lim_{\delta\to \infty} \sup_{L_t^0>\delta} (X_t-I_t) = 0.
	\end{equation*}
	Therefore, the process $R_0(\tilde{\rho}^{(\delta)})$ converges a.s.~to $\rho$ for the Skorokhod topology. This proves \eqref{eq: convergence without atom 0} and the proof is complete.
\end{proof}
\begin{remark}
	It should be clear from \eqref{eq: representation forest r} that the mass $\tilde{\rho}^{(\delta)}_0(0)$ of the atom at $0$ goes to $\infty$ as $\delta \to \infty$. This corresponds to the condensation phenomenon: a node with infinite mass appears at the limit. By introducing the operator $R_0$, we remove this mass which allows us to study the limiting behavior above the condensation node.
\end{remark}

Similarly to what was done in Section \ref{sect: conditioning on degree} (see \eqref{eq: definition bottom} and \eqref{eq: definition upper}), we split the path of the exploration process into two parts around the first node with mass $\Delta$: ${\rho^{\Delta,-}}$ is the pruned exploration process (that is the exploration process minus the first  node with mass $\Delta$) and $\rho^{\Delta,+}$ is the path of the exploration process above the first node with mass $\Delta$. Notice that $\rho^{\Delta,+}_0 $ is equal to $\Delta$ times the Dirac measure at $0$.
Let 
\begin{equation}\label{eq: definition Edelta}
	\unique{\delta} = \{\Delta= \delta, \, \Delta(\rho^{\Delta,-})< \delta\}
\end{equation}
be the event that the maximal degree is equal to $\delta$ and there is a unique first-generation node with mass $\delta$. Recall from \eqref{eq: definition g} the definition of $\g$.
\begin{lemma}\label{lemma: weak convergence atom}
	Assume that $\psi$ is subcritical and that the set of atoms of the Lévy measure $\pi$ is unbounded. The following holds as $\delta\to \infty$ along the set of atoms of $\pi$:
	\begin{equation}
		\n[\Delta=\delta] \sim \n[\unique{\delta}] \sim\frac{\g(\delta)}{\alpha}\cdot
	\end{equation}
\end{lemma}
\begin{proof}
	Under $\operatorname{\mathbf{N}^{\psi_{\delta--}}}$ and conditionally on $\rho$, let $\sum_{i=1}^N \delta_{(s_i,\rho_i)}$ be a Poisson point measure with intensity $\dd s \int_{[\delta,\infty)} \pi(\dd r)\for{r}(\dd \tilde{\rho})$. Thanks to Theorem \ref{thm: degree decomposition}, we have:
	\begin{equation*}
		\n[E_\delta] = \operatorname{\mathbf{N}^{\psi_{\delta--}}}\left[N=1, \Delta(\rho_1) \leqslant \delta\right].
	\end{equation*}
	But, under $\operatorname{\mathbf{N}^{\psi_{\delta--}}}$ and conditionally on $\rho$, $N$ has Poisson distribution with parameter $\pi[\delta,\infty) \sigma$, $\rho_1$ has distribution $\pi[\delta,\infty)^{-1} \int_{[\delta,\infty)} \pi(\dd r) \for{r}(\dd \tilde{\rho})$ and they are independent. It follows that
	\begin{align}\label{eq: E delta measure}
		\n[E_\delta] &= \operatorname{\mathbf{N}^{\psi_{\delta--}}}\left[\sigma \e^{-\pi[\delta,\infty)\sigma} \int_{[\delta,\infty)} \pi(\dd r) \for{r}(\Delta\leqslant \delta)\right]\nonumber\\
		&=\operatorname{\mathbf{N}^{\psi_{\delta--}}}\left[\sigma \e^{-\pi[\delta,\infty)\sigma} \pi(\delta) \e^{-\delta \n[\Delta>\delta]}\right]\nonumber\\
		&=\g(\delta)w(\delta),
	\end{align}
	where we used \eqref{eq: distribution degree forest} for the second equality and \eqref{eq: w general} for the last.
	
	Recall from \eqref{eq: definition w} the definition of $w$. Since $\psi$ is subcritical, it follows from \eqref{eq: expectation sigma} that $\lim_{\delta\to\infty} w(\delta) = \alpha^{-1}$. This proves that
	\begin{equation*}
		\n[\unique{\delta}] \sim \frac{\g(\delta)}{\alpha}\cdot
	\end{equation*}
	
	A similar computation yields:
	\begin{align}\label{eq: degree delta measure}
		\n[\Delta= \delta] &= \operatorname{\mathbf{N}^{\psi_{\delta--}}}\left[N\geqslant 1, \Delta(\rho_i) \leqslant \delta, \, \forall 1\leqslant i \leqslant N\right]\nonumber\\
		&= \operatorname{\mathbf{N}^{\psi_{\delta--}}}\left[ \e^{-\pi[\delta,\infty)\sigma}\left(\e^{\g(\delta)\sigma}-1\right)\right]\nonumber\\
		&= \n\left[\left(\e^{\g(\delta)\sigma}-1\right)\ind{\Delta<\delta}\right],
	\end{align}
	where we used Corollary \ref{cor: identity in distribution before big jump} for the last equality.
	
	Observe that since $\pi(1,\infty) <\infty$, $\pi(\delta)$ (and thus also $\g(\delta)$) converges to $0$
	as $\delta \to \infty$. It is clear that
	\begin{equation*}
		\lim_{\delta \to \infty}\frac{\e^{\g(\delta)\sigma}-1}{\g(\delta)} \ind{\Delta<\delta} = \sigma.
	\end{equation*}
	Furthermore, since $\psi$ is subcritical, there exists $\lambda_0>0$ such that $\n\left[\sigma \e^{\lambda_0 \sigma}\right]<\infty$. Thus, using Taylor's inequality, we have for $\delta >0$ large enough:
	\begin{equation*}
		\frac{\e^{\g(\delta)\sigma}-1}{\g(\delta)} \ind{\Delta<\delta} \leqslant \sigma \e^{\g(\delta)\sigma} \leqslant \sigma \e^{\lambda_0 \sigma},
	\end{equation*}
	Thanks to the dominated convergence theorem, we deduce that
	\begin{equation*}
		\lim_{\delta\to \infty}\frac{\n[\Delta=\delta]}{\g(\delta)} = \lim_{\delta\to \infty}\frac{1}{\g(\delta)}\n\left[\left(\e^{\g(\delta)\sigma}-1\right)\ind{\Delta<\delta}\right] = \n[\sigma] = \alpha^{-1}.
	\end{equation*}
	This finishes the proof.
\end{proof}

The first main result of this section concerns the limit of the subcritical Lévy tree conditioned on having a large maximal degree. Then there is a condensation phenomenon: the limit consists of a size-biased Lévy tree onto which one grafts -- at a uniformly chosen leaf -- an independent Lévy forest with infinite mass. In particular, the height of the condensation node is exponentially distributed. Recall from \eqref{eq: definition TDelta} that $T_\Delta$ is the first time that the exploration process contains an atom with mass $\Delta$. Recall also that $\rho^{\Delta,-}$ denotes the path of the exploration process after removing the first node with mass $\Delta$ while $\rho^{\Delta,+}$ denotes the path of the exploration process above that node. Finally, recall that $\explo$ is the distribution of the exploration process with branching mechanism $\psi$ starting from $0$.

\begin{thm}\label{thm: convergence subcritical delta}
	Assume that $\psi$ is subcritical and that $\pi$ is unbounded. Let $F\colon \real_+\times\D\to \real$ and $G\colon  \D\to \real$ be continuous and bounded. We have:
	\begin{equation}
		\lim_{\delta \to \infty} \n\left[F({T}_\Delta, \rho^{\Delta,-})G\circ R_0({\rho^{\Delta,+}})|\Delta=\delta\right] = \alpha \n\left[\int_0^\sigma F(s,\rho)\, \dd s\right]\explo(G(\rho)).
	\end{equation}
\end{thm}
\begin{proof}
	When $\delta \to \infty$ along the set of non-atoms $\{\delta>0\colon\, \pi(\delta)=0\}$, the convergence is a direct consequence of Theorem \ref{thm: disintegration wrt degree}, Lemma \ref{lemma: weak convergence of bottom part} and Lemma \ref{lemma: weak convergence of forest with no big nodes}. 
	
	Now assume that $\delta>0$ is an atom of $\pi$. Thanks to Lemma \ref{lemma: weak convergence atom} and since the inclusion $\unique{\delta}\subset \{\Delta=\delta\}$ holds, it is enough to show that the result holds when conditionning by $\unique{\delta}$. But, thanks to Remark \ref{rk: conditioning by Edelta}, we have:
	\begin{equation*}
		\n\left[F({T}_\Delta, \rho^{\Delta,-})G\circ R_0({\rho^{\Delta,+}})\middle|\unique{\delta}\right] = \int_{\real_+\times \D} F(s,\tilde{\rho})\, \tilted{\delta}(\dd s,\dd \tilde{\rho}) \for{\delta}(G\circ R_0(\rho)|\Delta\leqslant \delta).
	\end{equation*}
	The result readily follows from Lemma \ref{lemma: weak convergence of bottom part} and Lemma \ref{lemma: weak convergence of forest with no big nodes}.
\end{proof}

Next, we consider the critical case. Recall from \eqref{eq: definition w} the definition of $\w$. The next lemma is a key ingredient in the proof of the local convergence of the critical Lévy tree.
\begin{lemma}\label{lemma: convergence to immortal tree}
	Assume that $\psi$ is critical and that the Lévy measure $\pi$ is unbounded. For every $h>0$, we have
	\begin{equation}\label{eq: critical convergence bottom part to kesten}
		\lim_{\delta\to\infty} \frac{1}{w(\delta)} \n\left[\sigma F(r_h(\rho))\ind{\Delta< \delta}\right] = \n\left[L_\sigma^h \,F(r_h(\rho))\right].
	\end{equation}
\end{lemma}
\begin{proof}
	We shall use the decomposition of the exploration process above level $h$, see Section \ref{subsect: local times}. Let $(\rho^{i}, \, i\in I_h)$ be the excursions of the exploration process above level $h$. For every $i \in I_h$, let $\sigma^{i}$ (resp.~$\Delta^{i}$) be the lifetime (resp.~the maximal degree) of $\rho^{i}$. Similarly, denote by $\sigma_h$ (resp. $\Delta_h$) the lifetime (resp. the maximal degree) of $r_h(\rho)$. Thanks to Proposition \ref{prop: branching property}, we have:
	\begin{align*}
		\n\left[\sigma F(r_h(\rho)) \ind{\Delta< \delta}\right] &=  \n\left[\left(\sigma_h+\sum_{i\in I_h}\sigma^{i}\right) F(r_h(\rho)); \Delta_h< \delta, \,\Delta^{i}< \delta, \, \forall i \in I_h\right]\\
		&= \n\left[F(r_h(\rho))\ind{\Delta_h< \delta}\n\left[\sigma_h+\sum_{i\in I_h}\sigma^{i};\Delta^{i}< \delta, \, \forall i \in I_h\middle|\mathcal{E}_h\right]\right].
	\end{align*} 
	
	Thanks to the Mecke formula for Poisson random measures, see e.g. \cite[Chapter 4, Theorem 4.1]{last2017lectures}, we get:
	\begin{equation*}
		\n\left[\sigma_h+\sum_{i\in I_h}\sigma^{i};\Delta^{i}< \delta, \, \forall i \in I_h\middle|\mathcal{E}_h\right]= \left(\sigma_h +L_\sigma^h \n[\sigma \ind{\Delta< \delta}]\right)\e^{-L_\sigma^h \n[\Delta\geqslant\delta]}.
	\end{equation*}
	We deduce that
	\begin{equation}\label{eq: last step local convergence}
		\frac{1}{\w(\delta)} \n\left[\sigma F(r_h(\rho))\ind{\Delta< \delta}\right]
		= \n\left[F(r_h(\rho))\ind{\Delta_h < \delta}\left(L_\sigma^h + w(\delta)^{-1}\sigma_h\right)\e^{-L_\sigma^h \n[\Delta\geqslant\delta]}\right].
	\end{equation}
	Notice that $w(\delta) \to \infty$ as $\delta\to \infty$ since $\psi$ is critical. Furthermore, it is clear that $\sigma_h = \int_0^h L_\sigma^{a}\, \dd a$. Now letting $\delta\to \infty$ in \eqref{eq: computation local time exp} gives that $\n[L_\sigma^{a}] = 1$. It follows that $\n[\sigma_h] = h < \infty$. Thus, the dominated convergence theorem applies and we obtain the desired result by letting $\delta\to\infty$ in \eqref{eq: last step local convergence}.
\end{proof}

Recall from Theorem \ref{thm: disintegration wrt degree} that when the Lévy measure $\pi$ has an atom $\delta >0$, the exploration process conditioned on $\Delta=\delta$ has a random number of first-generation nodes with mass $\delta$. The next lemma gives a sufficient condition for there to be exactly one with high probability as $\delta\to \infty$. Recall from \eqref{eq: definition w} the definition of $w$. Recall also from \eqref{eq: definition Edelta} that $E_\delta$ denotes the event that the maximal degree is equal to $\delta$ and there is a unique first-generation node with mass $\delta$.
\begin{lemma}\label{lemma: critical atom equivalence of conditionings}
	Assume that $\psi$ is critical and that the Lévy measure $\pi$ is unbounded. Furthermore, assume that
	\begin{equation}\label{eq: pi small atoms}
		\lim_{\delta \to \infty} \frac{\pi(\delta)}{w(\delta)\pibar(\delta)\int_{[\delta,\infty)} r\, \pi(\dd r)} = 0.
	\end{equation} 
	We have as $\delta \to \infty$ along the set of atoms of $\pi$:
	\begin{equation}
		\n[\Delta=\delta]  \sim \n[E_\delta] = \g(\delta) w(\delta).
	\end{equation}
\end{lemma}
\begin{proof}
	Recall from \eqref{eq: E delta measure} and \eqref{eq: degree delta measure} that
	\begin{equation*}
		\n[E_\delta]=\g(\delta)w(\delta) \quad \text{and} \quad \n[\Delta= \delta]=\n\left[\left(\e^{\g(\delta)\sigma}-1\right)\ind{\Delta<\delta}\right].
	\end{equation*}
	Using Taylor's inequality, we deduce that
	\begin{equation}\label{eq: encadrement}
		1\leqslant \frac{\n[\Delta=\delta]}{\n[E_\delta]} \leqslant \frac{w_1(\delta)}{w(\delta)},
	\end{equation}
	where we set $w_1(\delta) = \n\left[\sigma\e^{\g(\delta)\sigma} \ind{\Delta<\delta}\right]$.
	

	Using \eqref{eq: w general} and the inequality $\e^{-x} \geqslant 1-x$ for every $x\geqslant 0$, we have:
	\begin{equation*}
		w(\delta) = \operatorname{\mathbf{N}^{\psi_{\delta--}}}\left[\sigma \e^{-\pi[\delta,\infty) \sigma}\right]
		\geqslant \operatorname{\mathbf{N}^{\psi_{\delta--}}}\left[\sigma (1-\pi(\delta)\sigma)\e^{-\pibar(\delta)\sigma}\right].
	\end{equation*}
	But thanks to Corollary \ref{cor: identity in distribution before big jump}, observe that
	\begin{equation*}
		w_1(\delta) = \operatorname{\mathbf{N}^{\psi_{\delta--}}}\left[\sigma \e^{(\g(\delta)-\pi[\delta,\infty))\sigma}\right]\leqslant \operatorname{\mathbf{N}^{\psi_{\delta--}}}\left[\sigma \e^{-\pibar(\delta)\sigma}\right],
	\end{equation*}
	where we used that $\g(\delta)\leqslant \pi(\delta)$ for the inequality. Furthermore, using that the function $x\mapsto x \e^{-x}$ is bounded on $\real_+$ by some constant $M>0$, we have:
	\begin{equation*}
		\operatorname{\mathbf{N}^{\psi_{\delta--}}}\left[ \sigma^2 \e^{-\pibar(\delta)\sigma}\right]\leqslant \frac{M}{\pibar(\delta)} \operatorname{\mathbf{N}^{\psi_{\delta--}}}[\sigma] = \frac{M}{\pibar(\delta)\int_{[\delta,\infty)} r\,\pi(\dd r)}\cdot
	\end{equation*}
	We deduce that
	\begin{equation*}
		w(\delta)\geqslant w_1(\delta)- \frac{M \pi(\delta)}{\pibar(\delta)\int_{[\delta,\infty)} r\,\pi(\dd r)}\cdot
	\end{equation*}

	It follows from \eqref{eq: encadrement} that
	\begin{equation*}
		1\leqslant \frac{\n[\Delta=\delta]}{\n[E_\delta]} \leqslant 1 + \frac{M \pi(\delta)}{w(\delta)\pibar(\delta) \int_{[\delta,\infty)}r\, \pi(\dd r)},
	\end{equation*}
	and the result readily follows by using \eqref{eq: pi small atoms}.
\end{proof}

		In the critical case, the Lévy tree conditioned on having a large maximal degree converges locally to the immortal Lévy tree. Intuitively, the condensation node goes to infinity and thus becomes invisible to local convergence.
		\begin{thm}\label{thm: convergence critical delta}
			Assume that $\psi$ is critical and that $\pi$ is unbounded. Furthermore, assume that \eqref{eq: pi small atoms} holds. Let $F \colon \D \to \real$ be continuous and bounded. For every $h >0$, we have:
			\begin{equation}
				\lim_{\delta\to\infty} \n\left[F(r_h(\rho))\middle|\Delta=\delta\right] = \n\left[L_\sigma^h\, F(r_h(\rho))\right].
			\end{equation}
		\end{thm}
		\begin{proof}
			First assume that $\delta >0$ is not an atom of $\pi$. Thanks to Theorem \ref{thm: disintegration wrt degree}, conditionally on $\Delta=\delta$, $\rho$ is distributed as $\tilde{\rho}\circledast (s,\hat{\rho})$, where $(s,\tilde{\rho})$ has distribution $\tilted{\delta}$, ${\hat{\rho}}$ has distribution $\for{\delta}(\cdot|\Delta\leqslant \delta)$ and they are independent. 
			
			Next, assume that $\delta>0$ is an atom of $\pi$. Recall that $E_\delta$ denotes the event that $\Delta=\delta$ and there is a unique first-generation node with mass $\delta$. Thanks to Lemma \ref{lemma: critical atom equivalence of conditionings}, since $E_\delta \subset \{\Delta=\delta\}$, the two conditionings are equivalent and it is enough to show that the result holds when conditioning on $E_\delta$. But Remark \ref{rk: conditioning by Edelta} gives that, conditionally on $E_\delta$, $\rho$ is again distributed as $\tilde{\rho}\circledast (s,\hat{\rho})$.
			
			Thus, in all cases, it is enough to show that
			\begin{equation*}
				\lim_{\delta\to \infty}\int_{\real_+\times \D} \tilted{\delta}(\dd s, \dd \tilde{\rho}) \int_{ \D} \for{\delta}(\dd \hat{\rho}|\Delta \leqslant \delta) F(r_h(\tilde{\rho}\circledast(s,\hat{\rho}))) \\
				= \n\left[L_\sigma^{h}\, F(r_h(\rho))\right].
			\end{equation*}
			Now, Lemma \ref{lemma: height under tilted} gives that the height $H(\tilde{\rho}_s)$ at which $\hat{\rho}$ is grafted is exponentially distributed with mean $\w(\delta)$. Since $\psi$ is critical, it holds that $\lim_{\delta \to \infty} w(\delta) = \infty$. Thus, we deduce that $H(\tilde{\rho}_s)>h$ with high probability as $\delta \to \infty$ under $\tilted{\delta}(\dd s, \dd \tilde{\rho})$, i.e.~we have:
			\begin{equation*}
				\lim_{\delta\to \infty}\int_{\real_+\times \D} \tilted{\delta}(\dd s, \dd \tilde{\rho}) \ind{H(\tilde{\rho}_s)\leqslant h} = 0.
			\end{equation*}
			Furthermore, on the event $\{H(\tilde{\rho_s})>h\}$, it holds that $r_h(\tilde{\rho}\circledast (s,\hat{\rho})) = r_h(\tilde{\rho})$, and the proof reduces to showing the following convergence:
			\begin{equation}\label{eq: proof critical convergence}
				\lim_{\delta\to \infty}\int_{\real_+\times \D} \tilted{\delta}(\dd s, \dd \tilde{\rho})  F(r_h(\tilde{\rho})) \\
				=  \n\left[L_\sigma^{h}\, F(r_h(\rho))\right].
			\end{equation}
			Recalling from \eqref{eq: definition Pdelta} the definition of $\tilted{\delta}$, Lemma \ref{lemma: convergence to immortal tree} yields \eqref{eq: proof critical convergence} and the proof is complete.
		\end{proof}
		We end this section with the following result dealing with the asymptotic behavior of the exploration process conditioned on having a large maximal degree at a fixed height $h$. Notice that this conditioning does not allow the condensation node to escape to infinity (even in the critical case as opposed to the conditioning of large maximal degree) and forces condensation to occur at a finite height. The limit consists of a Lévy tree biased by the population size at level $h$ onto which one grafts -- at a leaf chosen uniformly at random at height $h$ -- an independent Lévy forest with infinite mass.
		\begin{thm}\label{thm: convergence  delta h}
			Assume that $\psi$ is (sub)critical and that $\pi$ is unbounded and diffuse. Let $F\colon \real_+\times\D\to \real$ and $G\colon \D\to \real$ be continuous and bounded. We have:
			\begin{equation}
				\lim_{\delta \to \infty} \n\left[F({T}_\Delta, \rho^{\Delta,-})G(\rho^{\Delta,+})|\Delta=\delta, H_\Delta=h\right] =  \e^{\alpha h}\n\left[\int_0^\sigma F(s,\rho)\, L^h(\dd s)\right]\explo(G(\rho)).
			\end{equation}
		\end{thm}
		\begin{proof}
			Letting $\delta \to \infty$ in \eqref{eq: computation local time exp}, we have that $\lim_{\delta\to\infty} \n[L_\sigma^h \,\ind{\Delta< \delta}] = \e^{-\alpha h}$. Furthermore, the dominated convergence theorem yields:
			\begin{equation*}
				\lim_{\delta\to \infty} \n\left[ \int_0^\sigma F(s,\rho)\, L^h(\dd s)\,\ind{\Delta< \delta}\right]=\n\left[\int_0^\sigma F(s,\rho)\, L^h(\dd s)\right].
			\end{equation*}
			This proves that the following convergence holds:
			\begin{equation*}
				\lim_{\delta\to \infty} \int_{\real_+\times \D} F\, \dd\! \tiltedh{\delta} = \e^{\alpha h} \n\left[\int_0^\sigma F(s,\rho)\, L^h(\dd s)\right].
			\end{equation*}
			The result is then a direct consequence of Theorem \ref{thm: disintegration wrt degree and height} and Lemma \ref{lemma: weak convergence of forest with no big nodes}.
		\end{proof}
		
		\section{Other conditionings of large maximal degree}\label{sect: other conditionings}
		In this section, we look at other conditionings of large maximal degree. Recall from Section \ref{subsect: structure of big nodes} that $\init$ denotes the number of first-generation nodes with mass larger than $\delta$ while $\total$ denotes the total number of nodes with mass larger than $\delta$. Specifically, we study the conditionings $\Delta>\delta$ (which is equal to $\init \geqslant 1$ or $\total \geqslant 1$), $\init = 1$ and $\total = 1$. We shall see that, in the subcritical and critical cases, all three give rise to the same asymptotic behavior as conditioning by $\Delta= \delta$. 
		
		Notice that $\{\total =1\}$ (resp.~$\{\init=1\}$) is the event that $\rho$ contains exactly one node (resp.~one first-generation node) with mass larger than $\delta$. To begin, we compute the measure of these two events. In the subcritical case, they are equivalent in $\n$-measure to $\{\Delta >\delta\}$. However, this is no longer the case for critical branching mechanisms, see Proposition \ref{prop: three conditionings stable} for the (critical) stable case.
		\begin{proposition}\label{prop: exactly one node}
			We have:
			\begin{align}
				\n[\init =1] &= \frac{\pibar(\delta)}{\psi_\delta'(\n[\Delta>\delta])}, \label{eq: exactly one first-generation node}\\
				\n[\total =1] &= \frac{1}{\psi_\delta'(\n[\Delta>\delta])}\int_{(\delta,\infty)} \e^{-r\n[\Delta>\delta]} \, \pi(\dd r). \label{eq: exactly one node}
			\end{align}
			In particular, assuming that $\psi$ is subcritical and that $\pi$ is unbounded, we have as $\delta \to \infty$:
			\begin{equation}
				\n[\init=1]\sim\n[\total=1]\sim \n\left[\Delta>\delta\right] \sim \frac{\pibar(\delta)}{\alpha}\cdot
			\end{equation}
		\end{proposition}
		Since we have the inclusions $\{\total = 1\}\subset \{\init =1\} \subset \{\Delta>\delta\}$, Proposition \ref{prop: exactly one node} entails that, in the subcritical case, the three conditionings are equivalent as $\delta \to \infty$. In particular, conditionally on $\Delta>\delta$, there is exactly one node with mass larger than $\delta$ with probability tending to $1$ as $\delta \to \infty$.
		\begin{proof}
			Notice that $\{\total=1\}$ is the event that $\rho$ contains only one first-generation node with mass larger than $\delta$ and that this node has no descendants with mass larger than $\delta$. Thus, using the Poissonian decomposition of Theorem \ref{thm: degree decomposition}, we get $\n[\init =1] = \ndelta[\zeta=1]$ and
			\begin{equation}\label{eq: exactly one node proof}
				\n\left[\total=1\right] = \ndelta\left[\zeta = 1\right] \rfor(\Delta\leqslant \delta).
			\end{equation}
			
			Recall that under $\ndelta$ and conditionally on $\rho$, $\zeta$ has Poisson distribution with parameter $\pibar(\delta)\sigma$. Thus we have
			\begin{equation}\label{eq: exactly one first-generation node proof}
				\ndelta\left[\zeta = 1\right] = \ndelta\left[\pibar(\delta)\sigma \e^{-\pibar(\delta)\sigma}\right] =\frac{\pibar(\delta)}{\psi_\delta'\circ \psi_\delta^{-1}(\pibar(\delta))} = \frac{\pibar(\delta)}{\psi_\delta'(\n[\Delta>\delta])},
			\end{equation}
			where we used \eqref{eq: distribution degree strict} for the last equality. This proves \eqref{eq: exactly one  first-generation node}.
			
			Moreover, using the Poissonian decomposition of Proposition \ref{prop: degree decomposition forests} together with the fact that, under $\rfordelta$ and conditionally on $\rho$, $\xi$ has Poisson distribution with parameter $\pibar(\delta)\sigma$, we get:
			\begin{equation*}
				\rfor(\Delta\leqslant \delta)= \rfordelta(\xi = 0) 
				= \rfordelta(\e^{-\pibar(\delta)\sigma}) 
				= \frac{1}{\pibar(\delta)} \int_{(\delta,\infty)} \pi(\dd r) \fordelta{r}(\e^{-\pibar(\delta)\sigma}).
			\end{equation*}
			Thus, it follows from \eqref{eq: laplace transform of sigma for a forest} and \eqref{eq: distribution degree strict} that
			\begin{equation}\label{eq: no second-generation nodes}
				\rfor(\Delta\leqslant \delta) = \frac{1}{\pibar(\delta)} \int_{(\delta,\infty)} \e^{-r \psi_\delta^{-1}(\pibar(\delta))}\,\pi(\dd r)=\frac{1}{\pibar(\delta)} \int_{(\delta,\infty)} \e^{-r \n[\Delta>\delta]}\,\pi(\dd r).
			\end{equation}
			Finally, combining \eqref{eq: exactly one node proof}, \eqref{eq: exactly one first-generation node proof} and \eqref{eq: no second-generation nodes}, we deduce \eqref{eq: exactly one node}.

			Now assume that $\psi$ is subcritical and that $\pi$ is unbounded. Recall from \eqref{eq: degree tail equivalent} that
			\begin{equation*}
				\n[\Delta>\delta] \sim \frac{\pibar(\delta)}{\alpha}\cdot
			\end{equation*}
			On the other hand, differentiating \eqref{eq: definition psi delta}, we get:
			\begin{equation*}
				\psi_\delta'(\n[\Delta>\delta]) = \psi'(\n[\Delta>\delta]) + \int_{(\delta,\infty)}  r \e^{-r\n[\Delta>\delta] } \, \pi(\dd r).
			\end{equation*}
			Since $\int_{(1,\infty)} r\, \pi(\dd r) <\infty$, the dominated convergence theorem shows that the last integral converges to $0$ as $\delta \to \infty$. It follows that 
			\begin{equation}\label{eq: limit psi delta}
				\lim_{\delta \to \infty} 	\psi_\delta'(\n[\Delta>\delta])  = \psi'(0) = \alpha.
			\end{equation}
			In particular, we get that $\n[\init =1]\sim \alpha^{-1}\pibar(\delta)$.
			
			Furthermore, we have:
			\begin{align*}
				0\leqslant 1- \frac{1}{\pibar(\delta)}\int_{(\delta,\infty)} \e^{-r\n[\Delta>\delta]} \, \pi(\dd r)&= \frac{1}{\pibar(\delta)}\int_{(\delta,\infty)} \left(1-\e^{-r \n[\Delta>\delta]}\right)\, \pi(\dd r)\\
				&\leqslant \frac{\n[\Delta>\delta]}{\pibar(\delta)}\int_{(\delta,\infty)} r \, \pi(\dd r).
			\end{align*}
			The dominated convergence theorem gives $\lim_{\delta \to \infty}\int_{(\delta,\infty)} r\, \pi(\dd r) = 0$. Since $\lim_{\delta\to\infty}\n[\Delta>\delta]/\pibar(\delta) = \alpha^{-1}$, we deduce that
			\begin{equation*}
				\lim_{\delta \to \infty} \frac{1}{\pibar(\delta)}\int_{(\delta,\infty)} \e^{-r\n[\Delta>\delta]}\, \pi(\dd r) = 1.
			\end{equation*}
			Together with \eqref{eq: exactly one node} and \eqref{eq: limit psi delta}, this yields $\n[\total=1]\sim \alpha^{-1}\pibar(\delta)$. This concludes the proof.
		\end{proof}

		In the subcritical case, the three conditionings $\Delta >\delta$, $\init = 1$ and $\total = 1$ are equivalent as $\delta \to \infty$ and thus they yield the same asymptotic behavior: a condensation phenomenon occurs at the limit just like in Theorem \ref{thm: convergence subcritical delta} where we condition by $\Delta =\delta$. Recall from \eqref{eq: definition TDelta} that $T_\Delta$ is the first time that the exploration process contains an atom with mass $\Delta$. Recall also from Section \ref{sect: local limit} that $\rho^{\Delta,-}$ denotes the path of the exploration process after removing the first node with $\Delta$ while $\rho^{\Delta,+}$ denotes the path of the exploration process above that node.
		\begin{thm}\label{thm: convergence subcritical}
			Assume that $\psi$ is subcritical and that $\pi$ is unbounded. Let $F \colon \real_+\times\D \to \real$ and $G\colon  \D\to \real$ be continuous and bounded and let $A_\delta$ be equal to $\{\Delta>\delta\}$, $\{\init=1\}$ or $\{\total = 1\}$.
			We have:
			\begin{equation}
				\lim_{\delta\to \infty}\n\left[F({T}_\Delta, \rho^{\Delta,-})G\circ R_0(\rho^{\Delta,+})\middle|A_\delta\right] = \alpha \n\left[\int_0^\sigma F(s,\rho)\, \dd s\right] \explo(G (\rho)).
			\end{equation}
		\end{thm}
		\begin{proof}
			As the three events are equivalent it is enough to show the result for $A_\delta=\{\Delta >\delta\}$. Disintegrating with respect to $\Delta$, we have:
			\begin{multline*}
				\n\left[F({T}_\Delta, \rho^{\Delta,-})G\circ(\rho^{\Delta,+})\middle|\Delta>\delta\right]\\
				= \frac{1}{\n[\Delta>\delta]}\int_{(\delta,\infty)} \n[\Delta\in \dd r] \n\left[F({T}_\Delta, \rho^{\Delta,-})G\circ R_0(\rho^{\Delta,+})\middle|\Delta=r\right].
			\end{multline*}
			The conclusion follows from Theorem \ref{thm: convergence subcritical delta}.
		\end{proof}
		
		Recall from \eqref{eq: definition Tdelta} that $T_\delta$ is the first time $\rho$ contains a node with mass larger than $\delta$. Also recall from \eqref{eq: definition bottom} and \eqref{eq: definition upper} that $\rho^{\delta,-}$ denotes the path of the exploration process after removing the first node with mass larger than $\delta$ while $\rho^{\delta,+}$ denotes the path of the exploration process above that node. We shall determine the joint distribution of $(T_\delta, \rho^{\delta,-},\rho^{\delta,+})$ conditionally on $\init =1$ and $\total = 1$. Recall from \eqref{eq: definition w} the definition of $\wplus$.
		\begin{lemma}\label{lemma: decomposition conditionally on large degree}
			Assume that $\psi$ is (sub)critical and let $F \in \B(\real_+\times \D)$ and $G\in \B(\D)$. We have:
			\begin{align}
				\n\left[F(T_\delta,{\rho}^{\delta,-})G({\rho}^{\delta,+})\middle|\init =1\right] &= \frac{1}{\wplus(\delta)}\n\left[\int_0^\sigma F(s,\rho)\, \dd s\,\ind{\Delta\leqslant \delta}\right]\rfor(G(\rho)),\\
				\n\left[F(T_\delta,{\rho}^{\delta,-})G({\rho}^{\delta,+})\middle|\total =1\right] &= \frac{1}{\wplus(\delta)}\n\left[\int_0^\sigma F(s,\rho)\, \dd s\,\ind{\Delta\leqslant \delta}\right]\rfor(G(\rho)|\Delta\leqslant \delta).
			\end{align}
		\end{lemma}
		\begin{proof}
			We only prove the first identity, the second one being similar. Theorem \ref{thm: degree decomposition} gives:
			\begin{equation*}
				\n\left[F(T_\delta,{\rho}^{\delta,-})G({\rho}^{\delta,+})\mathbf{1}_{\{\init =1\}}\right] = \ndelta\left[F(U,\rho)G(\mathcal{F}^{\delta})\ind{\zeta=1}\right],
			\end{equation*}
			where under $\ndelta$ and conditionally on $\rho$, $\rho^{\delta}$ has distribution $\rfor$, $U$ is uniformly distributed on $[0,\sigma]$, $\zeta$ has Poisson distribution with parameter $\pibar(\delta)\sigma$ and they are independent. Therefore, conditioning on $\rho$ in the last term, we get:
			\begin{align*}
				\n\left[F(T_\delta,{\rho}^{\delta,-})G({\rho}^{\delta,+})\mathbf{1}_{\{\init =1\}}\right] &= \pibar(\delta)\ndelta\left[\e^{-\pibar(\delta)\sigma}\int_0^\sigma F(s,\rho)\, \dd s\right]\rfor(G(\rho)) \\
				&=\pibar(\delta)\n\left[\int_0^\sigma F(s,\rho)\, \dd s \,\ind{\Delta\leqslant\delta}\right]\rfor(G(\rho)),
			\end{align*}
			where we used Corollary \ref{cor: identity in distribution before big jump} for the last equality. This in conjunction with \eqref{eq: exactly one first-generation node} and \eqref{eq: w plus general} yields the desired result.
		\end{proof}

		In the critical case, the three conditionings $\Delta >\delta$, $\init = 1$ and $\total = 1$ are not equivalent but they still yield the same asymptotic behavior: local convergence to the immortal Lévy tree just like in Theorem \ref{thm: convergence critical delta} where we condition by $\Delta = \delta$.
		\begin{thm}\label{thm: convergence critical}
			Assume that $\psi$ is critical and that $\pi$ is unbounded. Let $F\colon \D\to \real$ be continuous and bounded and let $A_\delta$ be equal to $\{\Delta>\delta\}$, $\{\init=1\}$ or $\{\total = 1\}$.We have
			\begin{equation}
				\lim_{\delta \to \infty}\n\left[F(r_h(\rho))\middle|A_\delta\right] = \n\left[L_\sigma^h\, F(r_h(\rho))\right].
			\end{equation}
		\end{thm}
		
		\begin{proof}
			Since the conditioning by $\Delta >\delta$ was already treated in \cite{he2022local}, we only consider the other two. The proof uses similar arguments to that of Theorem \ref{thm: convergence critical delta} and we only give a sketch. By Lemma \ref{lemma: decomposition conditionally on large degree}, under $\n$ and conditionally on $\init = 1$, $\rho$ is distributed as $\tilde{\rho} \circledast (s,\hat{\rho})$, where $(s,\tilde{\rho})$ has distribution
			\begin{equation*}
				\ex{F(s,\tilde{\rho})}=\frac{1}{\wplus(\delta)} \n\left[\int_0^\sigma F(s,\rho)\,\dd s\,\ind{\Delta\leqslant \delta}\right],
			\end{equation*}
			$\hat{\rho}$ has distribution $\rfor$ and they are independent. But Lemma \ref{lemma: height under tilted} gives that the height $H(\tilde{\rho}_s)$ is exponentially distributed with mean $\wplus(\delta)$. Since $\psi$ is critical, this last quantity goes to $\infty$ as $\delta \to \infty$. In particular, it holds that $H(\tilde{\rho}_s)>h$ with high probability as $\delta \to \infty$. Furthermore, on the event $\{H(\tilde{\rho}_s)>h\}$, we have that 
			\begin{equation}\label{eq: invisible forest}
				r_h(\tilde{\rho}\circledast (s,\hat{\rho})) = r_h (\tilde{\rho}).
			\end{equation}
			As a consequence, in order to show the result, it is enough to prove that
			\begin{equation*}
				\lim_{\delta\to\infty} \frac{1}{\wplus(\delta)} \n\left[\sigma F(r_h(\rho))\ind{\Delta\leqslant \delta}\right] = \n\left[L_\sigma^h\, F(r_h(\rho))\right].
			\end{equation*}
			This last convergence holds by adapting the proof of Lemma \ref{lemma: convergence to immortal tree}. Finally, when conditioning on $\total = 1$, the only change is that $\hat{\rho}$ has distribution $\rfor(\cdot|\Delta\leqslant \delta)$ but this does not contribute to the limit because of \eqref{eq: invisible forest}. This completes the proof.
		\end{proof}
		
		\section{Stable case}\label{sect: stable}
		We consider the stable case $\psi(\lambda) = \lambda^\gamma$ with $\gamma \in (1,2)$. Notice that the branching mechanism is critical with $\alpha = \beta = 0$ and the Lévy measure $\pi$ is given by:
		\begin{equation*}
			\pi(\dd r)= a_\gamma r^{-1-\gamma}\, \dd r, \quad \text{where}\quad a_\gamma = \frac{\gamma(\gamma-1)}{\Gamma(2-\gamma)}\cdot
		\end{equation*}
		Then we have:
		\begin{equation}\label{eq: tail stable}
			\pibar(\delta) = \pi(\delta,\infty)=\frac{a_\gamma}{\gamma}\delta^{-\gamma}.
		\end{equation}
		Furthermore, the Grey condition \eqref{eq: grey condition} is satisfied and we can speak of the Lévy tree $\rdtree$, see Section \ref{subsect: levy tree}.
		
		We recall the scaling property of the stable tree. For every $\gamma \in (1,2)$, define the mapping $R_\gamma \colon \T\times (0,\infty)\to \T$ by:
		\begin{equation}
			R_\gamma((T,\root,d,\mu),a) = (T,\root,ad,a^{\gamma/(\gamma-1)}\mu), \quad \forall T \in \T.
		\end{equation}
		In words, the real tree $R_\gamma((T,\root,d,\mu),a)$ is obtained from $(T,\root,d,\mu)$ by multiplying the metric by $a$ and the measure by $a^{\gamma/(\gamma-1)}$. The choice of the exponent is justified by the following identity: for every $a>0$,
		\begin{equation}\label{eq: scaling property N}
			R_\gamma(\rdtree,a) \quad \text{under } \n \quad \lawd \quad \rdtree \quad \text{under } a^{1/(\gamma-1)} \n.
		\end{equation}
		Using this, one can define a regular conditional probability measure $\n[\cdot|\sigma=a]$ such that $\n[\cdot|\sigma=a]$-a.s. $\sigma = a$ and
		\begin{equation}\label{eq: disintegration wrt duration}
			\n[\dd \rdtree] = \frac{1}{\gamma \Gamma(1-1/\gamma)} \int_0^\infty \frac{\dd a}{a^{1+1/\gamma}} \n[\dd \rdtree|\sigma=a].
		\end{equation}
		Furthermore, under $\n[\cdot|\sigma=a]$, $\rdtree$ is distributed as $R_\gamma(\rdtree, a^{1-1/\gamma})$ under $\n[\cdot|\sigma=1]$. We shall now establish the scaling property of the degree.
		\begin{proposition}\label{prop: scaling property conditional on degree}
			Let $\psi(\lambda) = \lambda^\gamma$ with $\gamma \in (1,2)$. Then, under $\n[\cdot|\Delta=\delta]$, the stable tree $\rdtree$ is distributed as $R_\gamma(\rdtree,\delta^{\gamma-1})$ under $\n[\cdot|\Delta=1]$.
		\end{proposition}
		\begin{proof}
			Thanks to \cite[Theorem 4.7]{duquesne2005probabilistic}, we can write the degree of the stable tree $\rdtree$ as
			\begin{equation*}
				\Delta(\rdtree) = \sup_{x\in \rdtree} \left(\lim_{\epsilon \to 0} ((\gamma-1)\epsilon)^{-1/(\gamma-1)} n_\rdtree(x,\epsilon)\right),
			\end{equation*}
			where $n_\rdtree(x,\epsilon)$ is the number of subtrees originating from $x$ with height greater than $\epsilon$. In particular, it is straightforward to check that $\Delta(R_\gamma(\rdtree,a)) = a^{1/(\gamma-1)} \Delta(\rdtree)$. Then the conclusion readily follows from \eqref{eq: scaling property N}.
		\end{proof}
		
		Denote by $\Gamma(s,y)$ the upper incomplete gamma function:
		\begin{equation*}
			\Gamma(s,y) = \int_y^\infty t^{s-1}\e^{-t}\, \dd t, \quad \forall s \in \real, y >0.
		\end{equation*}
		Then the Laplace exponent $\psi_\delta$ is given by:
		\begin{equation}\label{eq: psi delta stable}
			\psi_\delta(\lambda) = \lambda^\gamma + a_\gamma\int_{\delta}^\infty (1-\e^{-\lambda r})\, \frac{\dd r}{r^{1+\gamma}} = \lambda^\gamma (1-a_\gamma \Gamma(-\gamma,\lambda\delta)) + \gamma^{-1}a_\gamma \delta^{-\gamma}.
		\end{equation}
		We will aslo need its derivative:
		\begin{equation*}
			\psi_\delta'(\lambda) = \lambda^{\gamma-1}(\gamma + a_\gamma \Gamma(1-\gamma,\lambda \delta)).
		\end{equation*}
		\begin{proposition}\label{prop: three conditionings stable}
			In the stable case $\psi(\lambda) = \lambda^\gamma$, we have:
			\begin{align}
				\n[\Delta >\delta] &= c_\gamma \delta^{-1}, \label{eq: degree distribution stable case}\\
				\n[\init = 1] &= \frac{c_\gamma }{\gamma}\e^{c_\gamma}\delta^{-1},\\
				\n[\total = 1] &=  \left(c_\gamma - \frac{\gamma c_\gamma^{\gamma+1}}{a_\gamma}\e^{c_\gamma}\right) \delta^{-1},
			\end{align}
			where $c_\gamma \in (0,\infty)$ is such that $\Gamma(-\gamma,c_\gamma) = a_\gamma^{-1}$.
			
		\end{proposition}
		\begin{proof}
			Thanks to \eqref{eq: distribution degree strict}, we have $\psi_\delta(\n[\Delta>\delta]) = \pibar(\delta)$.
			Together with \eqref{eq: psi delta stable}, this implies that $\delta\n[\Delta>\delta]$ is solution to $\Gamma(-\gamma,x) = a_\gamma^{-1}$. This proves \eqref{eq: degree distribution stable case}.
			
			To prove the remaining two identities, notice that
			\begin{equation*}
				\psi_\delta'(\n[\Delta>\delta]) = \psi_\delta'(c_\gamma \delta^{-1}) = c_\gamma^{\gamma-1}\delta^{1-\gamma} \left(\gamma + a_\gamma \Gamma(1-\gamma,c_\gamma)\right) = \frac{a_\gamma}{c_\gamma}\e^{-c_\gamma} \delta^{1-\gamma},
			\end{equation*}
			where we used the identity $\Gamma(s+1,x) = s\Gamma(s,x) + x^s \e^{-x}$ together with the definition of $c_\gamma$ for the last equality. The result readily follows from Proposition \ref{prop: exactly one node} by a straightforward computation.
		\end{proof}
		\begin{lemma}\label{lemma: inverse stable lambda}
			For every $\lambda \geqslant 0$, there exists a constant $c_\gamma(\lambda) \in (0,\infty)$ such that
			\begin{equation}\label{eq: inverse stable lambda}
				\psi_\delta^{-1}\left((1-\e^{-\lambda})\pibar(\delta)\right) = \frac{c_\gamma(\lambda)}{\delta}\cdot
			\end{equation}
			Moreover, $c_\gamma(\lambda)$ is the unique positive solution to $x^\gamma (a_\gamma \Gamma(-\gamma,x)-1) = \gamma^{-1} a_\gamma \e^{-\lambda}$.
		\end{lemma}
		\begin{proof}
			Fix $\lambda \geqslant 0$ and let
			\begin{equation*}
				u_{\gamma}^\lambda(x) = x^\gamma(1-a_\gamma \Gamma(-\gamma,x)) + \gamma^{-1}a_\gamma \e^{-\lambda}, \quad \forall x \geqslant 0.
			\end{equation*}
			Using the estimate $\Gamma(-\gamma,x) \sim \gamma^{-1}x^{\gamma}$ as $x\to 0$, elementary analysis gives that $u_\gamma^\lambda$ has a unique root which we denote by $c_\gamma(\lambda)$. Thanks to \eqref{eq: psi delta stable}, we get:
			\begin{equation*}
				\psi_\delta(\delta^{-1}c_\gamma(\lambda)) = (1-\e^{-\lambda}) \gamma^{-1} a_\gamma \delta^{-\gamma},
			\end{equation*}
			and the conclusion readily follows from \eqref{eq: tail stable}.
		\end{proof}

		In the stable case, we can make explicit the distribution of the Bienaymé-Galton-Watson forest $\rddtree_\delta$.
		\begin{proposition}\label{prop: galton watson stable}
			Under $\n$, conditionally on $\Delta >\delta$, the random forest $\rddtree_\delta$ consisting of nodes with mass larger than $\delta$ is a critical $(\init,\off)$-Bienaymé-Galton-Watson forest, where
			\begin{equation}
				\n\left[1-\e^{-\lambda \init}\middle|\Delta>\delta\right]= \frac{c_\gamma(\lambda)}{c_\gamma} \quad \text{and} \quad 
				\n\left[\e^{-\lambda \off}\middle|\Delta>\delta\right] = \e^{-\lambda} + \frac{\gamma}{a_\gamma}c_\gamma(\lambda).
			\end{equation}
			In particular, conditionally on $\Delta>\delta$, the distribution of $\rddtree_\delta$ is independent of $\delta$.
		\end{proposition}
		\begin{proof}
			Under $\ndelta$ and conditionally on $\rdtree$, let $\zeta$ be a Poisson random variable with parameter $\pibar(\delta)\sigma$. Notice that conditionally on $\Delta>\delta$, $\init$ is distributed as $\zeta$ under $\ndelta$ conditionally on $\zeta \geqslant 1$. Thus we have:
			\begin{equation}\label{eq: initial distribution conditioned proof}
				\n\left[1-\e^{-\lambda\init}\middle|\Delta>\delta\right] = \frac{\ndelta\left[(1-\e^{-\lambda \zeta}) \ind{\zeta \geqslant 1}\right]}{\ndelta\left[\zeta \geqslant 1\right]} =\frac{\ndelta\left[1-\e^{-\lambda \zeta}\right]}{\ndelta\left[\zeta \geqslant 1\right]}\cdot
			\end{equation}
			Since $\ndelta[\zeta \geqslant 1] = \n[\Delta>\delta]$ thanks to Theorem \ref{thm: degree decomposition}, it follows from \eqref{eq: initial distribution} that
			\begin{equation*}
				\n\left[1-\e^{-\lambda \init}\middle|\Delta>\delta\right]= \frac{\psi_\delta^{-1}\left((1-\e^{-\lambda})\pibar(\delta)\right)}{\n[\Delta>\delta]}\cdot
			\end{equation*}
			Combining \eqref{eq: degree distribution stable case} and \eqref{eq: inverse stable lambda}, we deduce that
			\begin{equation*}
				\n\left[1-\e^{-\lambda \init}\middle|\Delta>\delta\right]= \frac{c_\gamma(\lambda)}{c_\gamma}\cdot
			\end{equation*}
			
			Next, thanks to Theorem \ref{thm: degree decomposition}, it is easy to see than under $\n$, the random variables $\off$ and $\ind{\Delta>\delta}= \mathbf{1}_{\{\init \geqslant 1\}}$ are independent. It follows from \eqref{eq: offspring distribution} and Lemma \ref{lemma: inverse stable lambda} that
			\begin{equation*}
				\n\left[\e^{-\lambda \off}\middle|\Delta>\delta\right]  = \n\left[\e^{-\lambda \off}\right] = \e^{-\lambda} + \frac{\gamma}{a_\gamma} c_\gamma(\lambda).
			\end{equation*}
		\end{proof}
	\medskip	
	\bibliographystyle{amsplain}
	\bibliography{randomtrees}
	
	\newpage
	\cuthere \medskip
	\begin{center}
		\textbf{Index of notation}
	\end{center}
	
	\def\arraystretch{1.3}
	\begin{tabularx}{\textwidth}{p{2cm} X}
		\multicolumn{2}{l}{\textbf{Spaces}}\\
		$\Mf(E)$ & space of finite measures on $E$\\
		$\D$ & space of càdlàg functions from $\real_+$ to $\Mf(\real_+)$\\
		$\DD$ & space of càdlàg excursions from $\real_+$ to $\Mf(\real_+)$\\[1em]
		\multicolumn{2}{l}{\textbf{Random variables}}\\
		$\rho_t$ & exploration process\\
		$\eta_t$ & dual process\\
		$H_t$ & height process\\
		$\sigma$ & lifetime of the exploration process\\
		$L^{h}(\dd s)$ & local time at level $h$\\
		$\Delta$ & maximal degree of the exploration process\\
		$T_\delta$ & first time the exploration process contains a node with mass larger than $\delta$\\
		$\rho^{\delta,-}$ & path of the exploration process after removing the first node with mass larger than $\delta$\\
		$\rho^{\delta,+}$ & path of the exploration process above the first node with mass larger than $\delta$\\
		$\rddtree_\delta$ & discrete tree consisting of nodes with mass larger than $\delta$\\ 
		$\total$ & number of nodes with mass larger $\delta$\\
		$\init$ & number of first-generation nodes with mass larger than $\delta$\\
		$T_\Delta$ & first time the exploration process contains a node with mass $\Delta$\\
		$H_\Delta$ & height of the first node with mass $\Delta$\\
		$\rho^{\Delta,-}$ & path of the exploration process after removing the first node with mass $\Delta$\\
		$\rho^{\Delta,+}$ & path of the exploration process above the first node with mass $\Delta$\\[1em]
		\multicolumn{2}{l}{\textbf{Measures}}\\
		$\explo$ & distribution of the exploration process starting from $0$\\
		$\n$ & excursion measure of the exploration process\\
		$\mathbb{P}^{\psi,\ast}_{\nu}$ & distribution of the exploration process starting at $\nu$ and killed when it first reaches $0$\\
		$\for{r}$ & distribution of the exploration process with initial degree $r$\\
		$\rfor$ & distribution of the exploration process with random initial degree, \eqref{eq: definition random forest}\\
		$\tilted{\delta}$ & distribution of a marked exploration process with degree restriction, \eqref{eq: definition Pdelta}\\
		$\tiltedh{\delta}$ & distribution of a marked (at level $h$) exploration process with degree restriction, \eqref{eq: definition Ptilted h}\\[1em]
		\multicolumn{2}{l}{\textbf{Functions}}\\
		$\pibar(\delta)$ & tail of the Lévy measure $\pi$\\
		$\w(\delta)$ & $\n[\sigma \ind{\Delta< \delta}]$\\
		$\wplus(\delta)$ & $\n[\sigma \ind{\Delta\leqslant \delta}]$\\
		$\g(\delta)$ & $\pi(\delta)\e^{-\delta\n[\Delta>\delta]}$
	\end{tabularx}
\end{document}